\documentclass[12pt]{article}

\usepackage[utf8]{inputenc}
\usepackage[T1]{fontenc}

\usepackage{amsmath,
            mleftright,
            mathtools,
            amssymb,
            amsthm,
            nicefrac,
            xcolor,
            enumerate,
            comment,
            enumitem,
            mathrsfs,
            bbm,
            stmaryrd,
            geometry
            }
            
\mleftright  

\DeclareFontEncoding{LS1}{}{}
\DeclareFontSubstitution{LS1}{stix}{m}{n}
\DeclareMathAlphabet{\mathscr}{LS1}{stixscr}{m}{n} 

\usepackage[sort,nocompress]{cite}

\usepackage[colorinlistoftodos]{todonotes}
\tikzset{notestyleraw/.append style={align=justify}}
                        
\usepackage[english]{datetime2}
            
\usepackage[colorlinks=true]{hyperref}
            
\usepackage[sort,capitalize]{cleveref}

\crefformat{equation}{(#2#1#3)}
\crefname{enumi}{item}{items}
\crefname{subsection}{Subsection}{Subsections}

\newtheorem{theorem}{Theorem}[section]
\newtheorem{definition}[theorem]{Definition}
\newtheorem{proposition}[theorem]{Proposition}
\newtheorem{corollary}[theorem]{Corollary}
\newtheorem{lemma}[theorem]{Lemma}
\newtheorem{setting}[theorem]{Setting}

\numberwithin{equation}{section}

\newcommand{\N}{\ensuremath{\mathbb{N}}}
\newcommand{\Z}{\ensuremath{\mathbb{Z}}}

\newcommand{\R}{\ensuremath{\mathbb{R}}}
\newcommand{\C}{\ensuremath{\mathbb{C}}}
\newcommand{\E}{\ensuremath{\mathbb{E}}}
\renewcommand{\P}{\ensuremath{\mathbb{P}}}
\newcommand{\1}{\mathbbm{1}}
\newcommand{\fwpr}{W}
\newcommand{\fwprr}{W}
\newcommand{\smallU}{u}
\newcommand{\smallF}{f}
\newcommand{\funcF}{F}
\newcommand{\funcG}{g}
\newcommand{\boundFG}{\mathfrak{L}}
\newcommand{\LipConstF}{L}
\newcommand{\mlp}{U}
\newcommand{\cF}{\mathcal{F}}
\newcommand{\cR}{\mathcal{R}}
\newcommand{\cU}{\mathcal{U}}
\newcommand{\fB}{\mathfrak{B}}

\newcommand{\fD}{\mathfrak{D}}
\newcommand{\fR}{\mathscr{n}}
\newcommand{\fV}{\mathfrak{V}}
\newcommand{\fa}{\mathfrak{a}}
\newcommand{\fb}{\mathfrak{d}}
\newcommand{\fk}{\mathfrak{k}}
\newcommand{\fm}{\mathfrak{m}}

\newcommand{\fq}{\mathfrak{p}}
\newcommand{\ft}{\mathfrak{t}}
\newcommand{\fu}{\mathfrak{u}}
\newcommand{\fw}{\mathfrak{w}}
\newcommand{\sigmaAlgebra}{\sigma}
\newcommand{\gronA}{\beta}
\newcommand{\gronB}{\gamma}
\newcommand{\gronC}{\alpha}
\newcommand{\Ffn}{a}
\newcommand{\Gfn}{b}

\newcommand{\gronSymbb}{\Lambda}
\newcommand{\gronSymbbb}{\mathfrak{a}}
\newcommand{\Msymb}{\mathcal{M}}

\newcommand{\firstConstant}[1]{\mathfrak{K}_{#1}}
\newcommand{\secondConstant}[1]{\firstConstant{#1}\sqrt{#1 -1}}
\newcommand{\fnSymb}{\mathfrak{f}}
\newcommand{\littleM}{m}
\newcommand{\littleMM}{m}
\newcommand{\maxfn}{\varphi}
\newcommand{\power}{\beta}

\newcommand{\smallsum}{\textstyle\sum}
\newcommand{\SmallSum}{\textstyle\sum\limits}

\newcommand{\induct}{\dashrightarrow}
\newcommand{\with}{\curvearrowleft}
\newcommand{\lrSpace}{\ensuremath{\mkern-1.5mu}}
\newcommand{\Fsymb}[4]{{\mathbf{F}}_{#1,#2}^{#3,#4}}
\newcommand{\Gsymb}[1]{{\mathbf{G}}_{#1}}
\newcommand{\cost}[3]{{\mathfrak{C}}_{#3,#1,#2}}
\newcommand{\powerset}{\mathbbm{2}^\Omega}
\newcommand{\funcM}{\phi}
\newcommand{\funcMrep}[1]{\max\{ k \in \N \colon k \le \exp( \abs{\ln(#1)}^{1/2} ) \}}

\DeclarePairedDelimiter{\pr}{(}{)}
\DeclarePairedDelimiter{\br}{[}{]}

\DeclarePairedDelimiter{\abs}{\lvert}{\rvert}
\DeclarePairedDelimiter{\norm}{\lVert}{\rVert}
\DeclarePairedDelimiter{\floor}{\lfloor}{\rfloor}

\allowdisplaybreaks

\begin{document}

\title{Strong $L^p$-error analysis of nonlinear\\ 
Monte Carlo approximations for high-dimensional\\
semilinear partial differential equations}

\author{
Martin Hutzenthaler$^1$,
Arnulf Jentzen$^{2,3}$, \\
Benno Kuckuck$^4$,
and
Joshua Lee Padgett$^{5,6}$
\bigskip
\\
\small{$^1$ Faculty of Mathematics, University of Duisburg-Essen,}
\vspace{-0.1cm}\\
\small{Germany, e-mail: \texttt{martin.hutzenthaler@uni-due.de}}
\smallskip
\\
\small{$^2$ Applied Mathematics: Institute for Analysis and Numerics,}
\vspace{-0.1cm}\\
\small{University of M{\"u}nster, Germany, e-mail: \texttt{ajentzen@uni-muenster.de}}
\smallskip
\\
\small{$^3$ School of Data Science and Shenzhen Research Institute of Big Data,}
\vspace{-0.1cm}\\
\small{The Chinese University of Hong Kong, Shenzhen, China, e-mail: \texttt{ajentzen@cuhk.edu.cn}}
\smallskip
\\
\small{$^4$ Applied Mathematics: Institute for Analysis and Numerics,}
\vspace{-0.1cm}\\
\small{University of M{\"u}nster, Germany, e-mail: \texttt{bkuckuck@uni-muenster.de}}
\smallskip
\\
\small{$^5$ Department of Mathematical Sciences, University of Arkansas,}
\vspace{-0.1cm}\\
\small{Arkansas, USA, e-mail: \texttt{padgett@uark.edu}}
\smallskip
\\
\small{$^6$ Center for Astrophysics, Space Physics, and Engineering Research,}
\vspace{-0.1cm}\\
\small{Baylor University, Texas, USA, e-mail: \texttt{padgett@uark.edu}}
}

\date{\today}

\maketitle

\begin{abstract}
Full-history recursive multilevel Picard (MLP) approximation schemes have been shown to overcome the 
curse of dimensionality in the numerical approximation of high-dimensional semilinear partial differential equations (PDEs) with general time horizons
and Lipschitz continuous nonlinearities.
However, each of the error analyses for MLP approximation schemes in the existing literature studies the $L^2$-root-mean-square distance
between the exact solution of the PDE under consideration and the considered MLP approximation and none
of the error analyses in the existing literature provides an upper bound for the more general $L^p$-distance between the
exact solution of the PDE under consideration and the considered MLP approximation.
It is the key contribution of this article to extend the $L^2$-error analysis for MLP approximation schemes in the literature
to a more general $L^p$-error analysis with $p\in (0,\infty)$.
In particular, the main result of this article proves that the proposed MLP approximation scheme indeed overcomes the curse of dimensionality in the numerical approximation of high-dimensional semilinear PDEs with the approximation error measured in the $L^p$-sense with $p \in (0,\infty)$.
\end{abstract}

\pagebreak

\tableofcontents


\section{Introduction}
\label{sec:intro}

It is one of the most challenging topics in computational mathematics to design and analyze algorithms for the approximative solution of high-dimensional partial differential equations (PDEs) and there are several promising approaches to this topic in the scientific literature.

We refer, for instance, to \cite{DarbonOsher2016}
for approximation methods for certain high-dimensional
first-order Hamilton--Jacobi--Bellman PDEs.
We refer, for instance, 
to~\cite{%
chang2016branching,
le2017particle,
le2018monte,
LeCavilOudjaneRusso2019%
}
and the references mentioned therein for approximation methods for PDEs based on density estimations and particle systems.
We refer, for instance,
to~\cite{%
BenderDenk2007,%
bender2008time,%
GobetLabart2010,%
labart2013parallel,%
chassagneux2021learning%
} and the references mentioned therein for approximation methods based on Picard iterations and suitable projections on function spaces.
We refer, for instance,
to~\cite{skorokhod1964branching,watanabe1965branching,Henry-Labordere2012,
Henry-Labordere2014,chang2016branching,%
henry2019branching}
and the references mentioned therein
for approximation methods for semilinear parabolic PDEs based on branching diffusion
approximations.
We refer, for instance,
to~\cite{warin2018nesting,warin2018monte%
}
for approximation methods for semilinear parabolic PDEs 
based on 
standard Monte Carlo approximations for nested conditional expectations.
We refer, for instance, to~\cite{%
beck2020overview,EHanJentzen17,sirignano2017dgm,Han2018PNAS,%
nusken2021solving,MR4154658,
MR4253972,fujii2019asymptotic,de2021error,
hure2019some
}
and the references therein for deep learning-based approximation methods for high-dimensional PDEs.
We refer, for instance, to~\cite{EHutzenthaler2016,
EHutzenthaler2019,
HutzenthalerJentzenKruse2018}
for full-history recursive multilevel Picard approximation methods
for semilinear parabolic PDEs 
(in 
the following we abbreviate \emph{full-history recursive multilevel Picard} by MLP).

As of today, to the best of our knowledge, MLP approximation schemes are the only approximation schemes for 
high-dimensional PDEs in the scientific literature for which it has been proven that they overcome the 
curse of dimensionality in the numerical approximation of semilinear heat PDEs with general time horizons
and Lipschitz continuous nonlinearities; cf.~Hutzenthaler et al.\ \cite{HutzenthalerJentzenKruse2018}. 

The complexity analysis in \cite{HutzenthalerJentzenKruse2018} has been extended 
to 
more general MLP approximation
schemes and more general classes of nonlinear PDEs.
More specifically, we refer to
\cite{HutzenthalerPricing2019,HutzenthalerJentzenKruseNguyen2020} 
for complexity analyses for MLP approximation
schemes for parabolic semilinear PDEs involving more general second-order differential operators than just the Laplacian,
we refer to \cite{BeckHornungEtAl2019} for complexity analyses for MLP approximation schemes
for parabolic semilinear PDEs with possibly non-Lipschitz continuous nonlinearities such as
Allen-Cahn equations,
we refer to \cite{BeckGononJentzen2020} for complexity analyses for
MLP approximation schemes for elliptic semilinear PDEs with Lipschitz
continuous nonlinearities,
we refer to \cite{HutzenthalerJentzenKruse201912,HutzenthalerKruse2017} for complexity analyses for MLP approximation
schemes for parabolic semilinear PDEs with gradient-dependent nonlinearities,
and we refer to \cite{GilesJentzenWelti2019} for complexity analyses for a general class of MLP approximation
schemes for semilinear heat PDEs.
We also refer to 
\cite{EHutzenthaler2019,BeckerBraunwarthEtAl2020}
for numerical simulations for MLP approximation schemes.
Each of the error analyses for MLP approximation schemes in the above-mentioned articles studies the $L^2$-root-mean-square distance
between the exact solution of the PDE under consideration and the considered MLP approximation and none
of the error analyses in the above-mentioned articles provides an upper bound for the more general $L^p$-distance where $p \in (0,\infty)$ between the
exact solution of the PDE under consideration and the considered MLP approximation.

It is precisely the subject of this article to extend the $L^2$-error analyses for MLP approximation schemes
in \cite{HutzenthalerJentzenKruse2018} to a more general $L^p$-error analysis with $p\in (0,\infty)$ and, thereby, also introduce a slightly different variation of the previously studied MLP approximation schemes; see \cref{mlp_main} below.

It turns out that it is not straightforward to extend the $L^2$-error analysis for MLP approximation schemes from the literature to a more general $L^p$-error analysis with $p \in [2,\infty)$ (cf., e.g., Rio \cite[Theorem 2.1]{rio2009moment}).
A central difficulty 
is related to the issue that in our $L^p$-error analysis the growth of the number of samples used to approximate expectations via Monte Carlo averages must be more carefully chosen; see \cref{eq:1_6_above} and \cref{eq:1_7_above}
below 
for 
details.

To better illustrate the findings of this work, we present in the following result, \cref{th:1} below, a special
case of \cref{cor:final}, the main result of this paper. Below \cref{th:1} we add some explanatory
comments regarding the statement of \cref{th:1} and the mathematical objects appearing
in \cref{th:1}
and we also present a brief sketch of our proof of \cref{th:1}.

\begin{samepage}
\begin{theorem}\label{th:1}
Let $T,\kappa, \delta,p \in (0,\infty)$, $\Theta = \bigcup_{n\in\N}\! \Z^n$, let $\smallF \colon \R \to \R$ be Lipschitz continuous, let $\smallU_d \in C^{1,2}([0,T]\times \R^d,\R)$, $d\in\N$, satisfy for all $d\in\N$, $t \in [0,T]$, $x=(x_1,\allowbreak x_2,\allowbreak \dots, \allowbreak x_d)\in\R^d$ that $\abs{ \smallU_d(t,x)} \le \kappa d^\kappa \pr[]{ 1 + \sum_{k=1}^d \abs{ x_k } }^\kappa$ and 
\begin{equation}\label{eq:1}
\pr[]{\tfrac{\partial}{\partial t}\smallU_d}(t,x) = \pr[]{\Delta_x \smallU_d}(t,x) + \smallF(\smallU_d(t,x)),
\end{equation}
let $(\Omega, \cF ,\P)$ be a probability space, let $\fu^\theta \colon \Omega \to [0,1]$, $\theta\in\Theta$, be i.i.d.\ random variables, 
assume for all $r \in (0,1)$ that $\P(\fu^0 \le r ) = r$,
let $W^{d,\theta} \colon [0,T] \times \Omega\to \R^d$, $d\in\N$, $\theta\in\Theta$, be independent standard Brownian motions, assume that $(\fu^\theta)_{\theta\in\Theta}$ and $(W^{d,\theta})_{(d,\theta)\in\N\times\Theta}$ are independent, 
let $\funcM \colon \N \to \N$ and
$\mlp_{n,\littleM}^{d,\theta} \colon [0,T] \times \R^d \times \Omega \to \R$, $d,n,\littleM \in \Z$, $\theta \in \Theta$, satisfy for all $n \in \N_0$, $d,\littleM \in \N$, $\theta \in \Theta$, $t \in [0,T]$, $x \in \R^d$ that 
$\funcM(\littleM) = \funcMrep{\littleM}$
and
\begin{align}\label{mlp_main}
& \mlp_{n,\littleM}^{d,\theta}(t,x) 
= \SmallSum_{i=0}^{n-1} \tfrac{t}{(\funcM(\littleM))^{n-i}} \biggl[\SmallSum_{k=1}^{(\funcM(\littleM))^{n-i}} \Bigl[ \smallF \pr[\big]{ \mlp_{i,\littleM}^{d,(\theta,i,k)}(t\fu^{(\theta,i,k)}, x + \sqrt{2}\,W_{t-t\fu^{(\theta,i,k)}}^{d,(\theta,i,k)}) } \\
& - \1_\N(i) \, \smallF \pr[\big]{ \mlp_{i-1,\littleM}^{d,(\theta,-i,k)} (t\fu^{(\theta,i,k)}, x + \sqrt{2}\,W_{t-t\fu^{(\theta,i,k)}}^{d,(\theta,i,k)}) } \Bigr]\biggr] 
+ \tfrac{\1_\N(n)}{(\funcM(\littleM))^n} \biggl[ \SmallSum_{k=1}^{(\funcM(\littleM))^n} \smallU_d \pr[\big]{ 0,x + \sqrt{2}\,W_{t}^{d,(\theta,0,-k)} } \biggr], \nonumber
\end{align}
and for every $d,n,\littleM \in \N$ let $\cost{n}{\littleM}{d} \in \N$ be the number of function evaluations of $\smallF$ and $\smallU_d(0,\cdot)$ and the number of realizations of scalar random variables which are used to compute one realization of $\mlp_{n,\littleM}^{d,0}(T,0) \colon \Omega \to \R$ (see \cref{fc_def1} for a precise definition).
Then there exist $c\in\R$ and $\fR \colon \N \times (0,1] \to \N$ such that for all $d \in \N$, $\varepsilon \in (0,1]$ it holds that 
\begin{equation}
\textstyle
\pr[\big]{\E\br[\big]{\abs{\smallU_d(T,0) - \mlp_{\fR(d,\varepsilon),\fR(d,\varepsilon)}^{d,0}(T,0)}^p}}^{\!\nicefrac{1}{p}} \le \varepsilon
\qquad \text{and} \qquad
\cost{\fR(d,\varepsilon)}{\fR(d,\varepsilon)}{d} \le c d^c \varepsilon^{-(2+\delta)}
.
\end{equation}
\end{theorem}
\end{samepage}

\Cref{th:1} is an immediate consequence of \cref{cor:final} in \cref{sec:main_results} below.
\Cref{cor:final}, which is the main result of this article, in turn, follows from \cref{th:final} 
(see \cref{sec:main_results} below for details).
In the following we provide some explanatory comments concerning the mathematical objects appearing in \cref{th:1} above. 

In \cref{th:1} we intend to approximate the solutions of the PDEs in \cref{eq:1}.
The strictly positive real number $T \in (0,\infty)$ in \cref{th:1} describes the time horizon of the PDEs in \cref{eq:1}, 
the Lipschitz continuous function $f \colon \R \to \R$ specifies the nonlinearity of the PDEs in \cref{eq:1},
and the functions 
$\smallU_d \colon [0,T] \times \R^d \to \R$, $d\in\N$, are the solutions of the PDEs in \cref{eq:1}.

The strictly positive real number $\kappa \in (0,\infty)$ in \cref{th:1} is employed to formulate a regularity condition for the solutions $\smallU_d \colon [0,T] \times \R^d \to \R$, $d\in\N$, of the PDEs in \cref{eq:1} which we impose in \cref{th:1}.
More formally, in \cref{th:1} we assume that the solution functions $\smallU_d \colon [0,T] \times \R^d \to \R$, $d\in\N$, of the PDEs in \cref{eq:1} satisfy the regularity condition that for all $d \in \N$, $t\in[0,T]$, $x\in\R^d$ it holds that
\begin{equation}\label{th:1_cond}
\abs{ \smallU_d(t,x) } \le \kappa d^\kappa \pr[\big]{ 1 + \smallsum_{k=1}^d \abs{x_k} }^{\!\kappa} .
\end{equation}
This condition ensures that the solution functions $\smallU_d \colon [0,T] \times \R^d \to \R$, $d\in\N$, of the PDEs in \cref{eq:1} are at most polynomially growing both in the spatial variable $x\in\R^d$ and in the PDE dimension $d\in\N$.
Observe that the condition in \cref{th:1_cond} also ensures that solutions of the PDEs in \cref{eq:1} with the fixed initial value functions
$ \R^d \ni x \mapsto \smallU_d(0,x) \in \R $, $d \in \N$, are unique.

In \cref{mlp_main} we recursively specify the proposed MLP approximations which we employ in \cref{th:1} to approximate the solutions of the PDEs in \cref{eq:1}.
The proposed MLP approximation method is a random approximation 
algorithm which is defined on an artificial probability space. 
The probability space $( \Omega, \cF, \P )$ in \cref{th:1} is 
this artificial probability space on which we defined 
the proposed MLP approximations. 

To formulate the proposed MLP approximations, we need, 
roughly speaking, sufficiently many independent random variables as random input sources 
and to formulate these sufficiently many independent random variables, 
we need, roughly speaking, a sufficiently large index set over which 
the sufficiently many independent random variables are defined. 
The set $\Theta = \bigcup_{n\in\N}\! \Z^n$ in \cref{th:1} is precisely this sufficiently large index 
set which allows us to introduce sufficiently many independent random variables 
over this index set and the i.i.d.\ random variables $\fu^\theta \colon \Omega \to [0,1]$, $\theta\in\Theta$, and the independent standard Brownian motions $W^{d,\theta} \colon [0,T] \times \Omega\to \R^d$, $d\in\N$, $\theta\in\Theta$, are the sufficiently many independent random variables 
which we use to specify the MLP approximations in \cref{mlp_main}. 
Observe that the assumption that for all $r\in(0,1)$ it holds that $\P(\fu^0 \le r ) = r$ in \cref{th:1} 
ensures that the random variables $\fu^\theta \colon \Omega \to [0,1]$, $\theta\in\Theta$,
are on $[0,1]$ continuous uniformly distributed random variables. 

The MLP approximations specified in \cref{mlp_main} differ from 
previously introduced MLP approximations, roughly speaking, 
in the sense that a smaller number of Monte Carlo samples 
is employed. In \cref{th:1} this smaller number of Monte Carlo samples 
is formulated through the function $\funcM \colon \N \to \N$ 
which increases quite slowly to infinity. More formally, 
\cref{lem:fn_prop} in \cref{sec:4_4} below proves that for every 
$ \varepsilon \in (0,\infty) $
there exists $ c \in \R $ such that for all $ x \in [1,\infty) $
it holds that  
$ \lim_{ y \to \infty } \funcM(y) = \infty $
and
$ \funcM(x) \leq c x^{ \varepsilon } $. 
This slow increase to infinity is an important argument in 
our $L^p$-error analysis for the proposed MLP approximations 
(see \cref{eq:1_6_above}, \cref{eq:1_7_above}, and \cref{sec:4_4} below for further details).

The natural numbers $\cost{n}{\littleM}{d} \in \N$, $d,\littleM,n\in\N$, in \cref{th:1} measure the computational cost of the proposed MLP approximations.
More specifically, for every $d,\littleM,n\in\N$ we have that $\cost{n}{\littleM}{d}$ is the sum of the number of function evaluations of the nonlinearity $\smallF \colon \R \to \R$, of the number of function evaluations of the initial value function $\R^d \ni x \mapsto \smallU_d(0,x) \in \R$, 
and of the number of one-dimensional random variables which are used to compute one realization of the 
MLP approximation $\mlp_{n,\littleM}^{d,0}(T,0) \colon \Omega \to \R$.
We also refer to \cref{fc_def1} in \cref{th:final} in \cref{sec:main_results} below for the precise specification of the natural numbers $\cost{n}{\littleM}{d} \in \N$, $d,\littleM,n\in\N$.

\cref{th:1} reveals that the MLP approximations 
in \cref{mlp_main} approximate the values 
$ \smallU_d( T, 0 ) \in \R $, $ d \in \N$,
of the solution functions 
$\smallU_d \colon [0,T] \times \R^d \to \R$, $d\in\N$,
at the terminal time $ t = T $ 
and at the space point $ x = 0 \in \R^d $ 
with a computational effort which grows 
at most poly\-nom\-i\-al\-ly in the PDE dimension $ d \in \N $
and up to an arbitrarily small polynomial order 
at most quadratically in the reciprocal of the prescribed 
approximation accuracy $ \varepsilon > 0 $. 
This arbitrarily small polynomial order is described 
through the real number $ \delta \in (0,\infty) $ 
in \cref{th:1}.

Due to the fact that the MLP approximations
proposed in \cref{mlp_main} differ slightly from the MLP approximations which have been previously employed in $L^2$-error analyses in the scientific literature, we now briefly sketch the main ideas in the proof of \cref{th:1}.
The first step in our sketch of the proof of \cref{th:1}
is to reformulate the PDEs 
under consideration as stochastic fixed-point equations.
Specifically, in the context of \cref{eq:1} we have that the Feynman-Kac formula proves that the solution functions $\smallU_d \colon [0,T]\times \R^d \to \R$, $d\in\N$, of the PDEs in \cref{eq:1} are the unique at most polynomially growing functions which satisfy for all $d\in\N$, $\theta\in\Theta$, $t\in[0,T]$, $x\in\R^d$ that
\begin{equation}\label{eq:1_4_above}
\smallU_d(t,x) = \E\br[\big]{\smallU_d(0,x+\sqrt{2}\fwpr_{t}^{d,\theta})} + \int_0^t \E\br[\big]{\smallF(\smallU_d(s,x+\sqrt{2}\fwpr_{t-s}^{d,\theta}))}\,ds.
\end{equation}
In the next step 
we note that \cref{mlp_main},
the assumption that $W^{d,\theta} \colon [0,T] \times \Omega \to \R^d$, $d \in \N$, $\theta\in\Theta$, are independent standard Brownian motions, and the assumption that $\fu^\theta \colon \Omega \to [0,1]$, $\theta\in\Theta$, are i.i.d.\ random variables
assure that for all $n\in\N_0$, $d,m\in\N$, $\theta\in\Theta$, $t\in[0,T]$, $x\in\R^d$ it holds that
\begin{equation}\label{eq:1_5_above}
\begin{split}
& \E\br[\big]{ \mlp_{n,m}^{d,\theta}(t,x) }
- \1_\N(n)\,\E\br[\big]{ \smallU_d\pr[\big]{0, x + \sqrt{2} W^{d,\theta}_{t} } } \\
& = t \br[\Bigg]{ \sum_{i=0}^{n-1} \E\br[\Big]{ \smallF \pr[\big]{ \mlp_{i,m}^{d,(\theta,i)} (t \fu^\theta , x + \sqrt{2} W_{t-t \fu^\theta}^{d,\theta} ) } - \1_\N(i) \smallF \pr[\big]{ \mlp_{i-1,m}^{d,(\theta,-i)} ( t \fu^\theta , x + \sqrt{2} W_{t-t \fu^\theta}^{d,\theta} ) } } } \\
& = t \br[\Bigg]{ \sum_{i=0}^{n-1} \E\br[\Big]{ \smallF \pr[\big]{ \mlp_{i,m}^{d,\theta} (t \fu^\theta , x + \sqrt{2} W_{t-t \fu^\theta}^{d,\theta} ) } - \1_\N(i) \smallF \pr[\big]{ \mlp_{i-1,m}^{d,\theta} ( t \fu^\theta , x + \sqrt{2} W_{t-t \fu^\theta}^{d,\theta} ) } } } 
\end{split}
\end{equation}
(cf.\ \cref{properties_approx,lem:mlp_expectation} and \cref{mlp_stab_0a_pre1} in the proof of \cref{lem:mlp_expectation} for the details).
In addition, we observe that
the assumption that $W^{d,\theta} \colon [0,T] \times \Omega \to \R^d$, $d \in \N$, $\theta\in\Theta$, are independent standard Brownian motions, the assumption that $\fu^\theta \colon \Omega \to [0,1]$, $\theta\in\Theta$, are i.i.d.\ random variables, and a telescoping sum argument demonstrate that for all $n\in\N_0$, $d,m\in\N$, $\theta\in\Theta$, $t\in[0,T]$, $x\in\R^d$ it holds that
\begin{align}\label{eq:1_5a_above}
& t \br[\Bigg]{ \sum_{i=0}^{n-1} \E\br[\Big]{ \smallF \pr[\big]{ \mlp_{i,m}^{d,\theta} (t \fu^\theta , x + \sqrt{2} W_{t-t \fu^\theta}^{d,\theta} ) } - \1_\N(i) \smallF \pr[\big]{ \mlp_{i-1,m}^{d,\theta} ( t \fu^\theta , x + \sqrt{2} W_{t-t \fu^\theta}^{d,\theta} ) } } }
\\
& = \1_\N(n)\, t \, \E\br[\Big]{ \smallF \pr[\big]{ \mlp_{n-1,m}^{d,\theta} (t \fu^\theta , x + \sqrt{2} W_{t-t \fu^\theta}^{d,\theta} ) } } 
= \1_\N(n) \br*{ \int_0^t \E\br[\Big]{ \smallF \pr[\big]{ \mlp_{n-1,m}^{d,\theta} (s , x + \sqrt{2} W_{t-s}^{d,\theta} ) } } \, ds } \nonumber
\end{align}
(cf.\ \cref{lem:integrable,lem:mlp_expectation}).
Combining \cref{eq:1_4_above}, \cref{eq:1_5_above}, and \cref{eq:1_5a_above} indicates that for all $n\in\N_0$, $d,m\in\N$, $\theta\in\Theta$, $t\in[0,T]$, $x\in\R^d$ we have that
\begin{align}\label{eq:1_5b_above}
\E\br[\big]{ \mlp_{n,m}^{d,\theta}(t,x) }
& =  \1_\N(n) \pr*{ \E\br[\big]{ \smallU_d\pr[\big]{0, x + \sqrt{2} W^{d,\theta}_{t} } } 
+ \int_0^t \E\br[\Big]{ \smallF \pr[\big]{ \mlp_{n-1,m}^{d,\theta} (s , x + \sqrt{2} W_{t-s}^{d,\theta} ) } } \, ds }
\nonumber \\
& \approx \1_\N(n) \pr*{ \E\br[\big]{ \smallU_d\pr[\big]{0, x + \sqrt{2} W^{d,\theta}_{t} } } 
+ \int_0^t \E\br[\Big]{ \smallF \pr[\big]{ \smallU_d (s , x + \sqrt{2} W_{t-s}^{d,\theta} ) } } \, ds } \\
& = \1_\N(n) \, \smallU_d(t,x) \nonumber
\end{align}
(cf.\ \cref{lem:mlp_expectation,mlp_stab}).
Observe that \cref{eq:1_5b_above} suggests that the proposed MLP approximations $\mlp_{n,m}^{d,\theta} \colon [0,T] \times \R^d \times \Omega \to \R$, $n\in\N_0$, $d,m\in\N$, $\theta \in \Theta$, behave in expectation like Picard iterations for the stochastic fixed-point equations in \cref{eq:1_4_above}.
The final step in our sketch of the proof of \cref{th:1}
is to employ a Monte Carlo approach to approximate the expectations in \cref{eq:1_5_above}.
This final step is where the MLP approximations proposed in \cref{mlp_main} differ from the MLP approximations which have been previously employed in $L^2$-error analyses in the scientific literature.
Specifically, the MLP approximations proposed in \cref{mlp_main} use the fact that for all $n\in\N_0$, $i\in\{0,1,\ldots,n-1\}$, $d\in\N$, $\theta \in \Theta$, $t\in[0,T]$, $x\in\R^d$ we have that

\begin{samepage}
\begin{multline}\label{eq:1_6_above}
\tfrac{1}{(\funcM(\littleM))^{n-i}} \SmallSum_{k=1}^{(\funcM(\littleM))^{n-i}} \Bigl[ \smallF \pr[\big]{ \mlp_{i,m}^{d,(\theta,i,k)} (t \fu^{(\theta,i,k)} , x + \sqrt{2} W_{t-t \fu^{(\theta,i,k)}}^{d,{(\theta,i,k)}} ) } \\ - \1_\N(i) \smallF \pr[\big]{ \mlp_{i-1,m}^{d,(\theta,-i,k)} ( t \fu^{(\theta,i,k)} , x + \sqrt{2} W_{t-t \fu^{(\theta,i,k)}}^{d,{(\theta,i,k)}} ) } \Bigr]
\end{multline}
\end{samepage}
is a Monte Carlo approximation of 
\begin{equation}\label{eq:1_7_above}
\E\br[\Big]{ \smallF \pr[\big]{ \mlp_{i,m}^{d,(\theta,i)} (t \fu^\theta , x + \sqrt{2} W_{t-t \fu^\theta}^{d,\theta} ) } - \1_\N(i) \smallF \pr[\big]{ \mlp_{i-1,m}^{d,(\theta,-i)} ( t \fu^\theta , x + \sqrt{2} W_{t-t \fu^\theta}^{d,\theta} ) } }
\end{equation}
employing $(\phi(m))^{n-i} \in \N$ samples.
The function $\phi \colon \N \to \N$ thus determines the number of samples used in the Monte Carlo approximations in the MLP approximations proposed in \cref{mlp_main}.

In our $L^p$-error analysis the specific choice of $\phi$ is a subtle issue and, in particular, in our $L^p$-error analysis there is some fine-tuning needed in the choice of the function $\phi$.
On the one hand, the function $\phi$ must be chosen large enough so that the error due to approximating expectations via Monte Carlo averages is small enough.
On the other hand, in our recursive Gronwall-type $L^p$-error analysis in 
\cref{fn_gron} in \cref{sec:error_rec}
and \cref{mlp_stab} in \cref{sec:mlp_stab}
the exponential term $\exp(m^{\nicefrac{p}{2}}/p)$
arises in the upper bounds
(see \cref{eq:fn_gron} in \cref{fn_gron},
\cref{mlp_stab_estimate} in \cref{mlp_stab},
and \cref{lim_sup} in the proof of \cref{th:final})
where 
$\littleMM \in \N$ will be replaced by $\phi(\littleMM)$.
To control this term, our $L^p$-error analysis employs the assumption that 
$(\phi(\littleM)^{\nicefrac{p}{2}}/\littleM)_{\littleM\in\N}$ is a bounded sequence.
More specifically, if $p \in (0,2]$, then $\phi$ may be chosen to be the identity, but
if $p \in (2,\infty)$, then $\phi$ must grow much slower and the choice $\forall\,m\in\N \colon \funcM(\littleM) = \funcMrep{\littleM}$ is a suitable $p$-independent choice.

The remainder of this article is structured as follows.
In \cref{sec:2} we establish regularity properties for solutions of stochastic fixed-point equations.
Afterwards, in \cref{sec:3} we introduce MLP approximations for the stochastic fixed-point equations from \cref{sec:2}, we study their measurability and integrability properties, and we establish $L^p$-error bounds between the exact solutions of the stochastic fixed-point equations and the MLP approximations.
Finally, in \cref{sec:main_results} we establish some elementary estimates for full-history recursions and combine these estimates with the regularity properties for solutions of stochastic fixed-point equations, which we established in \cref{sec:2}, and the $L^p$-error analysis for MLP approximations for stochastic fixed-point equations, which we established in \cref{sec:3}, to obtain a computational complexity analysis for MLP approximations for semilinear partial differential equations.


\section{Stochastic fixed-point equations}\label{sec:2}

In this section we establish in \cref{cor:sol_bd1} below appropriate regularity results for solutions of stochastic fixed-point equations with polynomially growing solutions.
In \cref{cor:sol_bd1} we assume, among other things, that the nonlinearity $f \colon [0,T] \times \R^d \times \R \to \R$ and the terminal condition $g \colon \R^d \to \R$ of the stochastic fixed-point equation in \cref{eq:2_34} satisfy the polynomial growth bound that there exist $\boundFG , p \in [0,\infty)$ such that for all $t\in[0,T]$, $x = (x_1,x_2,\ldots,x_d) \in\R^d$ it holds that
\begin{equation}\label{explain_0_0}
\max\{ \abs{f(t,x,0)}, \abs{g(x)} \} \le \boundFG \pr[\big]{ 1 + [\textstyle\sum_{k=1}^d \abs{x_k}^2]^{\nicefrac{p}{2}} }
\end{equation}
(see above \cref{eq:2_34} in \cref{cor:sol_bd1}).
Observe that in the case $x = 0 \in \R^d$, $p=0$ we have that \cref{explain_0_0} reduces to the condition that for all $t\in[0,T]$ it holds that
$\max\{ \abs{f(t,0,0)}, \abs{g(0)} \} \le \boundFG \pr[]{ 1 + [\textstyle\sum_{k=1}^d \abs{0}^2]^{\nicefrac{p}{2}} }
= \boundFG ( 1 + 0^0 ) = 2\boundFG$.

Our proof of \cref{cor:sol_bd1} uses the regularity result for stochastic fixed-point equations with Lipschitz continuous nonlinearities in \cref{lemma:2.2} below. 
Similar regularity results for stochastic fixed-point equations can, e.g., be found in Hutzenthaler et al.\ \cite[Lemma 2.2]{HutzenthalerJentzenKruse2019}.
Our proof of \cref{lemma:2.2} uses the well-known backward formulation of the Gronwall inequality in \cref{lem:gronwall} below.
In our proof of \cref{lem:gronwall} we use the well-known forward formulation of the Gronwall inequality in
\cref{lem:gronwall_normal}.
\cref{lem:gronwall_normal} is a direct consequence of, e.g., the generalized Gronwall inequality in Henry \cite[Lemma 7.1.1]{h81}.

\subsection{Gronwall-type inequalities}

\begin{lemma}\label{lem:gronwall_normal}
Let $T, \gronB \in [0,\infty)$, let $\gronA \colon [0,T] \to [0,\infty)$ be a function,
let $\gronC  \colon [0,T] \to [0,\infty]$ be measurable, and assume for all $t\in[0,T]$ that $\int_0^t \gronC(s)\,ds < \infty$ and
\begin{equation}\label{1lem:gronwall_1}
\gronC(t) \le \gronA(t) + \gronB \int_0^t \gronC(s)\,ds. 
\end{equation}
Then
it holds for all $t\in[0,T]$ that 
$\gronC(t) \le \br[]{ \sup_{s\in[0,t]}\gronA(s) } \exp(\gronB t)$.
\end{lemma}

\begin{corollary}\label{lem:gronwall}
Let $T, \gronB \in [0,\infty)$, let $\gronA \colon [0,T] \to [0,\infty)$ be non-increasing, let $\gronC \colon [0,T] \to [0,\infty]$ be measurable, and assume for all $t\in[0,T]$ that $\int_t^T \gronC(s) \, ds < \infty$ and
\begin{equation}\label{lem:gronwall_1}
\gronC(t) \le \gronA(t) + \gronB \int_t^T \gronC(s) \, ds. 
\end{equation}
Then
it holds for all $t\in[0,T]$ that 
$\gronC(t) \le \gronA(t) \exp(\gronB (T-t)) < \infty$.
\end{corollary}

\begin{proof}[Proof of \cref{lem:gronwall}]
Throughout this proof let $\gronSymbb \colon [0,T] \to [0,\infty]$ and $\gronSymbbb \colon [0,T] \to [0,\infty)$ satisfy for all $t\in[0,T]$ that
\begin{equation}\label{reverse1}
\gronSymbb(t) = \gronC(T-t) \qquad \text{and} \qquad \gronSymbbb(t) = \gronA(T-t).
\end{equation}
Note that the hypothesis that $\gronC$ is measurable and \cref{reverse1} ensure that $\gronSymbb$ is measurable. In addition, observe that the hypothesis that for all $t\in[0,T]$ it holds that $\int_t^T \gronC(s)\,ds < \infty$ and \cref{reverse1} assure that for all $t\in[0,T]$ it holds that
\begin{equation}\label{reverse3}
\int_0^t \gronSymbb(s) \, ds 
= \int_0^t \gronC(T-s) \, ds 
= \int_{T-t}^T \gronC(s) \, ds 
\le \int_{0}^T \gronC(s) \, ds < \infty.
\end{equation}
Moreover, note that the hypothesis that $\gronA$ is non-increasing and \cref{reverse1} guarantee that $\gronSymbbb$ is non-decreasing.
Furthermore, observe that \cref{lem:gronwall_1}, \cref{reverse1}, and \cref{reverse3} demonstrate that for all $t\in[0,T]$ it holds that
\begin{equation}
\gronSymbb(t)
= \gronC(T-t)
\le \gronA(T-t) + \gronB \int_{T-t}^T \gronC(s) \, ds 
= \gronSymbbb(t) + \gronB \int_0^t \gronC(T-s) \, ds
= \gronSymbbb(t) + \gronB \int_0^t \gronSymbb(s) \, ds.
\end{equation}
This, \cref{reverse1}, the fact that $\gronSymbb$ is measurable, \cref{reverse3}, the fact that $\gronSymbbb$ is non-decreasing, and \cref{lem:gronwall_normal} (applied with $T \with T$, $\gronB \with \gronB$, $\gronA \with \gronSymbbb$, $\gronC \with \gronSymbb$ in the notation of \cref{lem:gronwall_normal}) prove that for all $t\in[0,T]$ it holds that
\begin{equation}
\gronSymbb(t) \le \left[ \sup_{s\in[0,t]} \gronSymbbb(s) \right] \exp(\gronB t) = \gronSymbbb(t) \exp(\gronB t) < \infty.
\end{equation}
Combining this and \cref{reverse1} establishes that for all $t\in[0,T]$ it holds that
\begin{equation}
\gronC(t) \le \gronA(t) \exp(\gronB (T-t)) < \infty.
\end{equation}
The proof of \cref{lem:gronwall} is thus complete.
\end{proof}

\subsection{A priori bounds for solutions of stochastic fixed-point equations}

\begin{lemma}\label{lemma:2.2}
Let $d\in\N$, $T, \LipConstF\in[0,\infty)$, $q\in[1,\infty)$, 
$\smallF \in C([0,T]\times \R^d \times \R,\R)$, $\funcG \in C(\R^d,\R)$, $\smallU \in C([0,T]\times \R^d, \R)$,
let $(\Omega, \cF,\P)$ be a probability space, let $\fwpr \colon [0,T] \times \Omega \to \R^d$ be a standard Brownian motion, and assume for all $t\in[0,T]$, $x\in\R^d$, $v,w\in\R$ that $\abs{ \smallF(t,x,v) - \smallF(t,x,w)} \le \LipConstF \abs{v-w}$, 
$\E[\abs{ \funcG(x + \fwpr_{T-t})} + \int_t^T \abs{ \smallF(s,x+\fwpr_{s-t},\smallU(s,x+\fwpr_{s-t}))} \,ds] + \int_t^T (\E[\abs{\smallU(s,x+\fwpr_s)}^q])^{\nicefrac{1}{q}} \, ds < \infty$,
and
\begin{equation}\label{lemma:2.2_u_rep}
\smallU(t,x) = \E\br[\big]{\funcG(x+\fwpr_{T-t})} + \int_t^T \E\br[\big]{\smallF(s,x+\fwpr_{s-t},\smallU(s,x+\fwpr_{s-t}))}\,ds.
\end{equation}
Then 
it holds for all $t\in[0,T]$, $x\in\R^d$ that
\begin{align}
& \pr[\Big]{\E\br[\Big]{\abs[\big]{\smallU(t,x+\fwpr_t)}^q}}^{\!\!\nicefrac{1}{q}} \label{eq:3_18} \\
& \le \exp\pr[\big]{L(T-t)} \br[\Bigg]{ \pr[\Big]{ \E\br[\Big]{\abs[\big]{ \funcG(x+\fwpr_{T}) }^q } }^{\!\!\nicefrac{1}{q}} 
+ (T-t)^{\nicefrac{(q-1)}{q}} \pr[\bigg]{ \int_{t}^T \E\br[\Big]{\abs[\big]{ \smallF(s,x+\fwpr_{s},0)}^q} \, ds }^{\!\!\nicefrac{1}{q}} }. \nonumber
\end{align}
\end{lemma}

\begin{proof}[Proof of \cref{lemma:2.2}]
Throughout this proof let $\alpha \colon [0,T]\times \R^d \to [0,\infty]$ satisfy for all $t\in[0,T]$, $x\in\R^d$ that
\begin{equation}\label{alpha_def1}
\alpha(t,x) = \pr[\big]{ \E\br[\big]{\abs[]{ \funcG(x+\fwpr_{T}) }^q } }^{\!\nicefrac{1}{q}} + (T-t)^{\nicefrac{(q-1)}{q}} \pr[\bigg]{ \int_{t}^T \E\br[\big]{\abs[]{ f(s,x+\fwpr_{s},0)}^q} \, ds }^{\!\!\nicefrac{1}{q}}
\end{equation}
and assume without loss of generality that for all $x\in\R^d$ it holds that $\alpha(0,x) < \infty$.
Note that \cref{alpha_def1} 
ensures that for all $x\in\R^d$ it holds that $[0,T] \ni t \mapsto \alpha(t,x) \in [0,\infty)$ is non-increasing.
In addition, observe that \cref{lemma:2.2_u_rep}, the triangle inequality, Jensen's inequality, Fubini's theorem, and the fact that $\fwpr$ has independent 
increments
assure that
for all $t \in [0,T]$, $x\in\R^d$ it holds that 
\begin{align}\label{upper_exact:eq1}
& \pr[\big]{ \E\br[\big]{ \abs[]{ \smallU(t,x +\fwpr_t) }^q } }^{\!\nicefrac{1}{q}} \nonumber\\
& = \pr[\bigg]{ \E\br[\bigg]{\abs[\Big]{ \E\br[\big]{ \funcG(x+\fwpr_{T-t} + \fwpr_t) } + \textstyle\int_t^{T} \E\br[\big]{ \smallF(s,x+\fwpr_{s-t}+\fwpr_t,\smallU(s,x+\fwpr_{s-t} + \fwpr_{t})) } \,ds }^q } }^{ \! \! \nicefrac{1}{q}} \nonumber \\
& = \pr[\bigg]{ \E\br[\bigg]{ \abs[\Big]{ \E\br[\big]{ \funcG(x+\fwpr_{T}) } + \textstyle\int_t^{T} \E\br[\big]{ \smallF(s,x+\fwpr_{s},\smallU(s,x+\fwpr_{s})) } \,ds }^q } }^{ \! \! \nicefrac{1}{q}} \\
& \le \pr[\Big]{ \E\br[\Big]{ \abs[\big]{ \E\br[\big]{ \funcG(x+\fwpr_{T}) } }^q } }^{\!\!\nicefrac{1}{q}} 
+ \pr[\bigg]{ \E\br[\Big]{ \abs[\big]{ \textstyle\int_{t}^T \E\br[\big]{ \smallF(s,x+\fwpr_{s},\smallU(s,x+\fwpr_{s})) } \,ds }^q } }^{\!\!\nicefrac{1}{q}} \nonumber \\
& \le \pr[\big]{ \E\br[\big]{ \abs[]{ \funcG(x+\fwpr_{T}) }^q } }^{\!\nicefrac{1}{q}} 
+ \int_{t}^T 
\pr[\big]{ \E\br[\big]{ \abs[]{ \smallF(s,x+\fwpr_{s},\smallU(s,x+\fwpr_{s})) }^q } }^{\!\nicefrac{1}{q}} \, ds. \nonumber
\end{align}
Next note that the triangle inequality and the hypothesis that for all $t\in[0,T]$, $x\in\R^d$, $v,w\in\R$ it holds that $\abs{ f(t,x,v)-f(t,x,w) } \le \LipConstF \abs{ v-w }$ demonstrate that for all $t\in[0,T]$, $x\in\R^d$ it holds that
\begin{align} \label{eq:3_18a}
& \int_{t}^T \pr[\big]{ \E\br[\big]{ \abs[]{ \smallF(s,x+\fwpr_{s},\smallU(s,x+\fwpr_{s})) }^q } }^{\!\nicefrac{1}{q}} \, ds
\le \int_{t}^T \pr[\big]{ \E\br[\big]{ \abs[]{ \smallF(s,x+\fwpr_{s},0)) }^q } }^{\!\nicefrac{1}{q}} \, ds \nonumber \\
& + \int_t^T \pr[\big]{ \E\br[\big]{ \abs[]{ \smallF(s,x+\fwpr_{s},\smallU(s,x+\fwpr_{s})) - \smallF(s,x+\fwpr_{s},0) }^q } }^{\!\nicefrac{1}{q}} \, ds \\
& \le \int_{t}^T \pr[\big]{ \E\br[\big]{ \abs[]{ \smallF(s,x+\fwpr_{s},0)) }^q } }^{\!\nicefrac{1}{q}} \, ds
+ \LipConstF \int_{t}^T \pr[\big]{ \E\br[\big]{ \abs[]{ \smallU(s,x+\fwpr_{s}) }^q } }^{\!\nicefrac{1}{q}} \, ds. \nonumber
\end{align}
Furthermore, observe that H{\"o}lder's inequality shows that for all $t\in[0,T]$, $x\in\R^d$ it holds that
\begin{align}
\int_{t}^T \pr[\big]{ \E\br[\big]{ \abs[]{ \smallF(s,x+\fwpr_{s},0)) }^q } }^{\!\nicefrac{1}{q}} \, ds 
& = \pr[\Bigg]{ \br[\bigg]{ \int_{t}^T \pr[\big]{ \E\br[\big]{ \abs[]{ \smallF(s,x+\fwpr_{s},0)) }^q } }^{\!\nicefrac{1}{q}} \, ds }^q }^{\!\!\nicefrac{1}{q}} \nonumber \\
& \le \pr[\bigg]{ (T-t)^{q-1} \int_{t}^T \E\br[\big]{ \abs[]{ \smallF(s,x+\fwpr_{s},0)) }^q } \, ds }^{\!\!\nicefrac{1}{q}} \\
& = (T-t)^{\nicefrac{(q-1)}{q}} \pr[\bigg]{ \int_{t}^T \E\br[\big]{ \abs[]{ \smallF(s,x+\fwpr_{s},0)) }^q } \, ds }^{\!\!\nicefrac{1}{q}} . \nonumber
\end{align}
Combining this, \cref{alpha_def1}, \cref{upper_exact:eq1}, and \cref{eq:3_18a}
guarantees that for all $t \in [0,T]$, $x\in\R^d$ it holds that
\begin{equation}
\label{upper_exact:eq8}
\pr[\big]{ \E\br[\big]{ \abs[]{ \smallU(t,x +\fwpr_t) }^q } }^{\!\nicefrac{1}{q}}  
\le \alpha(t,x) + \LipConstF \int_{t}^T \pr[\big]{ \E\br[\big]{ \abs[]{ \smallU(s,x+\fwpr_{s}) }^q } }^{\!\nicefrac{1}{q}} \, ds.
\end{equation}
This,  
\cref{alpha_def1},
the fact that for all $x\in\R^d$ it holds that $[0,T] \ni t \mapsto \alpha(t,x) \in [0,\infty)$ is non-increasing, the hypothesis that for all $t\in[0,T]$, $x\in\R^d$ it holds that $\int_t^T (\E[\lvert u(s,x+\fwpr_s)\rvert^q])^{\nicefrac{1}{q}}\,ds < \infty$, and
\cref{lem:gronwall} (applied for every $x\in\R^d$ with $T \with T$, $\gronB \with \LipConstF$, $\gronA \with \pr[]{ [0,T] \ni t \mapsto \alpha(t,x) \in [0,\infty) }$, $\gronC \with \pr[]{ [0,T] \ni t \mapsto (\E[|u(t,x+\fwpr_t)|^q])^{\nicefrac{1}{q}} \in [0,\infty] }$ in the notation of \cref{lem:gronwall}) establish that
for all $t \in [0, T]$, $x\in\R^d$
it holds that
\begin{equation}\label{eq_2_31}
\pr[\big]{ \E\br[\big]{ \abs[]{ \smallU(t,x+\fwpr_{t}) }^q } }^{\!\nicefrac{1}{q}} 
\le \alpha(t,x) \exp\pr[\big]{ \LipConstF (T-t) }.
\end{equation}
The proof of \cref{lemma:2.2} is thus complete.
\end{proof}

\begin{definition}\label{euclid_norm}
We denote by $\norm{\cdot} \colon \pr{ \bigcup_{d\in\N}\! \R^d } \to [0,\infty)$ the function which satisfies for all $d\in\N$, $x = (x_1,x_2,\dots,x_d) \in \R^d$ that $\norm{x} = \br{ \sum_{k=1}^d \abs{x_k}^2 }^{\nicefrac{1}{2}}$.
\end{definition}

\begin{corollary}\label{cor:sol_bd1}
Let $d \in \N$, $T, \LipConstF, \boundFG, p \in [0,\infty)$, $q \in [1,\infty)$, $\smallF \in C([0,T]\times \R^d \times \R,\R)$, $\funcG \in C(\R^d,\R)$, $\smallU \in C([0,T]\times \R^d, \R)$,
let $(\Omega, \cF ,\P)$ be a probability space, let $\fwpr \colon [0,T] \times \Omega \to \R^d$ be a standard Brownian motion,
and assume for all $t\in[0,T]$, $x\in\R^d$, $v,w\in\R$ that $\abs{ \smallF(t,x,v) - \smallF(t,x,w) } \le \LipConstF \abs{ v-w }$, $\max\{\abs{ \smallF(t,x,0) },\abs{ \funcG(x) } \} \le \boundFG (1 + \norm{x}^p)$,
$\E\br[]{ \abs{ \funcG(x + \fwpr_{T-t}) } + \int_t^T \abs{ \smallF(s,x+\fwpr_{s-t},\smallU(s,x+\fwpr_{s-t})) } \,ds } < \infty$,
and
\begin{equation}
\smallU(t,x) = \E\br[\big]{ \funcG(x+\fwpr_{T-t}) } + \int_t^T \E\br[\big]{ \smallF(s,x+\fwpr_{s-t},\smallU(s,x+\fwpr_{s-t})) }\,ds
\end{equation}
(cf.\ \cref{euclid_norm}).
Then it holds for all $t\in[0,T]$, $x\in\R^d$ that
\begin{equation}\label{eq:2_34}
\pr[\Big]{ \E\br[\Big]{ \abs[\big]{ \smallU(t,x+\fwpr_t) }^q } }^{\!\!\nicefrac{1}{q}} 
\le \boundFG (T+1) \exp(\LipConstF T) \left[ \sup_{s\in[0,T]} \pr[\Big]{ \E\br[\Big]{ \pr[\big]{ 1 + \norm{ x + \fwpr_s}^p }^{\!q} } }^{\!\!\nicefrac{1}{q}} \right] < \infty.
\end{equation}
\end{corollary}

\begin{proof}[Proof of \cref{cor:sol_bd1}]
Throughout this proof let
$\mathbb{F}_t \subseteq \cF$, $t\in[0,T]$, satisfy
for all $t\in[0,T]$ that
\begin{equation}\label{filtration1}
\mathbb{F}_t = 
\begin{cases}
\bigcap_{s\in(t,T]} \sigmaAlgebra \pr[\big]{ \sigmaAlgebra \pr{ \fwprr_r \colon r \in [0,s] } \cup \{A \in \cF \colon \P(A) = 0\} } & \colon t < T \\
\sigmaAlgebra \pr[\big]{ \sigmaAlgebra \pr{ \fwprr_s \colon s \in [0,T] } \cup \{A \in \cF \colon \P(A) = 0\} } & \colon t = T
\end{cases}
\end{equation}
and let $\Ffn \in C([0,T]\times \R^d,\R^d)$ and $\Gfn \in C([0,T]\times \R^d,\R^{d\times d})$ satisfy for all $t\in[0,T]$, $x,v\in\R^d$ that
$\Ffn(t,x) = 0$ 
and $\Gfn(t,x)v = v$.
Note that \cref{filtration1} guarantees that $\mathbb{F}_t \subseteq \cF$, $t\in[0,T]$, satisfies that
\begin{enumerate}[label=(\Roman*)]
\item\label{filt1} it holds that $\{ A\in\cF \colon \P(A) = 0\} \subseteq \mathbb{F}_0$ and
\item\label{filt2} it holds for all $t\in[0,T]$ that $\mathbb{F}_t = \cap_{s\in(t,T]} \mathbb{F}_s$.
\end{enumerate}
Combining \cref{filt1,filt2}, \cref{filtration1}, and, e.g., Hutzenthaler et al.\ \cite[Lemma 2.17]{HutzenthalerPricing2019} 
(applied with $m \with d$, $T \with T$, $W \with \fwprr$, $\mathbb{H}_t \with \mathbb{F}_t$, $(\Omega,\cF,\P,(\mathbb{F}_t)_{t\in[0,T]}) \allowbreak \with (\Omega,\cF,\P,(\sigmaAlgebra(\fwprr_s \colon s \in [0,t]) \cup \{ A\in\cF \colon \P(A) = 0\} )_{t\in[0,T]})$ in the notation of \cite[Lemma 2.17]{HutzenthalerPricing2019})
hence assures that $\fwprr \colon [0,T] \times \Omega \to \R$ is a standard $(\Omega,\cF,\P,(\mathbb{F}_t)_{t\in[0,T]})$-Brownian motion.
Combining this, the hypothesis that for all $t\in[0,T]$, $x\in\R^d$, $v,w\in\R$ it holds that $\abs{ \smallF(t,x,v) - \smallF(t,x,w) } \le \LipConstF \abs{ v-w }$, the hypothesis that for all $t\in[0,T]$, $x\in\R^d$ it holds that $\max\{ \abs{ \smallF(t,x,0) },\abs{ \funcG(x) } \} \le \boundFG (1 + \norm{ x }^p)$,
and Beck et al.\ \cite[Corollary 3.9]{beck2020nonlinear}
(applied with $d \with d$, $m \with d$, $T \with T$, $\LipConstF \with \max\{d^{\nicefrac{1}{2}},\LipConstF\}$, $\mathfrak{C} \with 0$, $f \with f$, $g \with g$, $\mu \with \Ffn$, $\sigma \with \Gfn$, $W \with \fwprr$, $(\Omega, \cF, \P, (\mathbb{F}_t)_{t\in[0,T]}) \with (\Omega, \cF, \P, (\mathbb{F}_t)_{t\in[0,T]})$ in the notation of \cite[Corollary 3.9]{beck2020nonlinear})
ensures that
\begin{align}
\sup_{s\in [0,T]}\sup_{y\in\R^d} \pr*{ \frac{\abs{ \smallU(s,y) }}{1 + \norm{ y }^p} } <
\infty.
\end{align}
This, the fact that for all $r, v, w \in [0,\infty)$ it holds that $(v + w)^r \le 2^{\max\{r-1,0\}} (v^r + w^r)$, the triangle inequality, 
and the fact that for all $r \in [0,\infty)$ it holds that $ \E [\norm{ \fwpr_T }^{r} ] < \infty$ demonstrate that for all $t\in[0,T]$, $x\in\R^d$ it holds that
\begin{align}\label{m03e}
& \int_t^T \pr[\big]{ \E\br[\big]{ \abs[]{ \smallU(s,x+\fwpr_{s}) }^q } }^{\!\nicefrac{1}{q}} \, ds
\le \left[\sup_{s\in [0,T]}\sup_{y\in\R^d} \frac{\abs{ u(s,y) }}{1 + \norm{ y }^p}\right]
\int_{0}^{T} \pr[\big]{ \E\br[\big]{ \pr[]{ 1 + \norm{ x + \fwpr_s }^p }^{q} } }^{\!\nicefrac{1}{q}} \, ds \nonumber \\
& \le \left[\sup_{s\in [0,T]}\sup_{y\in\R^d} \frac{\abs{ u(s,y) }}{1 + \norm{ y }^p}\right]
\int_{0}^{T} \br[\Big]{ 1 + 2^{\max\{p-1,0\}} \pr[\big]{ \E\br[\big]{ \pr[]{ \norm{ x }^p + \norm{ \fwpr_s }^p }^{q} } }^{\!\nicefrac{1}{q}} } \, ds \\
& \le
T \left[\sup_{s\in [0,T]}\sup_{y\in\R^d} \frac{\abs{ u(s,y) }}{1 + \norm{ y }^p}\right]
\br[\Big]{ 1 + 2^{\max\{p-1,0\}} \norm{ x }^p +
2^{\max\{p-1,0\}} \pr[\big]{ \E\br[\big]{ \norm{ \fwpr_T }^{pq} } }^{\!\nicefrac{1}{q}} }
< \infty. \nonumber
\end{align}
Combining this, the hypothesis that for all $t\in[0,T]$, $x\in\R^d$, $v,w\in\R$ it holds that $\abs{ \smallF(t,x,v) - \smallF(t,x,w) } \le \LipConstF \abs{ v-w }$, the hypothesis that for all $t\in[0,T]$, $x\in\R^d$ it holds that $\E\br[]{ \abs{ \funcG(x + \fwpr_{T-t}) } + \int_t^T \abs{ \smallF(s,x+\fwpr_{s-t},\smallU(s,x+\fwpr_{s-t})) } \,ds } < \infty$, and \cref{lemma:2.2} establishes that for all $t\in[0,T]$, $x\in\R^d$ it holds that
\begin{align}\label{upper_exact:eq2}
& \pr[\big]{ \E\br[\big]{ \abs[]{ \smallU(t,x+\fwpr_t) }^q } }^{\!\nicefrac{1}{q}} \\
& \le \exp\pr[\big]{ \LipConstF(T-t) } \br[\Bigg]{ \pr[\big]{ \E\br[\big]{ \abs[]{ \funcG(x+\fwpr_{T}) }^q } }^{\!\nicefrac{1}{q}} 
+ (T-t)^{\nicefrac{(q-1)}{q}} \pr[\bigg]{ \int_{t}^T \E\br[\big]{ \abs[]{ 
\smallF(s,x+\fwpr_{s},0) }^q } \, ds }^{\!\!\nicefrac{1}{q}} } \nonumber \\
& \le \exp\pr{ \LipConstF T } \br[\Bigg]{ \pr[\big]{ \E\br[\big]{ \abs[]{ \funcG(x+\fwpr_{T}) }^q } }^{\!\nicefrac{1}{q}} 
+ T^{\nicefrac{(q-1)}{q}} \pr[\bigg]{ \int_{0}^T \E\br[\big]{ \abs[]{ 
\smallF(s,x+\fwpr_{s},0) }^q } \, ds  }^{\!\!\nicefrac{1}{q}} }. \nonumber
\end{align}
Next observe that the hypothesis that for all $x\in\R^d$ it holds that $\abs{ g(x) } \le \boundFG (1 + \norm{ x }^p)$ and \cref{m03e} show that
for all $x\in\R^d$ it holds that
\begin{equation}
\label{upper_exact:eq3}
\E\br[\big]{ \abs[]{ \funcG\pr{ x +\fwpr_{T}} }^q }
\le \E\br[\big]{ \boundFG^q \pr[]{ 1 + \norm{ x + \fwpr_{T} }^p }^{q} } 
\le \sup_{s\in [0,T]} \E\br[\big]{ \boundFG^q \pr[]{ 1 + \norm{ x + \fwpr_s }^p }^{q} } < \infty.
\end{equation} 
In addition, note that the hypothesis that for all $t\in[0,T]$, $x\in\R^d$ it holds that $\abs{ f(t,x,0) } \le \boundFG (1 + \norm{ x }^p)$ and \cref{m03e} assure that for all $x\in\R^d$ it holds that
\begin{equation}\label{upper_exact:eq4}
\begin{split}
\pr[\bigg]{ \int_{0}^T \E\br[\big]{ \abs[]{ \smallF(s,x+\fwpr_{s},0) }^q } \, ds }^{\!\!\nicefrac{1}{q}}
& \le \pr[\bigg]{ \int_{0}^T \E\br[\big]{ \boundFG^q \pr[]{ 1 + \norm{ x + \fwpr_s }^p }^{q} } \, ds }^{\!\!\nicefrac{1}{q}} \\
& \le \boundFG T^{\nicefrac{1}{q}} \left[\sup_{s\in [0,T]} \pr[\big]{ \E\br[\big]{ \pr[]{ 1 + \norm{ x + \fwpr_s }^p }^{q} } }^{\!\nicefrac{1}{q}} \right] < \infty.
\end{split}
\end{equation}
Combining this, \cref{upper_exact:eq2}, and \cref{upper_exact:eq3} proves that for all $t\in[0,T]$, $x\in\R^d$ it holds that
\begin{equation}
\pr[\big]{ \E\br[\big]{ \abs[]{ \smallU(t,x+\fwpr_t) }^q } }^{\!\nicefrac{1}{q}} 
\le \boundFG (T+1) \exp(\LipConstF T) \left[ \sup_{s\in[0,T]} \pr[\big]{ \E\br[\big]{ \pr[]{ 1 + \norm{ x + \fwpr_s}^p }^{q} } }^{\!\nicefrac{1}{q}} \right] < \infty.
\end{equation}
The proof of \cref{cor:sol_bd1} is thus complete.
\end{proof}


\section{Full-history recursive multilevel Picard (MLP) approximations}\label{sec:3}

In this section we introduce and provide the $L^p$-error analysis for MLP approximations for solutions of stochastic fixed-point equations.
More specifically, we prove \cref{mlp_stab_cor} below, which is a non-recursive $L^p$-error bound for MLP approximations for solutions of stochastic fixed-point equations.
Our proof of \cref{mlp_stab_cor} uses \cref{mlp_stab}, which provides a potentially sharper $L^p$-error bound for MLP approximations for solutions of stochastic fixed-point equations.
Our proof of \cref{mlp_stab}, in turn, employs \cref{mlp_stab_pre} and the elementary auxiliary results in \cref{prop:factorial}, \cref{lem:talk1}, and \cref{fn_gron}.
Our proof of the recursive error bound in \cref{mlp_stab_pre} employs \cref{lem:mlp_variance}.
Our proof of \cref{lem:mlp_variance}, in turn, is based on \cref{lem:mlp_expectation} and the elementary Monte Carlo approximation results in \cref{Lp_monte_carlo}, \cref{prop:mlp}, and \cref{cor:exp_bd}.
Our proof of \cref{lem:mlp_expectation} uses \cref{properties_approx} and \cref{lem:integrable}, which are elementary results regarding the measurability and integrability of MLP approximations for solution of stochastic fixed-point equations, respectively.

\cref{properties_approx} is, e.g., proved as Hutzenthaler et al.\ \cite[Lemma 3.2]{HutzenthalerJentzenKruse2018}.
\cref{lem:integrable} is, e.g., proved as Hutzenthaler et al.\ \cite[Lemma 3.3]{HutzenthalerJentzenKruse2018}.
Only for completeness we include in this section the detailed proofs of \cref{properties_approx} and \cref{lem:integrable}, respectively.
\cref{prop:factorial} and \cref{lem:talk1} are well-known elementary results and we include their proofs for completeness, as well.
The elementary result \cref{fn_gron} is a slight generalization of the result in Hutzenthaler et al.\ \cite[Lemma 3.11]{HutzenthalerJentzenKruseNguyen2020}.

\subsection{MLP approximations}

\begin{definition}\label{rand_const}
Let $p \in [2,\infty)$. Then we denote by $\firstConstant{p} \in \R$ the real number given by
\begin{equation}\label{eq:rand_const}
\firstConstant{p} = \inf\left\{ c \in \R \colon \left[ 
\begin{aligned}
& \text{It holds for every probability space $(\Omega,\cF,\P)$ and every} \\
& \text{random variable $X \colon \Omega \to \R$ with $\E[\abs{X}] < \infty$ that } \\
& \pr[\big]{ \E\br[\big]{ \abs{ X - \E[X] }^p } }^{\!\nicefrac{1}{p}} \le c \pr[\big]{ \E\br[\big]{ \abs{X}^p } }^{\!\nicefrac{1}{p}}
\end{aligned}
\right] \right\} .
\end{equation}
\end{definition}

\begin{setting}\label{setting1}
Let $d, \littleM \in \N$, $T, \LipConstF, \boundFG, p \in [0,\infty)$, 
$\fq \in [2,\infty)$, 
$\fm = \secondConstant{\fq}$,
$\Theta = \bigcup_{n\in\N}\! \Z^n$, 
$\smallF \in C([0,T] \times \R^d \times \R,\R)$, $\funcG \in C(\R^d,\R)$, 
let $\funcF \colon C([0,T]\times\R^d,\R) \to C([0,T]\times\R^d,\R)$, assume for all $t\in[0,T]$, $x\in\R^d$, $w,\fw\in\R$, $v\in C([0,T]\times\R^d,\R)$ that
\begin{equation}\label{fn_cond}
\abs{ \smallF(t,x,w) - \smallF(t,x,\fw) } \le \LipConstF \abs{ w-\fw }, \qquad \max\{\abs{ \smallF(t,x,0) }, \abs{ \funcG(x) } \} \le \boundFG (1 + \norm{x}^p),
\end{equation}
and
\begin{equation}\label{big_F}
(\funcF(v))(t,x) = \smallF(t,x,v(t,x)),
\end{equation} 
let $(\Omega,\cF,\P)$ be a probability space,
let $\fu^\theta \colon \Omega \to [0,1]$, $\theta \in \Theta$, be i.i.d.\ random variables, 
assume for all $\theta\in\Theta$, $r\in (0,1)$ that $\P(\fu^\theta \le r ) = r$,
let $\cU^\theta \colon [0,T]\times \Omega\to [0,T]$, $\theta\in\Theta$, satisfy for all $t\in [0,T]$, $\theta\in\Theta$ that $\cU_t^\theta = t + (T-t)\fu^\theta$, 
let $W^\theta \colon [0,T] \times \Omega \to \R^d$, $\theta\in\Theta$, be independent standard Brownian motions,
assume that $(\cU^\theta)_{\theta\in\Theta}$ and $(W^\theta)_{\theta\in\Theta}$ are independent,
let $\smallU \in C([0,T]\times \R^d, \R)$ satisfy for all $t\in[0,T]$, $x\in\R^d$ that
$\E[\abs{ \funcG(x + W^0_{T-t}) } + \int_t^T \abs{ (\funcF(\smallU))(s,x+W^0_{s-t}) } \,ds] < \infty$
and
\begin{equation}\label{assumed_u}
\smallU(t,x) = \E\br[\big]{ \funcG(x+W^0_{T-t}) } + \int_t^T \E\br[\big]{ (\funcF(\smallU))(s,x+W^0_{s-t}) } \,ds,
\end{equation}
and let $\mlp_n^\theta \colon [0,T] \times \R^d\times \Omega\to \R$, 
$n\in\Z$, 
$\theta\in\Theta$, satisfy for all 
$n\in\N_0$, 
$\theta\in\Theta$, $t\in[0,T]$, $x\in\R^d$ that 
\begin{equation}\label{b37}
\begin{split}
\mlp_n^\theta(t,x) 
& = \frac{\1_\N(n)}{\littleM^n} \br[\Bigg]{ \sum_{k=1}^{\littleM^n} \funcG\pr[\bigl]{x + W_{T-t}^{(\theta,0,-k)}} } \\
& \quad + \sum_{i=0}^{n-1} \frac{(T-t)}{\littleM^{n-i}}\br[\Bigg]{ \sum_{k=1}^{\littleM^{n-i}} \pr[]{ \funcF(\mlp_i^{(\theta,i,k)}) - \1_\N(i) \funcF( \mlp_{i-1}^{(\theta,-i,k)}) } \pr[]{ \cU_t^{(\theta,i,k)}, x + W_{\mathcal{U}_t^{(\theta,i,k)}-t}^{(\theta,i,k)} } }
\end{split}
\end{equation}
(cf.\ \cref{euclid_norm,rand_const}).
\end{setting}

\subsection{Measurability properties of MLP approximations}\label{sec:mlp_props}

\begin{lemma}
\label{properties_approx}
Assume \cref{setting1}.
Then
\begin{enumerate}[label=(\roman *)]
\item  \label{properties_approx:item1}
it holds 
for all $n \in \N_0$, 
$\theta\in\Theta $ that
$\mlp_{ n}^{\theta } \colon [0, T] \times \R^d \times \Omega \to \R$ is a continuous random field,
\item  \label{properties_approx:item2}
it holds\footnote{Note that for every $\mathcal{A} \subseteq \powerset$ it holds that $\sigmaAlgebra(\mathcal{A})$ is a sigma-algebra on $\Omega$ and note that for every $\mathcal{A} \subseteq \powerset$ and every sigma-algebra $\mathcal{B}$ on $\Omega$ with $\mathcal{A} \subseteq \mathcal{B}$ it holds that $\sigmaAlgebra(\mathcal{A}) \subseteq \mathcal{B}$.} for all $n \in \N_0$, 
$\theta \in \Theta$ that
$\sigmaAlgebra( \mlp^\theta_{n} ) \subseteq \sigmaAlgebra( (\cU^{(\theta, \vartheta)})_{\vartheta \in \Theta}, (W^{(\theta, \vartheta)})_{\vartheta \in \Theta})$,
\item  \label{properties_approx:item3}
it holds 
for all $n \in \N_0$ that
$(\mlp_{ n}^{\theta })_{\theta\in\Theta}$, $(W^\theta)_{\theta\in\Theta}$, and $(\cU^\theta)_{\theta\in\Theta}$ are independent,
\item  \label{properties_approx:item4}
it holds
for all $n, m \in \N_0$, 
$i,k,\mathfrak{i},\mathfrak{k} \in \Z$ with $(i,k) \neq (\mathfrak{i},\fk)$ 
that
$(\mlp^{(\theta,i,k)}_{n})_{\theta\in\Theta}$ and $(\mlp^{(\theta,\mathfrak{i},\mathfrak{k})}_{m})_{\theta\in\Theta}$ are independent,
and
\item  \label{properties_approx:item5}
it holds 
for all $n \in \N_0$
that 
$(\mlp^\theta_{n})_{\theta \in \Theta}$
are identically distributed random variables.
\end{enumerate}
\end{lemma}

\begin{proof}[Proof of \cref{properties_approx}]
We first prove \cref{properties_approx:item1} by induction. For the base case $n=0$ note that \cref{b37} ensures that for all $\theta\in\Theta$, $t\in[0,T]$, $x\in\R^d$ it holds that $\mlp_0^\theta(t,x) = 0$. This implies that for all $\theta\in\Theta$ it holds that $\mlp_0^\theta \colon [0,T] \times \R^d \times \Omega \to \R$ is a continuous random field. This establishes \cref{properties_approx:item1} in the base case $n=0$. For the induction step $\N_0 \ni (n-1) \induct n \in \N$ let $n\in\N$ and assume that for every $k \in \{0,1,\dots,n-1\}$, $\theta\in\Theta$ it holds that $\mlp_k^\theta \colon [0,T] \times \R^d \times \Omega \to \R^d$ is a continuous random field. This, the hypothesis that $\smallF \in C([0,T] \times \R^d, \R)$, \cref{big_F}, and, e.g., Hutzenthaler et al.\ \cite[Item (i) in Lemma 2.9]{HutzenthalerJentzenKruse2018} 
(applied for every $n\in\N_0$, $\theta\in\Theta$ with $d \with d$, $T \with T$, $(\Omega,\cF,\P) \with (\Omega,\cF,\P)$, $F \with \funcF$, $U \with \mlp_k^\theta$ in the notation of \cite[Item (i) of Lemma 2.9]{HutzenthalerJentzenKruse2018})
imply that for all $k \in \{0,1,\dots,n-1\}$, $\theta\in\Theta$ it holds that
\begin{equation}\label{f_rand_fld}
[0,T] \times \R^d \times \Omega \ni (t,x,\omega) \mapsto \br[\big]{ \funcF\pr[\big]{ [0,T] \times \R^d \ni (s,y) \mapsto \mlp_k^\theta(s,y,\omega) \in \R } }(t,x) \in \R
\end{equation}
is a continuous random field.
Combining this, the hypothesis that $\funcG \in C(\R^d,\R)$, the fact that for all $\theta\in\Theta$ it holds that $W^\theta \colon [0,T] \times \Omega \to \R^d$ and $\cU^\theta \colon [0,T] \times \Omega \to [0,T]$ are continuous random fields, \cref{b37}, Hutzenthaler et al.\ \cite[Lemma 2.14]{HutzenthalerPricing2019}, and Beck et al.\ \cite[Lemma 2.4]{Kolmogorov} proves that for all $\theta\in\Theta$ it holds that $\mlp_n^\theta \colon [0,T] \times \R^d \times \Omega \to \R^d$ is a continuous random field. Induction thus establishes \cref{properties_approx:item1}.
Next note that 
\cref{f_rand_fld}, 
Beck et al.\ \cite[Lemma 2.4]{Kolmogorov}, and \cref{properties_approx:item1} assure that 
for all $n \in \N_0$, $\theta\in\Theta $ it holds that 
$\funcF(\mlp_n^\theta)$ is $ ( \mathcal{B}([0, T] \times \R^d ) \otimes \sigmaAlgebra(\mlp^\theta_{n}) )/ \mathcal{B}(\R )$-measurable.
This, \cref{b37}, 
the fact that 
for all $\theta \in \Theta$ it holds that 
$W^\theta$ is $ ( \mathcal{B}([0, T]) \otimes \sigmaAlgebra(W^\theta) )/ \mathcal{B}(\R^d )$-measurable,
the fact that 
for all $\theta \in \Theta$ it holds that 
$\cU^\theta$ is $ ( \mathcal{B}([0, T]) \otimes \sigmaAlgebra(\fu^\theta) )/ \mathcal{B}([0, T] )$-measurable,
and induction on $\N_0$ prove \cref{properties_approx:item2}.
Moreover, observe that \cref{properties_approx:item2} and the fact that 
for all $\theta \in \Theta$ it holds that
$(\cU^{(\theta, \vartheta)})_{\vartheta \in \Theta}, (W^{(\theta, \vartheta)})_{\vartheta \in \Theta}$,
$W^\theta$, and $\fu^\theta$ are independent establish \cref{properties_approx:item3}.
Furthermore, note that \cref{properties_approx:item2} and the fact that
for all $i,k,\mathfrak{i},\mathfrak{k}, \in \Z$, $\theta \in \Theta$ with $(i,k) \neq (\mathfrak{i},\mathfrak{k})$ it holds that 
$((\cU^{(\theta,i,k, \vartheta)})_{\vartheta \in \Theta}, (W^{(\theta,i,k, \vartheta)})_{\vartheta \in \Theta})$
and
$((\cU^{(\theta,\mathfrak{i},\mathfrak{k}, \vartheta)})_{\vartheta \in \Theta}, (W^{(\theta,\mathfrak{i},\mathfrak{k}, \vartheta)})_{\vartheta \in \Theta})$
are independent establish \cref{properties_approx:item4}.
In addition, observe that the fact that \cref{b37} implies that
for all $\theta\in\Theta $, $t\in[0,T]$, $x\in\R^d$ it holds that
$U^\theta_{0}(t,x) = 0$,
the hypothesis that $(W^\theta)_{\theta \in \Theta}$ are independent standard Brownian motions, 
the hypothesis that $(\fu^\theta)_{\theta \in \Theta}$ are i.i.d.\ random variables, 
\cref{properties_approx:item1,%
properties_approx:item2,%
properties_approx:item3,%
properties_approx:item4}, Hutzenthaler et al.\ \cite[Corollary 2.5]{HutzenthalerJentzenKruse2018}, and induction on $\N_0$ establish \cref{properties_approx:item5}.
The proof of \cref{properties_approx} is thus complete.
\end{proof}

\subsection{Integrability properties of MLP approximations}

\begin{lemma}\label{lem:integrable}
Assume \cref{setting1}.
Then 
it holds for all $n\in\N_0$, 
$\theta\in\Theta$, $s\in[0,T]$, $t\in[s,T]$, $x\in\R^d$ that
\begin{equation}\label{eq_3_6}
\begin{split}
& \E\br[\big]{ \abs[]{ \mlp_{n}^\theta(t,x+W^\theta_{t-s}) } } 
+ \E\br[\big]{ \abs[]{ \funcG(x + W^{\theta}_{t-s}) } } 
+ \E\br[\big]{ \abs[]{ \pr[]{ \funcF( \mlp_n^\theta) } \pr[]{ \cU_t^\theta, x + W^\theta_{\cU_t^\theta - t} } } } \\
& + \int_s^T \E\br[\big]{ \abs[]{ \mlp_{n}^\theta(r,x+W^\theta_{r-s}) } } \, dr 
+ \int_s^T \E\br[\big]{ \abs[]{ (\funcF(\mlp_{n}^\theta))(r,x+W_{r-s}^\theta) } } \, dr 
< \infty.
\end{split}
\end{equation}
\end{lemma}

\begin{proof}[Proof of \cref{lem:integrable}]
Throughout this proof let $x\in\R^d$ and assume without loss of generality that $ T \in (0,\infty)$.
Note that \cref{fn_cond}, the fact that for all $r, a, b\in[0,\infty)$ it holds that $(a + b)^r \le 2^{\max\{r-1,0\}}(a^r + b^r)$, and the fact that for all $\theta\in\Theta$ it holds that $\E[\norm{W_T^\theta}^p] < \infty$ assure that for all $s\in[0,T]$, $t\in[s,T]$, $\theta\in\Theta$ it holds that
\begin{equation}\label{g_exp_bd}
\E\br[\big]{ \abs[]{ \funcG(x + W_{t-s}^\theta) } } 
\le \E\br[\big]{ \boundFG \pr[]{ 1 + \norm{ x + W^\theta_{t-s} }^p } } 
\le \boundFG \br[\big]{ 1 + 2^{\max\{p-1,0\}} \pr[\big]{ \norm{x}^p + \E\br[\big]{ \norm{W_T^\theta}^p } } } 
< \infty.
\end{equation}
Next we claim that
for all $n\in\N_0$, $s\in [0,T]$, $t\in[s,T]$, $\theta\in\Theta$
it holds that
\begin{equation}\label{b38:eq1}
\begin{split}
& \E\br[\big]{ \abs[]{ \mlp_{n}^\theta(t,x+W^\theta_{t-s}) } } 
+ \E\br[\big]{ \abs[]{ \pr[]{ \funcF( \mlp_n^\theta) } \pr[]{ \cU_t^\theta, x + W^\theta_{\cU_t^\theta - t} } } } \\
& + \int_s^T \E\br[\big]{ \abs[]{ \mlp_{n}^\theta(r,x+W^\theta_{r-s}) } } \, dr 
+ \int_s^T \E\br[\big]{ \abs[]{ (\funcF(\mlp_{n}^\theta))(r,x+W_{r-s}^\theta) } } \, dr 
< \infty.
\end{split}
\end{equation}
We now prove \cref{b38:eq1} by induction on $n \in \N_0$.
For the base case $n=0$ note that  
the fact that \cref{b37} implies that for all $t\in[0,T]$, $\theta\in\Theta$ it holds that $\mlp^\theta_{0}(t,x) = 0$
ensures that for all $s\in[0,T]$, $t \in [s, T]$
it holds that 
\begin{align}\label{eq_2_10}
& \E\br[\big]{ \abs[]{ \mlp_{0}^\theta(t,x+W^\theta_{t-s}) } } 
+ \E\br[\big]{ \abs[]{ \pr[]{ \funcF( \mlp_0^\theta) } \pr[]{ \cU_t^\theta, x + W^\theta_{\cU_t^\theta - t} } } } \nonumber \\
& + \int_s^T \E\br[\big]{ \abs[]{ \mlp_{0}^\theta(r,x+W^\theta_{r-s}) } } \, dr 
+ \int_s^T \E\br[\big]{ \abs[]{ (\funcF(\mlp_{0}^\theta))(r,x+W_{r-s}^\theta) } } \, dr\\
& = \E\br[\big]{ \abs[]{ \pr[]{ \funcF( 0) } \pr[]{ \cU_t^\theta, x + W^\theta_{\cU_t^\theta - t} } } }
+ \int_s^T \E\br[\big]{ \abs[]{ \pr[]{ \funcF(0) } (r,x+W_{r-s}^\theta) } } \, dr. \nonumber
\end{align}
In addition, observe that \cref{fn_cond}, \cref{big_F}, 
and \cref{g_exp_bd}
guarantee that for all $s\in[0,T]$, $t \in [s, T]$, $\theta\in\Theta$
it holds that 
\begin{align}\label{f_0_bd}
& \E\br[\big]{ \abs[]{ \pr[]{ \funcF( 0) } \pr[]{ \cU_t^\theta, x + W^\theta_{\cU_t^\theta - t} } } }
+ \int_s^T \E\br[\big]{ \abs[]{ \pr[]{ \funcF(0) } (r,x+W_{r-s}^\theta) } } \, dr \nonumber \\
& \le \E \br[\big]{ \boundFG \pr[\big]{ 1 + \norm{ x + W^\theta_{\cU^\theta_t - t} }^p } } 
+ \int_s^T \E \br[\big]{ \boundFG \pr[\big]{ 1 + \norm{ x + W^\theta_{r-s} }^p } } \, dr  \\
& \le  \pr[]{ T + 1 } \sup_{r\in[0,T]}  \pr[\big]{ \E \br[\big]{ \boundFG \pr[\big]{ 1 + \norm{ x + W^\theta_{r} }^p } } } 
< \infty. \nonumber
\end{align}
Combining this and \cref{eq_2_10} establishes \cref{b38:eq1} in the base case $n=0$. 
For the induction step $\N_0 \ni (n-1) \induct n \in \N$ let $n\in \N$ and assume that
for all $k \in \{0,1,\dots,n-1\}$, $s\in[0,T]$, $t \in [s, T]$, $\theta\in\Theta$
it holds that 
\begin{align}\label{b38:eq2}
& \E\br[\big]{ \abs[]{ \mlp_{k}^\theta(t,x+W^\theta_{t-s}) } } 
+ \E\br[\big]{ \abs[]{ \pr[]{ \funcF( \mlp_k^\theta) } \pr[]{ \cU_t^\theta, x + W^\theta_{\cU_t^\theta - t} } } } \\
& + \int_s^T \E\br[\big]{ \abs[]{ \mlp_{k}^\theta(r,x+W^\theta_{r-s}) } } \, dr 
+ \int_s^T \E\br[\big]{ \abs[]{ (\funcF(\mlp_{k}^\theta))(r,x+W_{r-s}^\theta) } } \, dr 
< \infty. \nonumber
\end{align}
Observe that \cref{b37} and the triangle inequality demonstrate that
for all $s\in[0,T]$, $t \in [s, T]$, $\theta\in\Theta$ it holds that 
\begin{align}
\label{b38:eq3}
\E\br[\big]{ \abs[]{ \mlp_{n}^{\theta}(t,x+W_{t-s}^\theta) } } 
& \le \tfrac{\1_\N(n)}{\littleM^{n}}\left[ \SmallSum_{i=1}^{\littleM^n} \E\br[\big]{ \abs[]{ \funcG\pr[]{ x + W^\theta_{t-s}+W^{(\theta,0,-i)}_{T-t} } } } \right] \\
& \quad + \SmallSum_{i=0}^{n-1} \tfrac{(T-t)}{\littleM^{n-i}} \biggl[ \SmallSum_{k=1}^{\littleM^{n-i}} \E\br[\big]{ \abs[]{ \pr[]{ \funcF\pr[]{ \mlp_{i}^{(\theta,i,k)} } } \pr[]{ \cU_t^{(\theta,i,k)},x + W^\theta_{t-s}+W_{\cU_t^{(\theta,i,k)}-t}^{(\theta,i,k)} } } } \nonumber \\
& \qquad + \1_\N(i) \, \E\br[\big]{ \abs[]{ \pr[]{ \funcF \pr[]{ \mlp_{i-1}^{(\theta,-i,k)} } } \pr[]{ \cU_t^{(\theta,i,k)},x + W^\theta_{t-s}+W_{\cU_t^{(\theta,i,k)}-t}^{(\theta,i,k)} } } } \biggr]. \nonumber
\end{align}
Next note that \cref{g_exp_bd} and the fact that 
$(W^{\theta})_{\theta \in \Theta}$ are independent standard Brownian motions imply that
for all $s\in[0,T]$, $t \in [s, T]$, $\theta\in\Theta$, $i \in \Z$ it holds that 
\begin{equation}
\label{b38:eq4}
\E\br[\big]{ \abs[]{ \funcG\pr[]{ x+W^\theta_{t-s}+W^{(\theta,0,i)}_{T-t} } } }
= \E\br[\big]{ \abs[]{ \funcG\pr[]{ x+W^\theta_{(t-s) + (T-t)} } } }
= \E\br[\big]{ \abs[]{ \funcG(x + W_{T-s}^\theta) } } 
< \infty.
\end{equation}
Furthermore, observe that \cref{b38:eq2}, \cref{properties_approx}, the fact that $(W^\theta)_{\theta\in\Theta}$ are independent standard Brownian motions, the fact that $(\cU^\theta)_{\theta\in\Theta}$ are i.i.d.\ random variables, the hypothesis that $(W^\theta)_{\theta\in\Theta}$ and $(\cU^\theta)_{\theta\in\Theta}$ are independent, 
the hypothesis that for all $\theta\in\Theta$, $r\in(0,1)$ it holds that $\P(\fu^\theta \le r)=r$,
Hutzenthaler et al.\ \cite[Lemma 2.15]{HutzenthalerPricing2019}, and
Hutzenthaler et al.\ \cite[Lemma 3.7]{HutzenthalerPricing2019}
guarantee that
for all $i \in \{0,1,\dots,n-1\}$, $k\in\Z$, $s\in[0,T]$, $t \in [s, T]$, $\theta\in\Theta$ it holds that 
\begin{align}
& (T-t) \, \E\br[\big]{ \abs[]{ \pr[]{ \funcF \pr[]{ \mlp_{i}^{(\theta,i,k)} } } \pr[]{ \cU_t^{(\theta,i,k)},x+W^\theta_{t-s}+W_{\cU_t^{(\theta,i,k)}-t}^{(\theta,i,k)} } } } \nonumber \\
& = \int_t^T \E\br[\big]{ \abs[]{ \pr[]{ \funcF \pr[]{ \mlp_{i}^{(\theta,i,k)} } } \pr[]{ r ,x+W^\theta_{t-s}+W_{r-t}^{(\theta,i,k)} } } } \, dr \\
& = \int_t^T \E\br[\big]{ \abs[]{ \pr[]{ \funcF \pr[]{ \mlp_{i}^{\theta} } } \pr[]{ r ,x+W^\theta_{t-s}+W_{r-t}^{\theta} } } } \, dr
= \int_t^T \E\br[\big]{ \abs[]{ \pr[]{ \funcF \pr[]{ \mlp_{i}^{\theta} } } \pr[]{ r ,x+W^\theta_{r-s} } } } \, dr
< \infty.\nonumber
\end{align}
Combining this, 
\cref{b38:eq2}, \cref{b38:eq3}, and \cref{b38:eq4}
establishes that
for all 
$s\in[0,T]$, $t \in [s, T]$, $\theta\in\Theta$
it holds that 
\begin{align}\label{b38:eq6}
& \E\br[\big]{ \abs[]{ \mlp_{n}^{\theta}(t,x+W_{t-s}^\theta) } } 
\le \Biggl( \SmallSum_{i=0}^{n-1} \tfrac{1}{\littleM^{n-i}} \Biggl[ {\SmallSum_{k=1}^{\littleM^{n-i}}} 
\displaystyle\int_t^T \E\br[\big]{ \abs[]{ \pr[]{ \funcF \pr[]{ \mlp_{i}^{\theta} } } \pr[]{ r ,x+W^\theta_{r-s} } } } \, dr \nonumber \\
& + \1_\N(i) \int_t^T \E\br[\big]{ \abs[]{ \pr[]{ \funcF \pr[]{ \mlp_{i-1}^{\theta} } } \pr[]{ r ,x+W^\theta_{r-s} } } } \, dr \Biggr] \Biggr) + \frac{\1_\N(n)}{\littleM^{n}}\left[ \SmallSum_{i=1}^{\littleM^n} \E\br[\big]{ \abs[]{ \funcG\pr[]{ x + W^\theta_{T-s} } } } \right] \nonumber \\
& = \Biggl[ \SmallSum_{i=0}^{n-1} \displaystyle\int_t^T \E\br[\big]{ \abs[]{ \pr[]{ \funcF \pr[]{ \mlp_{i}^{\theta} } } \pr[]{ r ,x+W^\theta_{r-s} } } } \, dr \\
& \quad + \1_\N(i) \int_t^T \E\br[\big]{ \abs[]{ \pr[]{ \funcF \pr[]{ \mlp_{i-1}^{\theta} } } \pr[]{ r ,x+W^\theta_{r-s} } } } \, dr \Biggr] + \1_\N(n) \, \E\br[\big]{ \abs[]{ \funcG\pr[]{ x + W^\theta_{T-s} } } } \nonumber \\
& \le \1_\N(n) \, \E\br[\big]{ \abs[]{ \funcG\pr[]{ x + W^\theta_{T-s} } } }
+ 2 \left[ \SmallSum_{i=0}^{n-1} \displaystyle\int_t^T \E\br[\big]{ \abs[]{ \pr[]{ \funcF \pr[]{ \mlp_{i}^{\theta} } } \pr[]{ r ,x+W^\theta_{r-s} } } } \, dr \right]
< \infty. \nonumber
\end{align}
This implies that 
for all 
$s\in[0,T]$, $\theta\in\Theta$
it holds that 
\begin{align}
\label{b38:eq7}
& \int_s^T \E\br[\big]{ \abs[]{ \mlp_{n}^{\theta}(r,x+W_{r-s}^\theta) } } \, dr
\le (T-s) \sup_{r\in[s,T]} \E\br[\big]{ \abs[]{ \mlp_{n}^{\theta}(r,x+W_{r-s}^\theta) } } \\
& \le (T-s) \left( \1_\N(n) \, \E\br[\big]{ \abs[]{ \funcG\pr[]{ x + W^\theta_{T-s} } } }
+ 2 \left[ \SmallSum_{i=0}^{n-1} \displaystyle \int_s^T \E\br[\big]{ \abs[]{ \pr[]{ \funcF \pr[]{ \mlp_{i}^{\theta} } } \pr[]{ r ,x+W^\theta_{r-s} } } } \, dr \right] \right) 
< \infty. \nonumber
\end{align}
Combining this, the triangle inequality, \cref{fn_cond}, \cref{big_F}, and \cref{f_0_bd} proves that
for all 
$s\in[0,T]$, $\theta\in\Theta$
it holds that
\begin{align}
& \int_s^T \E\br[\big]{ \abs[]{ \pr[]{ \funcF\pr[]{ \mlp_{n}^{\theta} } } \pr[]{ r,x+W_{r-s}^\theta } } } \, dr \nonumber \\
& \le \int_s^T \E\br[\big]{ \abs[]{ \pr[]{ \funcF\pr[]{ \mlp_{n}^{\theta} } - \funcF(0) } \pr[]{ r,x+W_{r-s}^\theta } } } \, dr 
+ \int_s^T \E\br[\big]{ \abs[]{ \pr[]{ \funcF(0) } \pr[]{ r,x+W_{r-s}^\theta } } } \, dr \\
& \le \LipConstF \int_s^T \E\br[\big]{ \abs[]{ \mlp_n^\theta(r,x + W^\theta_{r-s}) } } \, dr
+ \int_s^T \E\br[\big]{ \abs[]{ \pr[]{ \funcF\pr[]{ 0 } } \pr[]{ r,x+W_{r-s}^\theta } } } \, dr 
< \infty. \nonumber
\end{align}
This, \cref{b38:eq6}, \cref{b38:eq7}, and induction prove \cref{b38:eq1}. 
Combining \cref{b38:eq1} with \cref{g_exp_bd} therefore establishes \cref{eq_3_6}.
The proof of \cref{lem:integrable} is thus complete.
\end{proof}

\subsection{Expectations of MLP approximations}

\begin{lemma}\label{lem:mlp_expectation}
Assume \cref{setting1}.
Then 
\begin{enumerate}[label=(\roman *)]
\item
\label{lem:mlp_expectation_item1}
it holds for all $n\in\N_0$, $\theta\in\Theta$, $t\in[0,T]$, $x\in\R^d$ that
$(\funcF(\mlp_{n}^{(\theta,n,k)}))(\cU_t^{(\theta,n,k)}, x + W_{\cU_t^{(\theta,n,k)}-t}^{(\theta,n,k)}) - \1_\N(n)(\funcF(\mlp_{n-1}^{(\theta,-n,k)}))(\cU_t^{(\theta,n,k)}, x + W_{\cU_t^{(\theta,n,k)}-t}^{(\theta,n,k)})$, $k\in\Z$,
are i.i.d.\ random variables,
\item
\label{lem:mlp_expectation_item2_a}
it holds for all $n\in\N_0$, $t\in[0,T]$, $x\in\R^d$ that
\begin{equation}\label{eq:3_20}
\begin{split}
& \E\br[\big]{ \abs[]{ \mlp_{n}^0(t,x) } }
+ \E\br[\big]{ \abs[]{ \funcG\pr[]{ x + W^{(0,0,-1)}_{T-t} } } } \\
& + \SmallSum_{i=0}^{n-1} \E\br[\big]{ \abs[]{ \pr[]{ \funcF (\mlp_{i}^{(0,i,1)} ) - \1_\N(i) \funcF(\mlp_{i-1}^{(0,-i,1)}) } \pr[]{ \cU_t^{(0,i,1)}, x + W_{\cU_t^{(0,i,1)}-t}^{(0,i,1)} } } }
< \infty,
\end{split}
\end{equation}
\item
\label{lem:mlp_expectation_item2}
it holds for all $n\in\N_0$, $t\in[0,T]$, $x\in\R^d$ that
\begin{equation}\label{eq:3_20a}
\begin{split}
& \E\br[\big]{ \mlp_{n}^0(t,x) } = \1_\N(n)\,\E\br[\big]{ \funcG\pr[\big]{ x + W^{(0,0,-1)}_{T-t} } } \\
& + (T-t)\br[\bigg]{ \SmallSum_{i=0}^{n-1} \E\br[\big]{ \pr[]{ \funcF (\mlp_{i}^{(0,i,1)} ) - \1_\N(i) \funcF(\mlp_{i-1}^{(0,-i,1)}) } \pr[]{ \cU_t^{(0,i,1)}, x + W_{\cU_t^{(0,i,1)}-t}^{(0,i,1)} } } },
\end{split}
\end{equation}
and
\item 
\label{lem:mlp_expectation_item3}
it holds for all $n\in\N_0$, $\theta\in\Theta$, $t\in[0,T]$, $x\in\R^d$ that
\begin{equation}\label{eq_3_7}
\begin{split}
& (T-t) \, \E\br[\big]{ \abs[]{ \pr[]{ \funcF( \mlp_n^\theta) - \1_\N(n) \funcF(\mlp_{n-1}^\theta) } \pr[]{ \cU_t^\theta, x + W^\theta_{\cU_t^\theta - t} } } } \\
& = \int_t^T \E\br[\big]{ \abs[]{ \pr[]{ \funcF( \mlp_n^\theta) - \1_\N(n) \funcF(\mlp_{n-1}^\theta) } \pr[]{ r , x + W^\theta_{r - t} } } } \, dr 
< \infty.
\end{split}
\end{equation}
\end{enumerate}
\end{lemma}

\begin{proof}[Proof of \cref{lem:mlp_expectation}]
Throughout this proof let $x\in\R^d$.
Observe that 
\cref{properties_approx},
the hypothesis that $(\cU^\theta)_{\theta\in\Theta}$ are i.i.d.\ random variables, the hypothesis that $(W^\theta)_{\theta\in\Theta}$ are independent standard Brownian motions, Hutzenthaler et al.\ \cite[Corollary 2.5]{HutzenthalerJentzenKruse2018}, and Hutzenthaler et al.\ \cite[Item (i) of Lemma 2.6]{HutzenthalerJentzenKruse2018} 
(applied for every $n\in\N_0$, $\theta\in\Theta$ with $d \with d$, $T \with T$, $(\Omega,\cF,\P) \with (\Omega,\cF,\P)$, $F \with \funcF$, $U \with \mlp_n^\theta$ in the notation of \cite[Item (i) of Lemma 2.6]{HutzenthalerJentzenKruse2018})
imply that for all $n\in\N_0$, $\theta\in\Theta$, $t\in[0,T]$ it holds that
\begin{equation}
\pr[\big]{ \funcF(\mlp_{n}^{(\theta,n,k)}) - \1_\N(n) \funcF(\mlp_{n-1}^{(\theta,-n,k)})} \lrSpace \pr[\big]{ \cU_t^{(\theta,n,k)}, x + W_{\cU_t^{(\theta,n,k)}-t}^{(\theta,n,k)} },\ k\in\Z,
\end{equation}
are i.i.d.\ random variables.
This establishes \cref{lem:mlp_expectation_item1}.
Next note that the triangle inequality, \cref{properties_approx}, 
\cref{lem:integrable}, \cref{big_F}, \cref{b37}, and the fact that $(W^\theta)_{\theta\in\Theta}$ are independent standard Brownian motions guarantee that for all $n \in \N_0$, $t \in [0,T]$ it holds that 
\begin{align}
& \E\br[\big]{ \abs[]{ \mlp_{n}^0(t,x) } }
+ \E\br[\big]{ \abs[]{ \funcG\pr[]{ x + W^{(0,0,-1)}_{T-t} } } } \nonumber \\
& + \smallsum_{i=0}^{n-1} \E\br[\big]{ \abs[]{ \pr[]{ \funcF (\mlp_{i}^{(0,i,1)} ) - \1_\N(i) \funcF(\mlp_{i-1}^{(0,-i,1)}) } \pr[]{ \cU_t^{(0,i,1)}, x + W_{\cU_t^{(0,i,1)}-t}^{(0,i,1)} } } } \nonumber \\
& \le \E\br[\big]{ \abs[]{ \mlp_{n}^0(t,x) } }
+ \E\br[\big]{ \abs[]{ \funcG\pr[]{ x + W^{(0,0,-1)}_{T-t} } } } + \smallsum_{i=0}^{n-1} \E\br[\big]{ \abs[]{ \pr[]{ \funcF (\mlp_{i}^{(0,i,1)} ) } \pr[]{ \cU_t^{(0,i,1)}, x + W_{\cU_t^{(0,i,1)}-t}^{(0,i,1)} } } } \nonumber \\
& \quad + \smallsum_{i=0}^{n-1} \1_\N(i) \, \E\br[\big]{ \abs[]{ \pr[]{ \funcF(\mlp_{i-1}^{(0,-i,1)}) } \pr[]{ \cU_t^{(0,i,1)}, x + W_{\cU_t^{(0,i,1)}-t}^{(0,i,1)} } } } \\
& \le \E\br[\big]{ \abs[]{ \mlp_{n}^0(t,x) } }
+ \E\br[\big]{ \abs[]{ \funcG\pr[]{ x + W^{0}_{T-t} } } } 
 + 2 \smallsum_{i=0}^{n-1} \E\br[\big]{ \abs[]{ \pr[]{ \funcF (\mlp_{i}^{0} ) } \pr[]{ \cU_t^{0}, x + W_{\cU_t^{0}-t}^{0} } } }
< \infty. \nonumber
\end{align}
This establishes \cref{lem:mlp_expectation_item2_a}.
Furthermore, observe that \cref{lem:mlp_expectation_item1},
\cref{lem:mlp_expectation_item2_a}, 
\cref{properties_approx}, 
\cref{big_F}, \cref{b37}, and the fact that $(W^\theta)_{\theta\in\Theta}$ are independent standard Brownian motions
ensure that for all $n\in\N_0$, $t\in[0,T]$ it holds that
\begin{align}\label{mlp_stab_0a_pre1}
& \E\br[\big]{ \mlp_{n}^0(t,x) } 
= \frac{\1_\N(n)}{\littleM^n} \br[\bigg]{ \SmallSum_{k=1}^{\littleM^n} \E\br[\big]{ \funcG\pr[]{ x + W_{T-t}^{(0,0,-k)} } } } \nonumber \\
& + \SmallSum_{i=0}^{n-1} \frac{(T-t)}{\littleM^{n-i}} \br[\bigg]{ \SmallSum_{k=1}^{\littleM^{n-i}} \E\br[\big]{ \pr[]{ \funcF (\mlp_{i}^{(0,i,k)} ) - \1_\N(i) \funcF (\mlp_{i-1}^{(0,-i,k)} ) } \pr[]{ \cU_t^{(0,i,k)}, x + W_{\cU_t^{(0,i,k)}-t}^{(0,i,k)} } } } \\
& = \1_\N(n)\,\E\br[\big]{ \funcG\pr[]{ x + W^{(0,0,-1)}_{T-t} } } \nonumber \\
& + (T-t)\br[\bigg]{ \SmallSum_{i=0}^{n-1} \E\br[\big]{ \pr[]{ \funcF (\mlp_{i}^{(0,i,1)} ) - \1_\N(i) \funcF(\mlp_{i-1}^{(0,-i,1)}) } \pr[]{ \cU_t^{(0,i,1)}, x + W_{\cU_t^{(0,i,1)}-t}^{(0,i,1)} } } }. \nonumber
\end{align}
This establishes \cref{lem:mlp_expectation_item2}.
In addition, observe that \cref{lem:mlp_expectation_item1},
\cref{lem:mlp_expectation_item2_a},
\cref{properties_approx}, the fact that $(\cU^\theta)_{\theta\in\Theta}$ are i.i.d.\ random variables,
the hypothesis that for all $\theta\in\Theta$, $r\in(0,1)$ it holds that $\P(\fu^\theta \le r)=r$, 
and
Hutzenthaler et al.\ \cite[Lemma 3.7]{HutzenthalerPricing2019}
demonstrate that for all $n\in\N_0$, $\theta \in \Theta$, $t\in[0,T]$ it holds that
\begin{equation}\label{f_exp}
\begin{split}
& (T-t) \, \E\br[\big]{ \abs[]{ \pr[]{ \funcF( \mlp_n^\theta) - \1_\N(n) \funcF(\mlp_{n-1}^\theta) } \pr[]{ \cU_t^\theta, x + W^\theta_{\cU_t^\theta - t} } } } \\
& = \int_t^T \E\br[\big]{ \abs[]{ \pr[]{ \funcF( \mlp_n^\theta) - \1_\N(n) \funcF(\mlp_{n-1}^\theta) } \pr[]{ r , x + W^\theta_{r - t} } } } \, dr 
< \infty.
\end{split}
\end{equation}
This establishes \cref{lem:mlp_expectation_item3}.
The proof of \cref{lem:mlp_expectation} is thus complete.
\end{proof}


\subsection{Monte Carlo approximations}\label{monte_carlo_section}

\begin{lemma}
\label{Lp_monte_carlo}
Let $p \in (2,\infty)$, $n \in \N$, let $ ( \Omega, \cF, \P ) $ be a probability space, 
and let 
$ X_i  \colon \Omega \to \R $, $ i \in \{1, 2, \dots, n\} $, be i.i.d.\ random variables with 
$\E[ \abs{ X_1 } ] < \infty$.
Then it holds that
\begin{equation}
\pr[\big]{ \E \br[\big]{ \abs{ \E[X_1] - \tfrac{1}{n} \pr[\big]{ \smallsum_{ i = 1 }^{ n } X_{ i } } }^p } }^{\!\nicefrac{1}{p}}
\le \br[\big]{ \tfrac{p-1}{n} }^{\nicefrac{1}{2}} \pr[\big]{ \E \br[\big]{ \abs{ X_1 - \E[ X_1 ] }^p } }^{\!\nicefrac{1}{p}}.
\end{equation}
\end{lemma}

\begin{proof}[Proof of \cref{Lp_monte_carlo}]
First, observe that  
the hypothesis that for all $i\in\{1,2,\allowbreak\dots,\allowbreak n\}$ it holds that 
$  X_i\colon \Omega\to \R $
are i.i.d.\ random variables assures that
\begin{equation}
\E\br[\big]{ \abs{ \E[X_1] - \tfrac{1}{n} \pr{ \smallsum_{ i = 1 }^{ n } X_{ i } } }^p }
= \E\br[\big]{ \abs{ \tfrac{1}{n} \pr{ \smallsum_{ i = 1 }^{ n } (\E[X_1] - X_{ i }) } }^p }
= n^{-p} \, \E\br[\big]{ \abs{ \smallsum_{ i = 1 }^{ n } (\E[X_i] - X_{ i }) }^p }.
\end{equation}
Combining this, the fact that for all $i\in\{1,2,\dots,n\}$ it holds that 
$  X_i\colon \Omega\to \R $
are i.i.d.\ random variables, and, e.g., \cite[Theorem 2.1]{rio2009moment} (applied with $p \with p$, $(S_i)_{i\in\{0,1,\dots,n\}} \with (\sum_{k=1}^i (\E[X_k] - X_k))_{i\in\{0,1,\dots,n\}}$, $(X_i)_{i \in\{1,2,\dots,n\}} \with (\E[X_i] - X_i)_{i\in\{1,2,\dots,n\}}$ in the notation of \cite[Theorem 2.1]{rio2009moment}) ensures that
\begin{equation}
\begin{split}
& \pr[\big]{ \E \br[\big]{ \abs{ \E[X_1] - \tfrac{1}{n} \pr[\big]{ \smallsum_{ i = 1 }^{ n } X_{ i } } }^p } }^{\!\nicefrac{2}{p}}
= \tfrac{1}{n^2} \pr[\big]{ \E\br[\big]{ \abs{ \smallsum_{ i = 1 }^{ n } (\E[X_i] - X_{ i }) }^p } }^{\!\nicefrac{2}{p}} \\
& \le \tfrac{(p-1)}{n^2} \br[\Big]{ \smallsum_{ i = 1 }^{ n } \pr[\big]{ \E\br[\big]{ \abs{ \E[X_i] - X_{ i } }^p } }^{\!\nicefrac{2}{p}} }
= \tfrac{(p-1)}{n^2} \br[\Big]{ n \pr[\big]{ \E\br[\big]{ \abs{ \E[X_1] - X_{ 1 } }^p } }^{\!\nicefrac{2}{p}} } \\
& = \tfrac{(p-1)}{n} \pr[\big]{ \E\br[\big]{ \abs{ \E[X_1] - X_{ 1 } }^p } }^{\!\nicefrac{2}{p}}.
\end{split}
\end{equation}
The proof of \cref{Lp_monte_carlo} is thus complete.
\end{proof}

\begin{corollary}\label{prop:mlp}
Let $p \in [2,\infty)$, $n \in \N$, 
let $ ( \Omega, \cF, \P ) $ be a probability space, 
and 
let $ X_i  \colon \Omega \to \R $,
$ i \in \{1, 2, \dots, n\} $,
be i.i.d.\ random variables with 
$\E\br[]{ \abs{ X_1 } } < \infty$. 
Then it holds that
\begin{equation}\label{eq:2_9}
\pr[\big]{\E\br[\big]{ \abs[]{ \E[X_1] - \tfrac{1}{n} \pr[]{ {\smallsum_{ i = 1 }^{ n }} X_{ i } } }^p } }^{\!\nicefrac{1}{p}}
\le \br[\big]{ \tfrac{p-1}{n} }^{\nicefrac{1}{2}} \pr[\big]{ \E\br[\big]{ \abs{ X_1 - \E[ X_1 ] }^p } }^{\!\nicefrac{1}{p}}.
\end{equation}
\end{corollary}

\begin{proof}[Proof of \cref{prop:mlp}]
Observe that, 
e.g., Grohs et al.\ \cite[Lemma 2.3]{grohs2018proof}
and \cref{Lp_monte_carlo} establish \cref{eq:2_9}. 
The proof of \cref{prop:mlp} is thus complete.
\end{proof}

\begin{corollary}\label{cor:exp_bd}
Let $p\in[2,\infty)$, $n\in\N$, let $(\Omega,\cF,\P)$ be a probability space, and let $X_i \colon \Omega \to \R$, $i\in\{1,2,\dots,n\}$, be i.i.d.\ random variables with $\E[\abs{X_1}] < \infty$. Then
\begin{equation}\label{prob_bd1}
\pr[\big]{\E\br[\big]{ \abs[]{ \E[X_1] - \tfrac{1}{n} \pr[]{ {\smallsum_{ i = 1 }^{ n }} X_{ i } } }^p } }^{\!\nicefrac{1}{p}}
\le \tfrac{\secondConstant{p}}{n^{\nicefrac{1}{2}}} \pr[\big]{ \E\br[\big]{ \abs{ X_1 }^p } }^{\!\nicefrac{1}{p}}
\end{equation}
(cf.\ \cref{rand_const}).
\end{corollary}

\begin{proof}[Proof of \cref{cor:exp_bd}]
Note that \cref{rand_const} and \cref{prop:mlp} demonstrate that \cref{prob_bd1} holds. The proof of \cref{cor:exp_bd} is thus complete.
\end{proof}

\subsection{Recursive error bounds for MLP approximations}\label{sec:error_rec}

\begin{lemma}\label{lem:mlp_variance}
Assume \cref{setting1}.
Then it holds for all 
$n\in\N_0$, 
$t\in[0,T]$, $x\in\R^d$ that
\begin{align}
& \pr[\Big]{ \E\br[\Big]{ \abs[\big]{ \mlp^0_{n}\pr[\big]{t,x+W^0_t} - \E\br[\big]{ \mlp_{n}^0\pr[\big]{ t,x+W^0_t } } }^\fq } }^{\!\!\nicefrac{1}{\fq}} \nonumber \\
& \le \frac{\1_\N(n)\fm}{\littleM^{\nicefrac{n}{2}}} \br[\Bigg]{ \pr[\Big]{ \E\br[\Big]{ \abs[\big]{ \funcG\pr[\big]{ x + W_{T}^{0} } }^\fq } }^{\!\!\nicefrac{1}{\fq}} 
+ (T-t)^{\nicefrac{(\fq-1)}{\fq}} \pr[\bigg]{ \int_t^T \E\br[\Big]{ \abs[\big]{ \smallF\pr[\big]{ s,x+W_s^0,0 } }^\fq } \, ds }^{\!\!\nicefrac{1}{\fq}} } \label{eq:4_20} \\
& + \sum_{i=0}^{n-1} \frac{\LipConstF (T-t)^{\nicefrac{(\fq-1)}{\fq}} \fm}{\littleM^{\nicefrac{(n-i)}{2}}} \br[\Bigg]{ \pr[\big]{ \1_{(0,n)}(i) + \1_{[0,n-1)}(i)\,\littleM^{\nicefrac{1}{2}} } \pr[\bigg]{ \int_t^T \E\br[\Big]{ \abs[\big]{ \pr[\big]{ \mlp_{i}^{0} - \smallU } \lrSpace \pr[\big]{ s, x + W_{s}^{0} } }^\fq } \, ds }^{\!\!\nicefrac{1}{\fq}} }. \nonumber
\end{align}
\end{lemma}

\begin{proof}[Proof of \cref{lem:mlp_variance}]
Throughout this proof 
let $\Gsymb{k}\colon [0,T]\times \R^d \times \Omega \to \R$, $k\in\Z$, satisfy for all $k\in\Z$, $t\in[0,T]$, $x\in\R^d$ that
\begin{equation}\label{G_symb}
\Gsymb{k}(t,x) = \funcG\pr[]{ x + W_{T-t}^{(0,0,-k)} }
\end{equation}
and let $\Fsymb{n}{i}{j}{k}\colon [0,T] \times \R^d \times \Omega \to \R$, $n,i,j,k\in\Z$, satisfy for all $n\in\N$, $i\in\{0,1,\dots,n-1\}$, $j,k\in\Z$, $t\in[0,T]$, $x\in\R^d$ that
\begin{equation}\label{F_symb}
\Fsymb{n}{i}{j}{k}(t,x) = \pr[]{ \funcF (\mlp_{i}^{(0,j,k)} ) - \1_\N(i) \funcF (\mlp_{i-1}^{(0,-j,k)} ) } \pr[]{ \cU_t^{(0,j,k)}, x + W_{\cU_t^{(0,j,k)}-t}^{(0,j,k)} }.
\end{equation}
Observe that the hypothesis that $(W^{\theta})_{\theta\in\Theta}$ are independent standard Brownian motions and the hypothesis that $\funcG \in C(\R^d,\R)$
assure that for all $t\in[0,T]$, $x\in\R^d$ it holds that $(\Gsymb{k}(t,x))_{k\in\Z}$ are i.i.d.\ random variables.
This and \cref{cor:exp_bd} (applied for every $n\in\N$, 
$t\in[0,T]$, $x\in\R^d$ with $p \with \fq$, $n \with \littleM^n$, $(X_k)_{k\in\{1,2,\dots,\littleM^n\}} \with (\Gsymb{k}(t,x))_{k\in\{1,2,\dots,\littleM^n\}}$ in the notation of \cref{cor:exp_bd}) ensure that for all $n\in\N_0$, 
$t\in[0,T]$, $x\in\R^d$
it holds that
\begin{equation}\label{mlp_g_bd1b_var}
\pr[\Big]{ \E\br[\Big]{ \abs[\big]{ \tfrac{1}{\littleM^n}\br[\big]{ \smallsum_{k=1}^{\littleM^n} \Gsymb{k}(t,x) } - \E\br[\big]{ \Gsymb{1}(t,x) } }^\fq } }^{\!\!\nicefrac{1}{\fq}} 
\le \tfrac{ \fm }{\littleM^{\nicefrac{n}{2}}} \pr[\big]{ \E\br[\big]{ \abs[]{ \Gsymb{1}(t,x) }^\fq } }^{\!\nicefrac{1}{\fq}}.
\end{equation}
Next note that \cref{lem:mlp_expectation_item1} of \cref{lem:mlp_expectation} and \cref{cor:exp_bd} (applied for every $n\in\N$, $i\in\{0,1,\allowbreak\dots,\allowbreak n-1\}$, 
$t\in[0,T]$, $x\in\R^d$ with $p \with \fq$, $n \with \littleM^{n-i}$, 
$(X_k)_{k\in\{1,2,\dots,\littleM^{n-i}\}} \with (\Fsymb{n}{i}{i}{k}(t,x))_{k\in\{1,2,\dots,\littleM^{n-i}\}}$
in the notation of \cref{cor:exp_bd}) demonstrate that for all $n\in\N$, $i\in\{0,1,\dots,n-1\}$, 
$t\in[0,T]$, $x\in\R^d$
it holds that
\begin{equation}\label{mlp_f_bd1b_var}
\pr[\Big]{ \E\br[\Big]{ \abs[\big]{ \tfrac{ 1 }{\littleM^{n-i}}\br[\big]{ \smallsum_{k=1}^{\littleM^{n-i}} \Fsymb{n}{i}{i}{k}(t,x) } - \E\br[\big]{ \Fsymb{n}{i}{i}{1}(t,x) } }^\fq } }^{\!\!\nicefrac{1}{\fq}} 
\le \tfrac{ \fm }{\littleM^{\nicefrac{(n-i)}{2}}} \pr[\big]{ \E\br[\big]{ \abs[]{ \Fsymb{n}{i}{i}{1}(t,x) }^\fq } }^{\!\nicefrac{1}{\fq}}.
\end{equation}
Combining this, \cref{b37}, \cref{G_symb}, \cref{F_symb}, \cref{mlp_g_bd1b_var}, \cref{lem:mlp_expectation_item2} of \cref{lem:mlp_expectation}, and the triangle inequality implies that for all $n\in\N_0$, 
$t\in[0,T]$, $x\in\R^d$ it holds that
\begin{align}\label{mlp_stab_1_var}
& \pr[\Big]{ \E\br[\Big]{ \abs[\big]{ \mlp^0_{n}(t,x) - \E\br[\big]{ \mlp_{n}^0(t,x) } }^\fq } }^{\!\!\nicefrac{1}{\fq}} \nonumber \\
& = \Bigl( \E\Bigl[ \bigl\lvert \pr[\big]{ \tfrac{\1_\N(n)}{\littleM^n}\br[\big]{ \smallsum_{k=1}^{\littleM^n} \Gsymb{k}(t,x) } 
+ \smallsum_{i=0}^{n-1} \tfrac{(T-t)}{\littleM^{n-i}}\br[\big]{ \smallsum_{k=1}^{\littleM^{n-i}} \Fsymb{n}{i}{i}{k}(t,x) } } \nonumber \\
& \quad - \pr[\big]{ \1_\N(n)\,\E\br[\big]{ \Gsymb{1}(t,x) } + \smallsum_{i=0}^{n-1} (T-t)\,\E\br[\big]{ \Fsymb{n}{i}{i}{1}(t,x) } } \bigr\rvert^\fq \Bigr] \Bigr)^{\!\!\nicefrac{1}{\fq}} \nonumber \\
& \le \1_\N(n) \pr[\Big]{ \E\br[\Big]{ \abs[\big]{ \tfrac{1}{\littleM^n}\br[\big]{ \smallsum_{k=1}^{\littleM^n} \Gsymb{k}(t,x) } - \E\br[\big]{ \Gsymb{1}(t,x) } }^\fq } }^{\!\!\nicefrac{1}{\fq}} \\
& \quad + \smallsum_{i=0}^{n-1} (T-t) \pr[\Big]{ \E\br[\Big]{ \abs[\big]{ \tfrac{1}{\littleM^{n-i}}\br[\big]{ \smallsum_{k=1}^{\littleM^{n-i}} \Fsymb{n}{i}{i}{k}(t,x) } - \E\br[\big]{ \Fsymb{n}{i}{i}{1}(t,x) } }^\fq } }^{\!\!\nicefrac{1}{\fq}} \nonumber \\
& \le \tfrac{\1_\N(n) \fm}{\littleM^{\nicefrac{n}{2}}}\pr[\big]{ \E\br[\big]{ \abs[]{ \Gsymb{1}(t,x) }^\fq } }^{\!\!\nicefrac{1}{\fq}} 
+ \SmallSum_{i=0}^{n-1} \tfrac{(T-t) \fm}{\littleM^{\nicefrac{(n-i)}{2}}} \pr[\big]{ \E\br[\big]{ \abs[]{ \Fsymb{n}{i}{i}{1}(t,x) }^\fq } }^{\!\!\nicefrac{1}{\fq}}. \nonumber
\end{align}
Moreover, observe that \cref{F_symb} and
\cref{lem:mlp_expectation_item1,%
lem:mlp_expectation_item3} of 
\cref{lem:mlp_expectation} assure that for all $n\in\N_0$, $t\in[0,T]$, $x\in\R^d$ it holds that
\begin{equation}\label{eq:4_13new}
\begin{split}
& \SmallSum_{i=0}^{n-1}\tfrac{(T-t)\fm }{\littleM^{\nicefrac{(n-i)}{2}}} \pr[\big]{ \E\br[\big]{ \abs[]{ \Fsymb{n}{i}{i}{1}(t,x) }^\fq } }^{\!\nicefrac{1}{\fq}}
= \SmallSum_{i=0}^{n-1}\tfrac{(T-t)^{\nicefrac{(\fq-1)}{\fq}} \fm }{\littleM^{\nicefrac{(n-i)}{2}}} \pr[\big]{ (T-t) \, \E\br[\big]{ \abs[]{ \Fsymb{n}{i}{i}{1}(t,x) }^\fq } }^{\!\nicefrac{1}{\fq}} \\
& = \SmallSum_{i=0}^{n-1}\tfrac{(T-t)^{\nicefrac{(\fq-1)}{\fq}} \fm }{\littleM^{\nicefrac{(n-i)}{2}}} \pr[\bigg]{ \displaystyle \int_t^T \E\br[\Big]{ \abs[\big]{ \pr[\big]{ \funcF
 (\mlp_{i}^{(0,i,1)} ) - \1_\N(i) \funcF(\mlp_{i-1}^{0,-i,1)} ) } \lrSpace  \pr[\big]{ s , x + W_{s-t}^{(0,i,1)} } }^\fq } \, ds }^{\!\!\nicefrac{1}{\fq}}.
\end{split}
\end{equation}
Furthermore, note that \cref{big_F}, \cref{b37}, 
and the triangle inequality guarantee that for all $n\in\N_0$, $t\in[0,T]$, $x\in\R^d$ it holds that
\begin{align}
& \sum_{i=0}^{n-1}\tfrac{(T-t)^{\nicefrac{(\fq-1)}{\fq}} \fm }{\littleM^{\nicefrac{(n-i)}{2}}} \pr[\bigg]{ \int_t^T \E\br[\Big]{ \abs[\big]{ \pr[\big]{ \funcF
 (\mlp_{i}^{(0,i,1)} ) - \1_\N(i) \funcF(\mlp_{i-1}^{0,-i,1)} ) } \lrSpace  \pr[\big]{ s , x + W_{s-t}^{(0,i,1)} } }^\fq } \, ds }^{\!\!\nicefrac{1}{\fq}} \nonumber \\
& \le \tfrac{\1_\N(n) (T-t)^{\nicefrac{(\fq-1)}{\fq}} \fm}{\littleM^{\nicefrac{n}{2}}} \pr[\bigg]{ \int_t^T \E\br[\Big]{ \abs[\big]{ \smallF(s,x+W_{s-t}^{(0,0,1)},0) }^\fq } \, ds }^{\!\!\nicefrac{1}{\fq}} \\
& \quad + \sum_{i=1}^{n-1}\tfrac{(T-t)^{\nicefrac{(\fq-1)}{\fq}} \fm}{\littleM^{\nicefrac{(n-i)}{2}}} \Biggl[ \pr[\bigg]{ \int_t^T \E\br[\Big]{ \abs[\big]{ \pr[\big]{ \funcF
 (\mlp_{i}^{(0,i,1)} ) - \funcF(\smallU) } \lrSpace  \pr[\big]{ s , x + W_{s-t}^{(0,i,1)} } }^\fq } \, ds }^{\!\!\nicefrac{1}{\fq}} \nonumber \\
& \quad + \pr[\bigg]{ \int_t^T \E\br[\Big]{ \abs[\big]{ \pr[\big]{ \funcF(\smallU) - \funcF (\mlp_{i-1}^{(0,-i,1)} ) } \lrSpace \pr[\big]{ s, x + W_{s-t}^{(0,i,1)} } }^\fq } \, ds }^{\!\!\nicefrac{1}{\fq}} \Biggr]. \nonumber
\end{align}
Combining this, \cref{properties_approx},
\cref{fn_cond}, \cref{big_F}, \cref{G_symb}, \cref{mlp_stab_1_var}, \cref{eq:4_13new}, and the fact that $(W^\theta)_{\theta\in\Theta}$ are independent standard Brownian motions
demonstrates that for all $n\in\N_0$, 
$t\in[0,T]$, $x\in\R^d$ it holds that
\begin{align}
& \pr[\Big]{ \E\br[\Big]{ \abs[\big]{ \mlp^0_{n}(t,x) - \E\br[\big]{ \mlp_{n}^0(t,x) } }^\fq } }^{\!\!\nicefrac{1}{\fq}} 
\le \tfrac{\1_\N(n)\fm}{\littleM^{\nicefrac{n}{2}}} \pr[\big]{ \E\br[\big]{ \abs[]{ \Gsymb{1}(t,x) }^\fq } }^{\!\!\nicefrac{1}{\fq}} 
+ \SmallSum_{i=0}^{n-1}\tfrac{(T-t)\fm }{\littleM^{\nicefrac{(n-i)}{2}}} \pr[\big]{ \E\br[\big]{ \abs[]{ \Fsymb{n}{i}{i}{1}(t,x) }^\fq } }^{\!\nicefrac{1}{\fq}} \nonumber \\
& \le \tfrac{\1_\N(n)\fm}{\littleM^{\nicefrac{n}{2}}}\br[\Bigg]{ \pr[\Big]{ \E\br[\Big]{ \abs[\big]{ \funcG\pr[\big]{ x + W_{T-t}^{1} } }^\fq } }^{\!\!\nicefrac{1}{\fq}} 
+ (T-t)^{\nicefrac{(\fq-1)}{\fq}} \pr[\bigg]{ \int_t^T \E\br[\Big]{ \abs[\big]{ \smallF(s,x+W_{s-t}^{(0,0,1)},0) }^\fq } \, ds }^{\!\!\nicefrac{1}{\fq}} } \nonumber \\
& \quad + \sum_{i=1}^{n-1}\tfrac{\LipConstF (T-t)^{\nicefrac{(\fq-1)}{\fq}} \fm}{\littleM^{\nicefrac{(n-i)}{2}}}\biggl[ \pr[\bigg]{ \int_t^T \E\br[\Big]{ \abs[\big]{ \pr[\big]{ \mlp_{i}^{(0,i,1)} - \smallU } \lrSpace  \pr[\big]{ s , x + W_{s-t}^{(0,i,1)} } }^\fq } \, ds }^{\!\!\nicefrac{1}{\fq}} \nonumber \\
& \quad + \pr[\bigg]{ \int_t^T \E\br[\Big]{ \abs[\big]{ \pr[\big]{ \smallU - \mlp_{i-1}^{(0,-i,1)} } \lrSpace \pr[\big]{ s, x + W_{s-t}^{(0,i,1)} } }^\fq } \, ds }^{\!\!\nicefrac{1}{\fq}} \biggr] \\
& = \tfrac{\1_\N(n)\fm}{\littleM^{\nicefrac{n}{2}}}\br[\Bigg]{ \pr[\Big]{ \E\br[\Big]{ \abs[\big]{ \funcG\pr[\big]{ x + W_{T-t}^{0} } }^\fq } }^{\!\!\nicefrac{1}{\fq}} 
+ (T-t)^{\nicefrac{(\fq-1)}{\fq}} \pr[\bigg]{ \int_t^T \E\br[\Big]{ \abs[\big]{ \smallF(s,x+W_{s-t}^0,0) }^\fq } \, ds }^{\!\!\nicefrac{1}{\fq}} } \nonumber \\
& \quad + \sum_{i=0}^{n-1} \tfrac{\LipConstF (T-t)^{\nicefrac{(\fq-1)}{\fq}} \fm }{\littleM^{\nicefrac{(n-i)}{2}}} \br[\Bigg]{ \pr[\big]{ \1_{(0,n)}(i) + \1_{[0,n-1)}(i)\,\littleM^{\nicefrac{1}{2}} } \pr[\bigg]{ \int_t^T \E\br[\Big]{ \abs[\big]{ \pr[\big]{ \mlp_{i}^{0} - \smallU } \lrSpace \pr[\big]{ s, x + W_{s-t}^{0} } }^\fq } \, ds }^{\!\!\nicefrac{1}{\fq}} }. \nonumber
\end{align}
This and the fact that $W^0$ has independent 
increments ensure that for all $n\in\N_0$, $t\in[0,T]$, $x\in\R^d$ it holds that
\begin{align}\label{mlp_stab_4_var}
& \pr[\Big]{ \E\br[\Big]{ \abs[\big]{ \mlp^0_{n}(t,x+W^0_t) - \E\br[\big]{ \mlp_{n}^0(t,x+W^0_t) } }^\fq } }^{\!\!\nicefrac{1}{\fq}} \nonumber \\
& \le \tfrac{\1_\N(n)\fm}{\littleM^{\nicefrac{n}{2}}}\br[\Bigg]{ \pr[\Big]{ \E\br[\Big]{ \abs[\big]{ \funcG\pr[\big]{ x + W_{T}^{0} } }^\fq } }^{\!\!\nicefrac{1}{\fq}} 
+ (T-t)^{\nicefrac{(\fq-1)}{\fq}} \pr[\bigg]{ \int_t^T \E\br[\Big]{ \abs[\big]{ \smallF(s,x+W_{s}^0,0) }^\fq } \, ds }^{\!\!\nicefrac{1}{\fq}} }\\
& \quad + \sum_{i=0}^{n-1} \tfrac{\LipConstF (T-t)^{\nicefrac{(\fq-1)}{\fq}} \fm }{\littleM^{\nicefrac{(n-i)}{2}}} \br[\Bigg]{ \pr[\big]{ \1_{(0,n)}(i) + \1_{[0,n-1)}(i)\,\littleM^{\nicefrac{1}{2}} } \pr[\bigg]{ \int_t^T \E\br[\Big]{ \abs[\big]{ \pr[\big]{ \mlp_{i}^{0} - \smallU } \lrSpace \pr[\big]{ s, x + W_{s-t}^{0} } }^\fq } \, ds }^{\!\!\nicefrac{1}{\fq}} }. \nonumber
\end{align}
The proof of \cref{lem:mlp_variance} is thus complete.
\end{proof}

\begin{lemma}\label{mlp_stab_pre}
Assume \cref{setting1}.
Then it holds for all 
$n\in\N_0$, $t\in[0,T]$, $x\in\R^d$ that
\begin{align}
& \pr[\Big]{ \E\br[\Big]{ \abs[\big]{ \mlp^0_{n}(t,x+W_t^0) - \smallU(t,x+W_t^0) }^\fq } }^{\!\!\nicefrac{1}{\fq}} \\
& \le \frac{\fm \exp(\LipConstF(T-t)) }{\littleM^{\nicefrac{n}{2}}}\Biggl[ \pr[\Big]{ \E\br[\Big]{ \abs[\big]{ \funcG(x+W_{T}^0) }^\fq } }^{\!\nicefrac{1}{\fq}} 
+ (T-t)^{\nicefrac{(\fq-1)}{\fq}} \pr[\bigg]{  \int_t^T \E\br[\Big]{ \abs[\big]{ \smallF(s,x+W_{s}^0,0) }^\fq } \, ds }^{\!\!\nicefrac{1}{\fq}} \Biggr] \nonumber \\
& \quad + \sum_{i=0}^{n-1} \frac{\LipConstF (T-t)^{\nicefrac{(\fq-1)}{\fq}} \fm }{\littleM^{\nicefrac{(n-i)}{2}}} \br[\Bigg]{ \pr[\big]{ \1_{(0,n)}(i) + \littleM^{\nicefrac{1}{2}} } \pr[\bigg]{ \int_t^T \E\br[\Big]{ \abs[\big]{ \pr[]{ \mlp_{i}^0 - \smallU } \pr[]{ s,x+W_{s}^0} }^\fq } \, ds }^{\!\!\nicefrac{1}{\fq}} }. \nonumber
\end{align}
\end{lemma}

\begin{proof}[Proof of \cref{mlp_stab_pre}]
First, observe that  
\cref{lem:integrable}, \cref{cor:sol_bd1},
and the triangle inequality ensure that for all $n\in\N_0$, $t\in[0,T]$, $x\in\R^d$ it holds that
\begin{equation}\label{error_triangle}
\begin{split}
& \pr[\big]{ \E\br[\big]{ \abs[]{ \mlp^0_{n}(t,x) - \smallU(t,x) }^\fq } }^{\!\nicefrac{1}{\fq}} \\
& \le \pr[\big]{ \E\br[\big]{ \abs[]{ \mlp^0_{n}(t,x) - \E\br[]{ \mlp_{n}^0(t,x) } }^\fq } }^{\!\nicefrac{1}{\fq}} 
+ \pr[\big]{ \E\br[\big]{ \abs[]{ \E\br[]{ \mlp^0_{n}(t,x) } - \smallU(t,x) }^\fq } }^{\!\nicefrac{1}{\fq}}.
\end{split}
\end{equation}
Next note that 
\cref{lem:mlp_expectation_item2_a,%
lem:mlp_expectation_item2,%
lem:mlp_expectation_item3} of  \cref{lem:mlp_expectation}, the fact that $(W^\theta)_{\theta\in\Theta}$ are independent standard Brownian motions, and \cref{b37} demonstrate that for all $n\in\N_0$, $t\in[0,T]$, $x\in\R^d$ it holds that
\begin{align}\label{eq:3_24}
& \E\br[\big]{ \mlp_{n}^0(t,x) } 
= \1_\N(n) \, \E\br[\big]{ \funcG(x+W_{T-t}^{(0,0,-1)}) } \nonumber \\
& + (T-t) \br[\bigg]{ \SmallSum_{i=0}^{n-1} \E\br[\big]{ \pr[]{ \funcF (\mlp_{i}^{(0,i,1)} ) - \1_\N(i) \funcF(\mlp_{i-1}^{(0,-i,1)}) } \pr[]{ \cU_t^{(0,i,1)}, x + W_{\cU_t^{(0,i,1)}-t}^{(0,i,1)} } } } \\
& = \1_\N(n) \, \E\br[\big]{ \funcG(x+W_{T-t}^{0}) } 
 + \br*{ \SmallSum_{i=0}^{n-1} \displaystyle \int_t^T \E\br[\big]{ \pr[]{ \funcF (\mlp_{i}^{(0,i,1)} ) - \1_\N(i) \funcF(\mlp_{i-1}^{(0,-i,1)}) } \pr[]{ s, x + W_{s-t}^{(0,i,1)} } } \, ds }. \nonumber 
\end{align}
In addition, observe that \cref{big_F}, the fact that $(W^\theta)_{\theta\in\Theta}$ are independent standard Brownian motions, the fact that $(\cU^\theta)_{\theta\in\Theta}$ are i.i.d.\ random variables, \cref{properties_approx:item3,%
properties_approx:item5} of \cref{properties_approx}, and \cite[Lemma 2.2]{HutzenthalerJentzenKruse2018} prove that for all $i\in\N_0$, $t\in[0,T]$, $s\in[t,T]$, $x\in\R^d$ it holds that
\begin{equation}
\begin{split}
& \E\br[\big]{ \pr[]{ \funcF (\mlp_{i}^{(0,i,1)} ) - \1_\N(i) \funcF(\mlp_{i-1}^{(0,-i,1)}) } \pr[]{ s, x + W_{s-t}^{(0,i,1)} } } \\
& = \E\br[\big]{ \pr[]{ \funcF (\mlp_{i}^{(0,i,1)} ) } \pr[]{ s, x + W_{s-t}^{(0,i,1)} } } - \1_\N(i) \, \E\br[\big]{ \pr[]{ \funcF(\mlp_{i-1}^{(0,-i,1)}) } \pr[]{ s, x + W_{s-t}^{(0,i,1)} } } \\
& = \E\br[\big]{ \pr[]{ \funcF (\mlp_{i}^{0} ) } \pr[]{ s, x + W_{s-t}^{0} } } - \1_\N(i) \, \E\br[\big]{ \pr[]{ \funcF(\mlp_{i-1}^{0}) } \pr[]{ s, x + W_{s-t}^{0} } }.
\end{split}
\end{equation}
Combining this, \cref{lem:integrable}, and \cref{eq:3_24} yields that for all $n\in\N_0$, $t\in[0,T]$, $x\in\R^d$ it holds that
\begin{equation}
\begin{split}
& \E\br[\big]{ \mlp_{n}^0(t,x) } 
= \1_\N(n) \, \E\br[\big]{ \funcG(x+W_{T-t}^{0}) } \\
& + \br[\Bigg]{ \sum_{i=0}^{n-1} \int_t^T \pr[\Big]{ \E\br[\big]{ \pr[]{ \funcF (\mlp_{i}^{0} ) } \pr[]{ s, x + W_{s-t}^{0} } } - \1_\N(i) \, \E\br[\big]{ \pr[]{ \funcF(\mlp_{i-1}^{0}) } \pr[]{ s, x + W_{s-t}^{0} } } } \, ds } \\
& = \1_\N(n) \br[\bigg]{ \E\br[\big]{ \funcG(x+W_{T-t}^{0}) } + \int_t^T \E\br[\big]{ \pr[]{ \funcF (\mlp_{n-1}^{0} ) } \pr[]{ s, x + W_{s-t}^{0} } } \, ds }.
\end{split}
\end{equation}
This and \cref{assumed_u} show that for all $n\in\N_0$, $t\in[0,T]$, $x\in\R^d$ it holds that
\begin{equation}
\smallU(t,x) - \E\br[\big]{ \mlp_{n}^0(t,x) }
= \begin{cases}
u(t,x) & \colon n = 0 \\[0.1em]
\int_t^T \E\br[\big]{ \pr[]{ \funcF(\smallU) - \funcF \pr[]{ \mlp_{n-1}^0 }} \pr[]{ s,x+W_{s-t}^{0} } } \, ds & \colon n \in \N
\end{cases}.
\end{equation}
This, \cref{fn_cond}, \cref{big_F}, \cref{cor:sol_bd1}, the triangle inequality, Jensen's inequality, Fubini's theorem, and the fact that $W^0$ has independent increments assure that for all $n\in \N_0$, $t\in[0,T]$, $x\in\R^d$ it holds that
\begin{align}\label{4_23}
& \pr[\big]{ \E\br[\big]{ \abs[]{ \E\br[]{ \mlp_{n}^0(t,x+W_t^0) } - \smallU(t,x+W_t^0) }^\fq } }^{\!\nicefrac{1}{\fq}} \\
& \le \1_{\{0\}}(n) \pr[\big]{ \E\br[\big]{ \abs{ \smallU(t,x+W_t^0) }^\fq } }^{\!\nicefrac{1}{\fq}} 
+ \1_\N(n) \pr[\Bigg]{ \E\br[\Bigg]{ \abs[\bigg]{ \int_t^T \E\br[\big]{ \pr[]{ \funcF(\smallU) - \funcF \pr[]{ \mlp_{n-1}^0 }} \pr[]{ s,x+W_{s}^{0} } } \, ds }^\fq } }^{\!\!\nicefrac{1}{\fq}} \nonumber \\
& \le \1_{\{0\}}(n) \pr[\big]{ \E\br[\big]{ \abs{ \smallU(t,x+W_t^0) }^\fq } }^{\!\nicefrac{1}{\fq}} 
+ \1_\N(n) \int_t^T \pr[\big]{ \E\br[\big]{ \abs{ \pr[]{ \funcF(\smallU) - \funcF \pr[]{ \mlp_{n-1}^0 }} \pr[]{ s,x+W_{s}^{0} } }^\fq } }^{\!\nicefrac{1}{\fq}} \, ds \nonumber \\
& \le \1_{\{0\}}(n) \pr[\big]{ \E\br[\big]{ \abs{ \smallU(t,x+W_t^0) }^\fq } }^{\!\nicefrac{1}{\fq}} 
+ \1_\N(n) \, \LipConstF \int_t^T \pr[\big]{ \E\br[\big]{ \abs{ \pr[]{ \smallU - \mlp_{n-1}^0 } \pr[]{ s,x+W_{s}^{0} } }^\fq } }^{\!\nicefrac{1}{\fq}} \, ds . \nonumber
\end{align}
Next observe that H{\"o}lder's inequality ensures that for all $n\in\N$, $t\in[0,T]$, $x\in\R^d$ it holds that
\begin{align}
& \int_t^T \pr[\big]{ \E\br[\big]{ \abs{ \pr[]{ \smallU - \mlp_{n-1}^0 } \pr[]{ s,x+W_{s}^{0} } }^\fq } }^{\!\nicefrac{1}{\fq}} \, ds
= \pr[\Bigg]{ \br[\bigg]{ \int_t^T \pr[\big]{ \E\br[\big]{ \abs{ \pr[]{ \smallU - \mlp_{n-1}^0 } \pr[]{ s,x+W_{s}^{0} } }^\fq } }^{\!\nicefrac{1}{\fq}} \, ds }^\fq }^{\!\!\nicefrac{1}{\fq}} \nonumber \\
& \le \pr[\bigg]{ (T-t)^{\fq-1} \int_t^T \E\br[\big]{ \abs{ \pr[]{ \smallU - \mlp_{n-1}^0 } \pr[]{ s,x+W_{s}^{0} } }^\fq } \, ds }^{\!\!\nicefrac{1}{\fq}} \\
& = (T-t)^{\nicefrac{(\fq-1)}{\fq}} \pr[\bigg]{ \int_t^T \E\br[\big]{ \abs{ \pr[]{ \smallU - \mlp_{n-1}^0 } \pr[]{ s,x+W_{s}^{0} } }^\fq } \, ds }^{\!\!\nicefrac{1}{\fq}}. \nonumber
\end{align}
Combining this, \cref{lemma:2.2}, and \cref{4_23} demonstrates that for all $n\in\N_0$, $t\in[0,T]$, $x\in\R^d$ it holds that
\begin{align}
& \pr[\Big]{ \E\br[\Big]{ \abs[\big]{ \E\br[\big]{ \mlp_{n}^0(t,x+W_t^0) } - \smallU(t,x+W_t^0) }^\fq } }^{\!\!\nicefrac{1}{\fq}} \nonumber \\
& \le \1_{\{0\}}(n) \exp\pr[\big]{L(T-t)} \biggl[ \pr[\big]{ \E\br[\big]{\abs[]{ \funcG(x+\fwpr_{T}) }^\fq } }^{\!\nicefrac{1}{\fq}} + (T-t)^{\nicefrac{(\fq-1)}{\fq}} \pr[\Big]{ \textstyle\int_{t}^T \E\br[\big]{\abs[]{ \smallF(s,x+\fwpr_{s},0)}^\fq} \, ds }^{\!\!\nicefrac{1}{\fq}} \biggr] \nonumber \\
& \quad + \1_\N(n) \, \LipConstF (T-t)^{\nicefrac{(\fq-1)}{\fq}} \pr[\Big]{ \textstyle\int_t^T \E\br[\big]{ \abs{ \pr[]{ \smallU - \mlp_{n-1}^0 } \pr[]{ s,x+W_{s}^{0} } }^\fq } \, ds }^{\!\!\nicefrac{1}{\fq}}.
\end{align}
This, \cref{lem:mlp_variance}, \cref{error_triangle}, and the fact that $\fm \in [1,\infty)$ guarantee that for all $n\in\N_0$, $t\in[0,T]$, $x\in\R^d$ it holds that
\begin{align}\label{mlp_stab_5}
& \pr[\big]{ \E\br[\big]{ \abs[]{ \mlp^0_{n}(t,x+W_t^0) - \smallU(t,x+W_t^0) }^\fq } }^{\!\nicefrac{1}{\fq}} \\
& \le \frac{\fm \exp(\LipConstF(T-t)) }{\littleM^{\nicefrac{n}{2}}}\Biggl[ \pr[\big]{ \E\br[\big]{ \abs[]{ \funcG(x+W_{T}^0) }^\fq } }^{\!\nicefrac{1}{\fq}} 
+ (T-t)^{\nicefrac{(\fq-1)}{\fq}} \pr[\bigg]{  \int_t^T \E\br[\big]{ \abs{ \smallF(s,x+W_{s}^0,0) }^\fq } \, ds }^{\!\!\nicefrac{1}{\fq}} \Biggr] \nonumber \\
& \quad + \sum_{i=0}^{n-1} \frac{\LipConstF (T-t)^{\nicefrac{(\fq-1)}{\fq}} \fm }{\littleM^{\nicefrac{(n-i)}{2}}} \br[\Bigg]{ \pr[\big]{ \1_{(0,n)}(i) + \littleM^{\nicefrac{1}{2}} } \pr[\bigg]{ \int_t^T \E\br[\big]{ \abs[]{ \pr[]{ \mlp_{i-1}^0 - \smallU } \pr[]{ s,x+W_{s}^0} }^\fq } \, ds }^{\!\!\nicefrac{1}{\fq}} }. \nonumber
\end{align}
The proof of \cref{mlp_stab_pre} is thus complete.
\end{proof}

\begin{lemma}
\label{prop:factorial}
It holds for all $ n \in \N$ that
\begin{equation}\label{prop:factorial1}
\left[\frac{n}{3}\right]^n \le \left[\frac{n}{e}\right]^n < e \left[\frac{n}{e}\right]^n < \frac{e}{2^{\nicefrac{1}{2}}} \left[\frac{n}{e}\right]^{n + \nicefrac{1}{2}} \le n! \le e \left[\frac{n}{e}\right]^{n + \nicefrac{1}{2}} < e\left[\frac{n+1}{e}\right]^{n+1}
\end{equation}
and
\begin{equation}
n^n \le 2^{-\nicefrac{1}{2}} e^{\nicefrac{1}{2}} n^{n + \nicefrac{1}{2}} \le (n!)e^n \le e^{\nicefrac{1}{2}} n^{n + \nicefrac{1}{2}} \le (n+1)^{(n+1)}.
\end{equation}
\end{lemma}

\begin{proof}[Proof of \cref{prop:factorial}]
Throughout this proof let $f \colon \R \to \R$ satisfy for all $x \in [2,\infty)$ that
$
f(x) = (x - \tfrac{1}{2}) \pr[]{ \ln(x) - \ln(x-1) }.
$
Observe that for all $n\in\N$ it holds that
\begin{equation}\label{eq:ln_fac0}
\begin{split}
\ln(n!) 
& = \ln\bigl(n \cdot (n-1) \cdot \ldots \cdot 2 \cdot 1 \bigr) 
= \SmallSum_{k=1}^{n} \ln(k) \\
& = \SmallSum_{k=2}^n \left[ \displaystyle \int_{k-1}^k \ln(x) \, dx + \left( \ln(k) - \int_{k-1}^k \ln(x) \, dx \right) \right].
\end{split}
\end{equation}
In addition, note that for all $ k \in \N$ it holds that
\begin{align}
\ln(k) - \int_{k-1}^k \ln(x) \, dx
& = \ln(k) - \left[ \bigl( k\ln(k) - k \bigr) - \bigl( (k-1)\ln(k-1) - (k-1) \bigr) \right] \nonumber \\
& = 1 - (k-1) \bigl( \ln(k) - \ln(k-1) \bigr).
\end{align}
This and \cref{eq:ln_fac0} yield that for all $n\in\N$ it holds that
\begin{align}\label{eq_3_58}
\ln(n!) 
& = \SmallSum_{k=2}^n \displaystyle \int_{k-1}^k \ln(x) \, dx
+ \SmallSum_{k=2}^n \bigl[ 1 - (k-1) \bigl( \ln(k) - \ln(k-1) \bigr) \bigr] \nonumber \\
& = \int_1^n \ln(x) \, dx 
+ \SmallSum_{k=2}^n \bigl[ 1 - (k-1) \bigl( \ln(k) - \ln(k-1) \bigr) \bigr] \\
& = n\ln(n) - n + 1 
+ \tfrac{1}{2}\SmallSum_{k=2}^n \bigl( \ln(k) - \ln(k-1) \bigr) 
+ \sum_{k=2}^n \bigl[ 1 - (k-\tfrac{1}{2}) \bigl( \ln(k) - \ln(k-1) \bigr) \bigr] \nonumber \\
& = ( n + \tfrac{1}{2} ) \ln(n) - n + 1 
+ \SmallSum_{k=2}^n [ 1 - f(k) ].\nonumber 
\end{align}
Next observe that the fact that for all $x\in[2,\infty)$ it holds that $f(x) = (x - \tfrac{1}{2}) \pr[]{ \ln(x) - \ln(x-1) }$ implies that for all $x\in[2,\infty)$ it holds that
\begin{equation}\label{eq_3_59}
f'(x) 
= \pr[\big]{ \ln(x) - \ln(x-1) } + (x-\tfrac{1}{2}) \pr[\big]{ \tfrac{1}{x} - \tfrac{1}{x-1} }.
\end{equation}
This ensures that for all $x \in [2,\infty)$ it holds that
\begin{equation}
\begin{split}
f''(x) & = 2\pr[\big]{ \tfrac{1}{x} - \tfrac{1}{x-1} } - (x-\tfrac{1}{2}) \pr[\big]{ \tfrac{1}{x^2} - \tfrac{1}{(x-1)^2} } \\
& = \tfrac{2}{x^2(x-1)^2} \br[\Big]{ 2\pr[\big]{ x(x-1)^2 - x^2(x-1) } - (x-\tfrac{1}{2}) \pr[\big]{ (x-1)^2 - x^2 } } \\
& = \tfrac{2}{x^2(x-1)^2} \br[\big]{ -2x(x-1) + (x-\tfrac{1}{2})(2x-1) }
= \tfrac{1}{x^2(x-1)^2} > 0.
\end{split}
\end{equation}
Combining this and \cref{eq_3_59} shows that for all $x\in[2,\infty)$ it holds that $f'$ is increasing. This and the fact that \cref{eq_3_59} implies that $\lim_{x\to\infty} f'(x) = 0$ demonstrates that for all $x \in [2,\infty)$ it holds that $f'(x) \in (-\infty,0]$.
Hence, we obtain that for all $x \in [2,\infty)$ it holds that $f$ is non-increasing. Combining this and the fact that $\lim_{x\to\infty} f(x) = 1$ assures that for all $x \in [2,\infty)$ it holds that $f(x) \in [1,\infty)$.
This and \cref{eq_3_58} guarantee that for all $n\in\N$ it holds that
\begin{equation}\label{eq:ln_fac}
\ln(n!) 
= ( n + \tfrac{1}{2} ) \ln(n) - n + 1 
+ \SmallSum_{k=2}^n [ 1 - f(k) ] 
\le ( n + \tfrac{1}{2} ) \ln(n) - n + 1.
\end{equation}
Furthermore, note that for all $n\in\N$ it holds that
\begin{equation}
\begin{split}
\SmallSum_{k=2}^n [ 1 - f(k) ]
& = \SmallSum_{k=2}^n \bigl[ 1 - (k-\tfrac{1}{2}) \bigl( \ln(k) - \ln(k-1) \bigr) \bigr] \\
& = \1_{\N}(n-1) \br[\big]{ (n-1) - (n-\tfrac{1}{2})\ln(n) + \ln((n-1)!) }.
\end{split}
\end{equation}
Combining this, the fact that for all $x\in[2,\infty)$ it holds that $f$ is non-increasing, and \cref{eq_3_58} ensures that for all $n\in\N$ it holds that
\begin{align}
\SmallSum_{k=2}^n [ 1 - f(k) ]
& = \1_{\N}(n-1) \br[\big]{ (n-1) - (n-\tfrac{1}{2})\ln(n) + \ln((n-1)!) } \nonumber \\
& = \1_{\N}(n-1) \bigl[ \ln((n-1)!) -  \pr[\big]{ (n-\tfrac{1}{2})\ln(n-1) - (n-1) + 1 } \\
& \quad + 1 - (n-\tfrac{1}{2})\ln(n) + (n-\tfrac{1}{2})\ln(n-1) \bigr] \nonumber \\
& \ge \1_{\N}(n-1) [ 1 - f(n) ] 
\ge \1_{\N}(n-1) [ 1 - \tfrac{3}{2}\ln(2) ] 
\ge -\tfrac{1}{2}\ln(2). \nonumber
\end{align}
This and \cref{eq_3_58} show that for all $n\in\N$ it holds that
\begin{equation}
\ln(n!) 
= ( n + \tfrac{1}{2} ) \ln(n) - n + 1 
+ \SmallSum_{k=2}^n [ 1 - f(k) ] 
\ge ( n + \tfrac{1}{2} ) \ln(n) - n + 1 
- \tfrac{1}{2}\ln(2).
\end{equation}
Combining this and \cref{eq:ln_fac} proves that for all $n\in\N$ it holds that
\begin{equation}\label{eq:3_44}
\begin{split}
\frac{n^n}{\exp(n-1)} 
& \le 2^{-\nicefrac{1}{2}} \frac{n^{(n+\nicefrac{1}{2})}}{\exp(n-\nicefrac{1}{2})} 
= \exp\bigl((n+\nicefrac{1}{2})\ln(n) - n + 1 - \ln(2^{\nicefrac{1}{2}})\bigr) 
\le n! \\
& \le \exp((n+\nicefrac{1}{2})\ln(n) - n + 1) 
= \frac{n^{(n+\nicefrac{1}{2})}}{\exp(n-1)} 
\le \frac{(n+1)^{(n+1)}}{\exp(n)}.
\end{split}
\end{equation}
The proof of \cref{prop:factorial}
is thus complete.
\end{proof}

\begin{lemma}\label{lem:talk1}
Let $\floor{\cdot} \colon \R \to \Z$ satisfy for all $x\in\R$ that $\floor{x} = \max\{n \in \Z \colon n \le x\}$.
Then for all $\littleM\in [1,\infty)$ it holds that
\begin{equation}\label{lem:talk1_1}
\max_{n\in\N_0} \frac{\littleM^{\nicefrac{n}{2}}}{n!} 
\le \frac{\littleM^{\nicefrac{\lfloor \littleM^{1/2} \rfloor}{2}}}{\lfloor \littleM^{\nicefrac{1}{2}} \rfloor!} 
< \frac{\exp\pr[]{ \littleM^{\nicefrac{1}{2}} }}{(\lfloor \littleM^{\nicefrac{1}{2}} \rfloor )^{\nicefrac{1}{2}}} 
\le \exp\pr[]{ \littleM^{\nicefrac{1}{2}} }.
\end{equation}
\end{lemma}

\begin{proof}[Proof of \cref{lem:talk1}]
Throughout this proof let $\littleM\in [1,\infty)$, let $\left\lceil \cdot \right\rceil \colon \R \to \Z$ 
satisfy for all $x\in\R$ that $\lceil x \rceil = \min\{n \in \Z \colon x \le n\}$, and let $f \colon \N_0 \to \R$ satisfy for all $n\in\N_0$ that
$f(n) = \ln(\littleM^{\nicefrac{n}{2}}) - \ln(n!)$.
We claim that
\begin{equation}\label{M_max}
\max_{n\in\N_0} f(n) = f(\lfloor \littleM^{\nicefrac{1}{2}} \rfloor).
\end{equation}
Note that for all $n\in\N_0$ it holds that
\begin{equation}\label{eq_3_67}
f(n) = \ln(\littleM^{\nicefrac{n}{2}}) - \ln\bigl(n \cdot (n-1) \cdot \ldots \cdot 2 \cdot 1 \bigr) 
= \tfrac{n}{2}\ln(\littleM) - \smallsum_{k=1}^n \ln(k) .
\end{equation}
This guarantees that for all $n\in\N$ it holds that
\begin{equation}\label{eq_3_68}
\begin{split}
f(n) - f(n-1) 
& = \left[ \tfrac{n}{2}\ln(\littleM) - \smallsum_{k=1}^{n} \ln(k) \right] - \left[ \tfrac{n-1}{2}\ln(\littleM) - \smallsum_{k=1}^{n-1} \ln(k) \right] \\
& = \tfrac{1}{2} \ln(\littleM) - \ln(n)
= \ln(\littleM^{\nicefrac{1}{2}}) - \ln(n).
\end{split}
\end{equation}
This and the fact that $(0,\infty) \ni x \mapsto \ln(x) \in \R$ is increasing show that for all $n \in \{1,2,\dots,\lfloor \littleM^{\nicefrac{1}{2}} \rfloor\}$ it holds that $f(n) - f(n-1) \ge 0$. 
Furthermore, note that \cref{eq_3_68} and the fact that $(0,\infty) \ni x \mapsto \ln(x) \in \R$ is increasing assure that for all $n \in \N \cap [\lceil \littleM^{\nicefrac{1}{2}} \rceil, \infty)$ it holds that $f(n) - f(n-1) \le 0$.
Combining this and the fact that for all $n \in \{1,2,\dots,\lfloor \littleM^{\nicefrac{1}{2}} \rfloor \}$ it holds that $f(n) - f(n-1) \ge 0$ demonstrates that
\begin{equation}\label{f_max_1}
\max_{n\in\N_0} f(n) 
= \max\{ f(\lfloor \littleM^{\nicefrac{1}{2}} \rfloor), f(\lceil \littleM^{\nicefrac{1}{2}} \rceil) \}.
\end{equation}
Next observe that \cref{eq_3_67}, the fact that $(0,\infty) \ni x \mapsto \ln(x) \in \R$ is increasing, and the fact that for all $x\in\R$ it holds that $\lfloor x \rfloor \le \lceil x \rceil$ guarantee that
\begin{align}\label{eq:3_47}
& f(\lceil \littleM^{\nicefrac{1}{2}} \rceil) - f(\lfloor \littleM^{\nicefrac{1}{2}} \rfloor) \nonumber \\
& = \pr[\big]{ \lceil \littleM^{\nicefrac{1}{2}} \rceil \ln(\littleM^{\nicefrac{1}{2}}) - \smallsum_{k=1}^{\lceil \littleM^{\nicefrac{1}{2}} \rceil} \ln(k) } 
- \pr[\big]{ \lfloor \littleM^{\nicefrac{1}{2}} \rfloor \ln(\littleM^{\nicefrac{1}{2}}) - \smallsum_{k=1}^{\lfloor \littleM^{\nicefrac{1}{2}} \rfloor} \ln(k) } \nonumber \\
& = \pr[\big]{ \lceil \littleM^{\nicefrac{1}{2}} \rceil - \lfloor \littleM^{\nicefrac{1}{2}} \rfloor } \ln(\littleM^{\nicefrac{1}{2}}) 
- \pr[\big]{ \smallsum_{k=1}^{\lceil \littleM^{\nicefrac{1}{2}} \rceil} \ln(k) - \smallsum_{k=1}^{\lfloor \littleM^{\nicefrac{1}{2}} \rfloor} \ln(k) } \\
& = \1_{\R\backslash\N}(\littleM^{\nicefrac{1}{2}}) \ln(\littleM^{\nicefrac{1}{2}}) - \smallsum_{k=\lfloor \littleM^{\nicefrac{1}{2}} \rfloor + 1}^{\lceil \littleM^{\nicefrac{1}{2}} \rceil} \ln(k) 
= \1_{\R\backslash\N}(\littleM^{\nicefrac{1}{2}}) \br[\big]{ \ln(\littleM^{\nicefrac{1}{2}}) - \ln(\lceil \littleM^{\nicefrac{1}{2}} \rceil) } \le 0. \nonumber
\end{align}
Combining this and \cref{f_max_1} establishes \cref{M_max}.
In addition, observe that 
\cref{prop:factorial}, the fact that for all $x\in\R$ it holds that $\lfloor x \rfloor \le x$, and the fact that $\ln(6) < 2$ ensure that
\begin{align}\label{talk_bd1}
f(\lfloor \littleM^{\nicefrac{1}{2}} \rfloor) 
& = \ln(\littleM^{\nicefrac{\lfloor \littleM^{\nicefrac{1}{2}} \rfloor}{2}}) - \ln(\lfloor \littleM^{\nicefrac{1}{2}} \rfloor!) \nonumber \\
& \le \ln(\littleM^{\nicefrac{\lfloor \littleM^{\nicefrac{1}{2}} \rfloor}{2}}) 
- \pr[\big]{ \lfloor \littleM^{\nicefrac{1}{2}} \rfloor + \tfrac{1}{2} } \ln(\lfloor \littleM^{\nicefrac{1}{2}} \rfloor) 
+ \lfloor \littleM^{\nicefrac{1}{2}} \rfloor - 1 + \tfrac{1}{2}\ln(2) \nonumber \\
& = \br[\big]{ \ln(\littleM^{\nicefrac{\lfloor \littleM^{\nicefrac{1}{2}} \rfloor}{2}}) - \lfloor \littleM^{\nicefrac{1}{2}} \rfloor \ln(\lfloor \littleM^{\nicefrac{1}{2}} \rfloor) } 
- \tfrac{1}{2}\ln(\lfloor \littleM^{\nicefrac{1}{2}} \rfloor) + \lfloor \littleM^{\nicefrac{1}{2}} \rfloor - 1 + \tfrac{1}{2}\ln(2) \\
& \le \tfrac{1}{2}\ln(3) - \tfrac{1}{2}\ln(\lfloor \littleM^{\nicefrac{1}{2}} \rfloor) + \lfloor \littleM^{\nicefrac{1}{2}} \rfloor - 1 + \tfrac{1}{2}\ln(2) \nonumber \\
& < \lfloor \littleM^{\nicefrac{1}{2}} \rfloor - \tfrac{1}{2}\ln(\lfloor \littleM^{\nicefrac{1}{2}} \rfloor) 
\le \littleM^{\nicefrac{1}{2}} - \tfrac{1}{2}\ln(\lfloor \littleM^{\nicefrac{1}{2}} \rfloor). \nonumber
\end{align}
Combining this, \cref{M_max},
and the fact that $\R \ni x \mapsto \exp(x) \in (0,\infty)$ is monotone
yields that 
\begin{equation}
\begin{split}
\max_{n\in\N_0} \exp(f(n))
& = \max_{n\in\N_0} \frac{\littleM^{\nicefrac{n}{2}}}{n!} 
\le \frac{\littleM^{\nicefrac{\lfloor \littleM^{\nicefrac{1}{2}} \rfloor}{2}}}{\lfloor \littleM^{\nicefrac{1}{2}} \rfloor!} 
< \exp\pr[\big]{ \littleM^{\nicefrac{1}{2}} - \tfrac{1}{2}\ln(\lfloor \littleM^{\nicefrac{1}{2}} \rfloor) } \\
& = \frac{\exp\pr[]{ \littleM^{\nicefrac{1}{2}} }}{(\lfloor \littleM^{\nicefrac{1}{2}} \rfloor )^{\nicefrac{1}{2}}} 
\le \exp\pr[]{ \littleM^{\nicefrac{1}{2}} }.
\end{split}
\end{equation}
The proof of \cref{lem:talk1} is thus complete.
\end{proof}

\begin{lemma}\label{fn_gron}
Let $M,N \in \N$, $T \in (0,\infty)$, $\tau \in [0,T]$, $a,b \in [0,\infty)$, $p \in [1,\infty)$, let $\floor{\cdot} \colon \R \to \Z$ satisfy for all $x\in\R$ that $\floor{x} = \max\{ n \in \Z \colon n \le x \}$, let $f_n \colon [\tau,T] \to [0,\infty]$, $n \in \N_0$, be measurable, assume $\sup_{s \in [\tau,T]} \abs{f_0(s)} < \infty$, and assume for all $n \in \{1,2,\dots,N\}$, $t\in[\tau,T]$ that
\begin{equation}
\abs{ f_n(t) } \le \frac{a}{M^{\nicefrac{n}{2}}} + \sum_{i=0}^{n-1} \br[\Bigg]{ \frac{b}{M^{\nicefrac{(n-i-1)}{2}}} \br[\bigg]{ \int_t^T \abs{ f_i(s) }^p \, ds }^{\!\nicefrac{1}{p}} }.
\end{equation}
Then
\begin{equation}\label{eq:fn_gron}
\begin{split}
f_N(\tau) & \le \br[\bigg]{ a + b(T-\tau)^{\nicefrac{1}{p}} \br[\bigg]{ \sup_{s \in [\tau,T]} \abs{ f_0(s) } } } \frac{ \pr[\big]{ 1 + b(T-\tau)^{\nicefrac{1}{p}} }^{\!N-1} }{ M^{\nicefrac{(N-\floor{M^{\nicefrac{p}{2}}})}{2}} (\floor{ M^{\nicefrac{p}{2}} }!)^{\nicefrac{1}{p}} } \\
& \le \br[\bigg]{ a + b(T-\tau)^{\nicefrac{1}{p}} \br[\bigg]{ \sup_{s \in [\tau,T]} \abs{ f_0(s) } } } \frac{ \pr[\big]{ 1 + b(T-\tau)^{\nicefrac{1}{p}} }^{\!N-1} }{ M^{\nicefrac{N}{2}} \exp\pr[\big]{ - \nicefrac{M^{\nicefrac{p}{2}}}{p} } }.
\end{split}
\end{equation}
\end{lemma}

\begin{proof}[Proof of \cref{fn_gron}]
Note that Hutzenthaler et al.\ \cite[Lemma 3.10]{HutzenthalerJentzenKruseNguyen2020} (applied with $c \with M^{-\nicefrac{1}{2}}$, $\alpha \with \tau$, $\beta \with T$ in the notation of Hutzenthaler et al.\ \cite[Lemma 3.10]{HutzenthalerJentzenKruseNguyen2020}) assures that
\begin{equation}
f_N(\tau) \le \br[\bigg]{ a + b(T-\tau)^{\nicefrac{1}{p}} \br[\bigg]{ \sup_{s \in [\tau,T]} \abs{ f_0(s) } } } \left[ \sup_{k\in\N_0} \tfrac{M^{\nicefrac{k}{2}}}{(k!)^{\nicefrac{1}{p}}} \right] \pr[\big]{ 1 + b(T-\tau)^{\nicefrac{1}{p}} }^{\!N-1} .
\end{equation}
This, the fact that $a,b\in[0,\infty)$, and \cref{lem:talk1} (applied with $M \with M^p$ in the notation of \cref{lem:talk1}) prove that
\begin{align}
f_N(\tau) & \le \br[\bigg]{ a + b(T-\tau)^{\nicefrac{1}{p}} \br[\bigg]{ \sup_{s \in [\tau,T]} \abs{ f_0(s) } } } \frac{ \pr[\big]{ 1 + b(T-\tau)^{\nicefrac{1}{p}} }^{\!N-1} }{ M^{\nicefrac{(N-\floor{M^{\nicefrac{p}{2}}})}{2}} (\floor{ M^{\nicefrac{p}{2}} }!)^{\nicefrac{1}{p}} } \\
& \le \br[\bigg]{ a + b(T-\tau)^{\nicefrac{1}{p}} \br[\bigg]{ \sup_{s \in [\tau,T]} \abs{ f_0(s) } } } \frac{ \pr[\big]{ 1 + b(T-\tau)^{\nicefrac{1}{p}} }^{\!N-1} }{ M^{\nicefrac{N}{2}} \exp\pr[\big]{ - \nicefrac{M^{\nicefrac{p}{2}}}{p} } }. \nonumber
\end{align}
The proof of \cref{fn_gron} is thus complete.
\end{proof}

\subsection{Non-recursive error bounds for MLP approximations}\label{sec:mlp_stab}

\begin{lemma}\label{mlp_stab}
Assume \cref{setting1} and let $\floor{\cdot} \colon \R \to \Z$ satisfy for all $x\in\R$ that $\lfloor x \rfloor = \max\{n \in \Z \colon n \le x\}$.
Then it holds for all 
$n\in\N_0$, 
$t\in[0,T]$, $x\in\R^d$ that
\begin{align}\label{mlp_stab_estimate}
& \pr[\big]{ \E\br[\big]{ \abs[]{ \mlp^0_{n}(t,x+W_t^0) - \smallU(t,x+W_t^0) }^\fq } }^{\!\nicefrac{1}{\fq}}
 \nonumber \\
& \le \frac{ \fm \boundFG (T+1) \exp(\LipConstF T) (1 + 2\LipConstF T)^{n} }{ \littleM^{\nicefrac{(n-\floor{\littleM^{\nicefrac{\fq}{2}}})}{2}} (\floor{ \littleM^{\nicefrac{\fq}{2}} }!)^{\nicefrac{1}{\fq}} } \left[ \sup_{s\in[0,T]} \pr[\Big]{ \E\br[\Big]{ \pr[\big]{ 1 +  \norm{ x + W_{s}^{0} }^p }^{\!\fq} } }^{\!\!\nicefrac{1}{\fq}} \right] \\
& \le \frac{ \fm \boundFG (T+1) \exp(\LipConstF T) (1 + 2\LipConstF T)^{n} }{ \littleM^{\nicefrac{n}{2}} \exp\pr[\big]{ - \nicefrac{\littleM^{\nicefrac{\fq}{2}}}{\fq} } } \left[ \sup_{s\in[0,T]} \pr[\Big]{ \E\br[\Big]{ \pr[\big]{ 1 +  \norm{ x + W_{s}^{0} }^p }^{\!\fq} } }^{\!\!\nicefrac{1}{\fq}} \right]. \nonumber
\end{align}
\end{lemma}

\begin{proof}[Proof of \cref{mlp_stab}]
Throughout this proof assume without loss of generality that $T\in(0,\infty)$, let $t\in[0,T]$, $b = 2\LipConstF T^{\nicefrac{(\fq-1)}{\fq}} \fm$, let $a \colon \R^d \to \R$ satisfy 
\begin{equation}\label{eq_3_76a}
a(x) = \fm \exp(\LipConstF T) \Biggl[ \pr[\big]{ \E\br[\big]{ \abs[]{ \funcG(x+W_{T}^0) }^\fq } }^{\!\nicefrac{1}{\fq}} 
+ T^{\nicefrac{(\fq-1)}{\fq}} \pr[\bigg]{  \int_0^T \E\br[\big]{ \abs{ \smallF(s,x+W_{s}^0,0) }^\fq } \, ds }^{\!\!\nicefrac{1}{\fq}} \Biggr],
\end{equation}
and let $\fnSymb_{n,k} \colon [t,T] \times \R^d \to [0,\infty]$, $n\in\N_0$, $k \in \{0,1,\dots,n\}$, satisfy for all $n\in\N_0$, $k\in\{0,1,\dots,n\}$, $r\in[t,T]$, $x\in\R^d$ that
\begin{equation}\label{eq_3_77}
\fnSymb_{n,k}(r,x) = \pr[\big]{ \E\br[\big]{ \abs[]{ \mlp^0_{k}(r,x+W_r^0) - \smallU(r,x+W_r^0) }^\fq } }^{\!\nicefrac{1}{\fq}}.
\end{equation}
Note that \cref{eq_3_77} ensures that for all $n\in\N_0$, $k\in\{0,1,\dots,n\}$ it holds that $\fnSymb_{n,k}$ is measurable.
In addition, observe that \cref{eq_3_76a}, \cref{eq_3_77}, and \cref{mlp_stab_pre} assure that for all $n\in\N_0$, $k \in \{0,1,\dots,n\}$, $r\in[t,T]$, $x\in\R^d$ it holds that
\begin{align}\label{eq_3_78}
& \abs{ \fnSymb_{n,k}(r,x) } = \pr[\big]{ \E\br[\big]{ \abs[]{ \mlp^0_{k}(r,x+W_r^0) - \smallU(r,x+W_r^0) }^\fq } }^{\!\nicefrac{1}{\fq}} \nonumber \\
& \le \frac{\fm \exp(\LipConstF(T-r)) }{\littleM^{\nicefrac{n}{2}}}\Biggl[ \pr[\big]{ \E\br[\big]{ \abs[]{ \funcG(x+W_{T}^0) }^\fq } }^{\!\nicefrac{1}{\fq}} 
+ (T-r)^{\nicefrac{(\fq-1)}{\fq}} \pr[\bigg]{  \int_r^T \E\br[\big]{ \abs{ \smallF(s,x+W_{s}^0,0) }^\fq } \, ds }^{\!\!\nicefrac{1}{\fq}} \Biggr] \nonumber \\
& \quad + \sum_{i=0}^{k-1} \frac{\LipConstF (T-r)^{\nicefrac{(\fq-1)}{\fq}} \fm }{\littleM^{\nicefrac{(n-i)}{2}}} \br[\Bigg]{ \pr[\big]{ \1_{(0,n)}(i) + \littleM^{\nicefrac{1}{2}} } \pr[\bigg]{ \int_r^T \E\br[\big]{ \abs[]{ \pr[]{ \mlp_{i-1}^0 - \smallU } \pr[]{ s,x+W_{s}^0} }^\fq } \, ds }^{\!\!\nicefrac{1}{\fq}} } \nonumber \\
& \le \frac{\fm \exp(\LipConstF T) }{\littleM^{\nicefrac{n}{2}}}\Biggl[ \pr[\big]{ \E\br[\big]{ \abs[]{ \funcG(x+W_{T}^0) }^\fq } }^{\!\nicefrac{1}{\fq}} 
+ T^{\nicefrac{(\fq-1)}{\fq}} \pr[\bigg]{  \int_0^T \E\br[\big]{ \abs{ \smallF(s,x+W_{s}^0,0) }^\fq } \, ds }^{\!\!\nicefrac{1}{\fq}} \Biggr] \\
& \quad + \sum_{i=0}^{k-1} \frac{2 \LipConstF T^{\nicefrac{(\fq-1)}{\fq}} \fm }{\littleM^{\nicefrac{(n-i-1)}{2}}} \pr[\bigg]{ \int_r^T \E\br[\big]{ \abs[]{ \pr[]{ \mlp_{i-1}^0 - \smallU } \pr[]{ s,x+W_{s}^0} }^\fq } \, ds }^{\!\!\nicefrac{1}{\fq}} \nonumber \\
& = \frac{a(x)}{\littleM^{\nicefrac{n}{2}}} + \sum_{i=0}^{k-1} \br[\Bigg]{ \frac{b}{\littleM^{\nicefrac{(n-i-1)}{2}}} \br[\bigg]{ \int_r^T \abs{ \fnSymb_{n,i}(s,x) }^p \, ds }^{\!\nicefrac{1}{p}} }. \nonumber
\end{align}
Next note that \cref{b37}, \cref{eq_3_76a}, and \cref{cor:sol_bd1} assure that for all $n\in\N_0$, $r \in [t,T]$, $x\in\R^d$ it holds that
\begin{align}\label{eq_3_79}
\abs{ \fnSymb_{n,0}(r,x) } & = \pr[\big]{ \E\br[\big]{ \abs[]{ \mlp^0_{0}(r,x+W_r^0) - \smallU(r,x+W_r^0) }^\fq } }^{\!\nicefrac{1}{\fq}} = \pr[\big]{ \E\br[\big]{ \abs[]{ \smallU(r,x+W_r^0) }^\fq } }^{\!\nicefrac{1}{\fq}} \nonumber \\
& \le \boundFG (T+1) \exp(\LipConstF T) \left[ \sup_{s\in[0,T]} \pr[\big]{ \E\br[\big]{ \pr[]{ 1 + \norm{ x + \fwpr_s}^p }^{q} } }^{\!\nicefrac{1}{q}} \right] < \infty.
\end{align}
Combining this,
\cref{eq_3_76a}, \cref{eq_3_77}, 
\cref{eq_3_78},
and \cref{fn_gron} (applied for every $n\in\N$, $x\in\R^d$ with $a \with a(x)$, $b \with b$, $N \with n$, $\tau \with t$, $T \with T$, $(f_k)_{k\in\{0,1,\dots,n\}} \with ( [t,T] \ni r \mapsto \fnSymb_{n,k}(r,x) \in [0,\infty])_{k\in\{0,1,\dots,n\}}$
in the notation of \cref{fn_gron}) guarantees that for all $n\in\N$, $x\in\R^d$ it holds that
\begin{align}\label{3_result}
& \pr[\big]{ \E\br[\big]{ \abs[]{ \mlp^0_{n}(t,x+W_t^0) - \smallU(t,x+W_t^0) }^\fq } }^{\!\nicefrac{1}{\fq}}
= \fnSymb_{n,n}(t,x) \nonumber \\
& \le \br[\bigg]{ a(x) + b(T-t)^{\nicefrac{1}{\fq}} \br[\bigg]{ \sup_{s \in [t,T]} \abs{ \fnSymb_{n,0}(s,x) } } } \frac{ \pr[\big]{ 1 + b(T-t)^{\nicefrac{1}{\fq}} }^{\!n-1} }{ \littleM^{\nicefrac{(n-\floor{\littleM^{\nicefrac{\fq}{2}}})}{2}} (\floor{ \littleM^{\nicefrac{\fq}{2}} }!)^{\nicefrac{1}{\fq}} } \\
& \le \br[\bigg]{ a(x) + b(T-t)^{\nicefrac{1}{\fq}} \br[\bigg]{ \sup_{s \in [t,T]} \abs{ \fnSymb_{n,0}(s,x) } } } \frac{ \pr[\big]{ 1 + b(T-t)^{\nicefrac{1}{\fq}} }^{\!n-1} }{ \littleM^{\nicefrac{n}{2}} \exp\pr[\big]{ - \nicefrac{\littleM^{\nicefrac{\fq}{2}}}{\fq} } }. \nonumber
\end{align}
In addition, observe that \cref{fn_cond} demonstrates that for all $x\in\R^d$ it holds that
\begin{align}\label{eq_3_80}
& \pr[\big]{ \E\br[\big]{ \abs[]{ \funcG(x+W_{T}^0) }^\fq } }^{\!\nicefrac{1}{\fq}} 
+ T^{\nicefrac{(\fq-1)}{\fq}} \pr[\bigg]{  \int_0^T \E\br[\big]{ \abs{ \smallF(s,x+W_{s}^0,0) }^\fq } \, ds }^{\!\!\nicefrac{1}{\fq}} \nonumber \\
& \le \boundFG \pr[\big]{ \E\br[\big]{ \pr[]{ 1 +  \norm{ x + W_{T}^{0} }^p }^{\fq} } }^{\!\nicefrac{1}{\fq}} 
+ T^{\nicefrac{(\fq-1)}{\fq}} \pr[\bigg]{ \boundFG T \sup_{s\in[0,T]} \E\br[\big]{ \pr[]{ 1 +  \norm{ x + W_{s}^{0} }^p }^{\fq} } }^{\!\!\nicefrac{1}{\fq}} \\
& \le \boundFG (T+1) \left[ \sup_{s\in[0,T]} \pr[\big]{ \E\br[\big]{ \pr[]{ 1 +  \norm{ x + W_{s}^{0} }^p }^{\fq} } }^{\!\nicefrac{1}{\fq}} \right]. \nonumber
\end{align}
This, the fact that $b = 2\LipConstF T^{\nicefrac{(\fq-1)}{\fq}} \fm$, \cref{eq_3_76a}, and \cref{eq_3_79} show that for all $n\in\N$, $x\in\R^d$ it holds that
\begin{align}
& a(x) + b(T-t)^{\nicefrac{1}{\fq}} \br[\bigg]{ \sup_{s \in [t,T]} \abs{ \fnSymb_{n,0}(s,x) } } \nonumber \\
& \le \left[ 1 + 2\LipConstF T^{\nicefrac{(\fq-1)}{\fq}} (T-t)^{\nicefrac{1}{\fq}} \right] \fm \boundFG (T+1) \exp(\LipConstF T) \left[ \sup_{s\in[0,T]} \pr[\big]{ \E\br[\big]{ \pr[]{ 1 +  \norm{ x + W_{s}^{0} }^p }^{\fq} } }^{\!\nicefrac{1}{\fq}} \right] \\
& \le \left[ 1 + 2\LipConstF T \right] \fm \boundFG (T+1) \exp(\LipConstF T) \left[ \sup_{s\in[0,T]} \pr[\big]{ \E\br[\big]{ \pr[]{ 1 +  \norm{ x + W_{s}^{0} }^p }^{\fq} } }^{\!\nicefrac{1}{\fq}} \right]. \nonumber
\end{align}
Combining this, \cref{3_result}, and the fact that for all $n\in\N$ it holds that
\begin{equation}
\pr[\big]{ 1 + b(T-t)^{\nicefrac{1}{\fq}} }^{\!n-1} = \pr[\big]{ 1 + 2\LipConstF T^{\nicefrac{(\fq-1)}{\fq}} (T-t)^{\nicefrac{1}{\fq}} }^{\!n-1}
\le (1 + 2\LipConstF T)^{n-1}
\end{equation}
proves that for all $n\in\N$, $x\in\R^d$ it holds that
\begin{align}
& \pr[\big]{ \E\br[\big]{ \abs[]{ \mlp^0_{n}(t,x+W_t^0) - \smallU(t,x+W_t^0) }^\fq } }^{\!\nicefrac{1}{\fq}}
 \nonumber \\
& \le \frac{ \fm \boundFG (T+1) \exp(\LipConstF T) (1 + 2\LipConstF T)^{n} }{ \littleM^{\nicefrac{(n-\floor{\littleM^{\nicefrac{\fq}{2}}})}{2}} (\floor{ \littleM^{\nicefrac{\fq}{2}} }!)^{\nicefrac{1}{\fq}} } \left[ \sup_{s\in[0,T]} \pr[\big]{ \E\br[\big]{ \pr[]{ 1 +  \norm{ x + W_{s}^{0} }^p }^{\fq} } }^{\!\nicefrac{1}{\fq}} \right] \\
& \le \frac{ \fm \boundFG (T+1) \exp(\LipConstF T) (1 + 2\LipConstF T)^{n} }{ \littleM^{\nicefrac{n}{2}} \exp\pr[\big]{ - \nicefrac{\littleM^{\nicefrac{\fq}{2}}}{\fq} } } \left[ \sup_{s\in[0,T]} \pr[\big]{ \E\br[\big]{ \pr[]{ 1 +  \norm{ x + W_{s}^{0} }^p }^{\fq} } }^{\!\nicefrac{1}{\fq}} \right]. \nonumber
\end{align}
Combining this and \cref{eq_3_79} establishes \cref{mlp_stab_estimate}.
The proof of \cref{mlp_stab} is thus complete.
\end{proof}

\begin{corollary}\label{mlp_stab_cor}
Assume \cref{setting1}.
Then it holds for all $n\in\N_0$, $t\in[0,T]$, $x\in\R^d$ that
\begin{equation}\label{eq_3_76}
\begin{split}
& \pr[\Big]{ \E\br[\Big]{ \abs[\big]{ \mlp^0_{n}(t,x) - \smallU(t,x) }^\fq } }^{\!\!\nicefrac{1}{\fq}} \\
& \le \frac{ \fm \boundFG (T+1) \exp(\LipConstF T) (1 + 2\LipConstF T)^{n} }{ \littleM^{\nicefrac{n}{2}} \exp\pr[\big]{ - \nicefrac{\littleM^{\nicefrac{\fq}{2}}}{\fq} } } \left[ \sup_{s\in[0,T]} \pr[\Big]{ \E\br[\Big]{ \pr[\big]{ 1 +  \norm{ x + W_{s}^{0} }^p }^{\!\fq} } }^{\!\!\nicefrac{1}{\fq}} \right] .
\end{split}
\end{equation}
\end{corollary}

\begin{proof}[Proof of \cref{mlp_stab_cor}]
Throughout this proof 
let $V_\ft \colon [0,T-\ft]\times \R^d \to \R$, $\ft \in [0,T]$, satisfy for all $\ft\in[0,T]$, $t\in[0,T-\ft]$, $x\in\R^d$ that 
\begin{equation}\label{mlp_stab_cor1}
V_\ft(t,x) = \smallU(t+\ft,x),
\end{equation}
let $G_\ft \colon C([0,T-\ft]\times\R^d,\R) \to C([0,T-\ft]\times\R^d,\R)$, $\ft\in[0,T]$, satisfy for all $\ft\in[0,T]$, $t\in[0,T-\ft]$, $x\in\R^d$, $v\in C([0,T-\ft]\times\R^d,\R)$ that
\begin{equation}\label{mlp_stab_cor2}
(G_\ft(v))(t,x) = (\funcF(v))(t+\ft,x) ,
\end{equation}
let $\cR^{\ft,\theta} \colon [0,T-\ft]\times\Omega \to [0,T-\ft]$, $\ft\in[0,T]$, $\theta\in\Theta$, satisfy for all $\ft\in[0,T]$, $t\in[0,T-\ft]$, $\theta\in\Theta$ that
\begin{equation}\label{mlp_stab_cor_r}
\cR^{\theta,\ft}_t = t + (T-(t+\ft))\fu^\theta,
\end{equation}
and let $\fV_{n}^{\theta,\ft} \colon [0,T-\ft]\times\R^d\times \Omega \to \R$, $\ft\in[0,T]$, $n\in\N_0$, $\theta\in\Theta$, 
satisfy for all $\ft\in[0,T]$, $n\in\N_0$, $\theta\in\Theta$, $t\in[0,T-\ft]$, $x\in\R^d$ that 
\begin{equation}\label{mlp_stab_cor_mlp}
\fV_{n}^{\theta,\ft}(t,x) = \mlp_{n}^\theta(t+\ft,x).
\end{equation}
Observe that \cref{assumed_u}, \cref{mlp_stab_cor1}, \cref{mlp_stab_cor2}, 
and the fact that $W^0$ has independent increments ensure that for all $\ft\in[0,T]$, $t\in[0,T-\ft]$, $x\in\R^d$ it holds that
\begin{equation}\label{mlp_stab_cor3}
\begin{split}
V_\ft(t,x) 
= \smallU(t+\ft,x) 
& = \E\br[\big]{ \funcG(x + W^0_{T-(t+\ft)}) } + \int_{(t+\ft)}^T \E\br[\big]{ (\funcF(\smallU))(s,x+W^0_{s-(t+\ft)}) } \, ds \\
& = \E\br[\big]{ \funcG(x + W^0_{(T-\ft) -t}) } + \int_{t}^{(T-\ft)} \E\br[\big]{ (\funcF(\smallU))(s+\ft,x+W^0_{s-t}) } \, ds \\
& = \E\br[\big]{ \funcG(x + W^0_{(T-\ft) -t}) } + \int_{t}^{(T-\ft)} \E\br[\big]{ (G_\ft(V_\ft))(s,x+W^0_{s-t}) } \, ds.
\end{split}
\end{equation}
Combining this, \cref{mlp_stab_cor1}, \cref{mlp_stab_cor2}, 
and the hypothesis that for all $t\in[0,T]$, $x\in\R^d$ it holds that $\E[\abs{ \funcG(x+W^0_{T-t}) } + \int_t^T \abs{ \funcF(\smallU))(s,x+W^0_{s-t}) } \, ds] < \infty$ implies that for all $\ft\in[0,T]$, $t\in[0,T-\ft]$, $x\in\R^d$ it holds that
\begin{equation}\label{mlp_stab_cor4}
\begin{split}
& \E\left[\abs[\big]{ \funcG(x + W^0_{(T-\ft) -t})} + \int_{t}^{(T-\ft)} \abs[\big]{ (G_\ft(V_\ft))(s,x+W^0_{s-t}) } \, ds \right] \\
& = \E\left[\abs[\big]{ \funcG(x + W^0_{T-(t+\ft)}) } + \int_{(t+\ft)}^T \abs[\big]{ (\funcF(\smallU))(s,x+W^0_{s-(t+\ft)}) } \, ds \right] < \infty.
\end{split}
\end{equation}
Next note that
\cref{fn_cond}, \cref{big_F}, and \cref{mlp_stab_cor2} demonstrate that for all $\ft\in[0,T]$, $t\in[0,T-\ft]$, $x\in\R^d$, $v,w\in C([0,T-\ft]\times\R^d,\R)$ it holds that
\begin{equation}\label{mlp_stab_cor5}
\begin{split}
& \abs{ (G_\ft(v))(t,x) - (G_\ft(w))(t,x) } 
= \abs{ (\funcF(v))(t+\ft,x) - (\funcF(w))(t+\ft,x) } \\
& = \abs{ \smallF(t+\ft,x,v(t+\ft,x)) - \smallF(t+\ft,x,w(t+\ft,x)) } 
\le \LipConstF \abs{ v(t+\ft,x) - w(t+\ft,x) }.
\end{split}
\end{equation}
In addition, observe that \cref{fn_cond}, \cref{big_F}, and \cref{mlp_stab_cor2} show that for all $\ft\in[0,T]$, $t\in[0,T-\ft]$, $x\in\R^d$ it holds that
\begin{equation}\label{mlp_stab_cor6}
\abs{ (G_\ft(0))(t,x) } 
= \abs{ \smallF(t+\ft,x,0) } 
\le \boundFG (1 + \norm{ x }^p).
\end{equation}
Moreover, note that
\cref{mlp_stab_cor1}, 
\cref{mlp_stab_cor3}, and the hypothesis that $\smallU \in C([0,T]\times \R^d,\R)$ assure that for all $\ft\in[0,T]$, $t\in[0,T-\ft]$, $x\in\R^d$ it holds that $V_\ft \in C([0,T-\ft]\times\R^d,\R)$. 
Furthermore, observe that 
\cref{mlp_stab_cor_r}, 
the hypothesis that $(\fu^\theta)_{\theta\in\Theta}$ are i.i.d.\ random variables, and the hypothesis that $(W^\theta)_{\theta\in\Theta}$ and $(\cU^\theta)_{\theta\in\Theta}$ are independent
ensure that for all $\ft\in[0,T]$ it holds that 
$(W^\theta)_{\theta\in\Theta}$ and $(\cR ^{\theta,\ft})_{\theta\in\Theta}$ are independent on $[0,T-\ft]$.
Next note that \cref{mlp_stab_cor_mlp} implies 
that for all $\ft\in[0,T]$, $n\in\N_0$, $\theta\in\Theta$, $t\in[0,T-\ft]$, $x\in\R^d$ it holds that
\begin{align}\label{eq_4_79}
& \fV_{n}^{\theta,\ft}(t,x) 
= \mlp_{n}^\theta(t+\ft,x) 
= \tfrac{\1_\N(n)}{\littleM^n} \br[\bigg]{ \SmallSum_{k=1}^{\littleM^n} \funcG\pr[]{ x + W_{T-(t+\ft)}^{(\theta,0,-k)} } } \\
& + \SmallSum_{i=0}^{n-1} \frac{(T-(t+\ft))}{\littleM^{n-i}} \br[\bigg]{ \SmallSum_{k=1}^{\littleM^{n-i}} \br[\big]{ \pr[]{ \funcF (\mlp_{i}^{(\theta,i,k)} ) - \1_\N(i) \funcF ( \mlp_{i-1}^{(\theta,-i,k)} ) } \pr[]{ \cU_{(t+\ft)}^{(\theta,i,k)}, x + W_{\cU_{(t+\ft)}^{(\theta,i,k)}-(t+\ft)}^{(\theta,i,k)} } } }. \nonumber
\end{align}
Combining this, the fact that for all $t\in[0,T]$, $\theta\in\Theta$ that $\cU_t = t + (T-t)\fu^\theta$, \cref{mlp_stab_cor2}, \cref{mlp_stab_cor_r}, and  \cref{mlp_stab_cor_mlp} shows that for all $\ft\in[0,T]$, $n\in\N_0$, $\theta\in\Theta$, $t\in[0,T-\ft]$, $x\in\R^d$ it holds that
\begin{align}
& \fV_{n}^{\theta,\ft}(t,x) 
= \tfrac{\1_\N(n)}{\littleM^n}\br[\bigg]{ \SmallSum_{k=1}^{\littleM^n} \funcG\pr[]{ x + W_{(T-\ft)-t}^{(\theta,0,-k)} } } \nonumber \\
& + \SmallSum_{i=0}^{n-1} \frac{(T-\ft)-t)}{\littleM^{n-i}} \br[\bigg]{ \sum_{k=1}^{\littleM^{n-i}} \br[\big]{ \pr[]{ \funcF (\mlp_{i}^{(\theta,i,k)} ) - \1_\N(i)  \funcF ( \mlp_{i-1}^{(\theta,-i,k)} ) } \pr[]{ \ft + \cR_{t}^{(\theta,i,k),\ft}, x + W_{\cR_{t}^{(\theta,i,k),\ft}-t}^{(\theta,i,k)} } } } \nonumber \\
& = \tfrac{\1_\N(n)}{\littleM^n}\br[\bigg]{ \SmallSum_{k=1}^{\littleM^n} \funcG\pr[]{ x + W_{(T-\ft)-t}^{(\theta,0,-k)} } } \\
& \quad + \SmallSum_{i=0}^{n-1} \frac{(T-\ft)-t)}{\littleM^{n-i}} \br[\bigg]{ \sum_{k=1}^{\littleM^{n-i}} \br[\big]{ \pr[]{ G_\ft (\fV_{i}^{(\theta,i,k),\ft} ) - \1_\N(i) G_\ft ( \fV_{i-1}^{(\theta,-i,k),\ft} ) } \pr[]{ \cR_{t}^{(\theta,i,k),\ft}, x + W_{\cR_{t}^{(\theta,i,k),\ft}-t}^{(\theta,i,k)} } } }. \nonumber
\end{align}
Combining this, \cref{mlp_stab_cor3}, \cref{mlp_stab_cor4}, \cref{mlp_stab_cor5}, \cref{mlp_stab_cor6}, the fact that $1 + \littleM^{-\nicefrac{1}{2}} \le 2$, the fact that for all $\ft\in[0,T]$, $t\in[0,T-\ft]$, $x\in\R^d$ it holds that $V_\ft \in C([0,T-\ft]\times\R^d,\R)$, and \cref{mlp_stab} (applied for every $\ft\in[0,T]$ with $\LipConstF \with \LipConstF$, $\boundFG \with \boundFG$, $p \with p$, $\fq \with \fq$, $T \with (T-\ft)$, $\funcG \with \funcG$, $\funcF \with G_\ft$, $(\cU^\theta)_{\theta\in\Theta} \with (\cR ^{\theta,\ft})_{\theta\in\Theta}$, $\smallU \with V_\ft$, $(\mlp_{n}^\theta)_{(n,\theta)\in\N_0\times\Theta} \with (\fV_{n}^{\theta,\ft})_{(n,\theta)\in\N_0\times\Theta}$ in the notation of \cref{mlp_stab}) demonstrates that for all $\ft\in[0,T]$, $t\in[0,T-\ft]$, $x\in\R^d$, $n\in\N_0$ it holds that
\begin{align}
& \pr[\big]{ \E\br[\big]{ \abs[]{ \fV^{0,\ft}_{n}(t,x+W_t^0) - V_\ft(t,x+W_t^0) }^\fq } }^{\!\nicefrac{1}{\fq}} \nonumber \\
& \le \frac{ \fm \boundFG ((T-\ft)+1) \exp(\LipConstF (T-\ft)) (1 + 2\LipConstF (T-\ft))^{n} }{ \littleM^{\nicefrac{n}{2}} \exp\pr[\big]{ - \nicefrac{\littleM^{\nicefrac{\fq}{2}}}{\fq} } } \left[ \sup_{s\in[0,T-\ft]} \pr[\big]{ \E\br[\big]{ \pr[]{ 1 +  \norm{ x + W_{s}^{0} }^p }^{\fq} } }^{\!\nicefrac{1}{\fq}} \right] \\
& \le \frac{ \fm \boundFG (T+1) \exp(\LipConstF T) (1 + 2\LipConstF T)^{n} }{ \littleM^{\nicefrac{n}{2}} \exp\pr[\big]{ - \nicefrac{\littleM^{\nicefrac{\fq}{2}}}{\fq} } } \left[ \sup_{s\in[0,T]} \pr[\big]{ \E\br[\big]{ \pr[]{ 1 +  \norm{ x + W_{s}^{0} }^p }^{\fq} } }^{\!\nicefrac{1}{\fq}} \right]. \nonumber
\end{align}
This and \cref{mlp_stab_cor_mlp} prove that for all $\ft\in[0,T]$, $x\in\R^d$, $n\in\N_0$ it holds that
\begin{equation}
\begin{split}
& \pr[\big]{ \E\br[\big]{ \abs[]{ \mlp^0_{n}(\ft,x) - \smallU(\ft,x) }^\fq } }^{\!\nicefrac{1}{\fq}} 
= \pr[\big]{ \E\br[\big]{ \abs[]{ \fV^{0,\ft}_{n}(0,x+W_0^0) - V_\ft(0,x+W_0^0) }^\fq } }^{\!\nicefrac{1}{\fq}} \\
& \le \frac{ \fm \boundFG (T+1) \exp(\LipConstF T) (1 + 2\LipConstF T)^{n} }{ \littleM^{\nicefrac{n}{2}} \exp\pr[\big]{ - \nicefrac{\littleM^{\nicefrac{\fq}{2}}}{\fq} } } \left[ \sup_{s\in[0,T]} \pr[\big]{ \E\br[\big]{ \pr[]{ 1 +  \norm{ x + W_{s}^{0} }^p }^{\fq} } }^{\!\nicefrac{1}{\fq}} \right].
\end{split}
\end{equation}
The proof of \cref{mlp_stab_cor} is thus complete.
\end{proof}

\section{Computational complexity analysis for MLP approximations}\label{sec:main_results}

In this section we use the results from \cref{sec:3} to provide the complexity analysis for MLP approximations of solutions to stochastic fixed-point equations and semilinear PDEs.
The main result of this section is 
\cref{cor:final} in \cref{sec:4_4} below.
The proof of \cref{cor:final} employs \cref{th:final} and the elementary auxiliary result in \cref{lem:fn_prop}.
The proof of \cref{th:final}, in turn, is based on \cref{mlp_stab_cor} and the elementary estimate for full-history recursions in \cref{cor:full_history_gamma_1}.
Our proof of \cref{cor:full_history_gamma_1} employs the elementary result for full-history recursions in \cref{lem:full_history}.
Our proof of \cref{lem:full_history}, in turn, is based on the elementary result for two-step recursions in \cref{lem:three_term_rec_inhom}.
\cref{lem:three_term_rec_inhom} is a special case of Hutzenthaler et al.\ \cite[Lemma 2.1]{hutzenthaler2021multilevel}.
Only for completeness we include in this section the detailed proof of \cref{lem:three_term_rec_inhom}.

\subsection{Elementary estimates for two-step recursions}
\label{sec:two_step}

\begin{lemma}
\label{lem:three_term_rec_inhom}
Let $ \beta_1, \beta_2, b_1, b_2, \alpha_0, \alpha_1, \alpha_2, \ldots \in \C $
and let $x_k\in\C$, $k\in\N_0$,
satisfy for all $ k \in \N_0 $, $ j \in \{ 1, 2 \} $
that $ x_k = \alpha_k  + \1_{ [1,\infty) }( k ) \, \beta_1 \, x_{ \max\{ k - 1 , 0 \} } + \allowbreak \1_{ [2,\infty) }( k ) \, \beta_2 \, x_{ \max\{ k - 2 , 0 \} }$, $ ( \beta_1 )^2 \neq - 4 \beta_2 $, and $ b_j = \frac{ 1 }{ 2 } \pr[\big]{ \beta_1 - ( - 1 )^j \sqrt{ ( \beta_1 )^2 + 4 \beta_2 } }$.
Then it holds for all $ k \in \N_0 $ that
$ b_1 - b_2 = \sqrt{ ( \beta_1 )^2 + 4 \beta_2 } \neq 0 $
and
\begin{equation}
\begin{split}
x_k 
& = \frac{ 1 }{ \left( b_1 - b_2 \right) } \sum_{ l = 0 }^k \alpha_l \pr[\big]{ \left[ b_1 \right]^{ k + 1 - l } - \left[ b_2 \right]^{ k + 1 - l } } \\ 
& =  \sum_{ l = 0 }^k \frac{ \alpha_l \pr[\big]{ \pr[\big]{ \beta_1 + \sqrt{ ( \beta_1 )^2 + 4 \beta_2 } }^{ k + 1 - l } - \br[\big]{ \beta_1 - \sqrt{ ( \beta_1 )^2 + 4 \beta_2 } }^{ k + 1 - l } } }{ 2^{ ( k + 1 - l ) } \sqrt{ ( \beta_1 )^2 + 4 \beta_2 } }.
\end{split}
\end{equation}
\end{lemma}

\begin{proof}[Proof of \cref{lem:three_term_rec_inhom}]
Throughout this proof 
let $ x_{ - 1 }, x_{ - 2 } \in \C $
satisfy that
$
  x_{ - 1 } = x_{ - 2 } = 0
$
and 
let $ y_k \in \C $, $k\in\N_0$,
satisfy for all $ k \in \N_0 $ that
\begin{equation}
\label{eq:def_yk}
  y_k = 
  \frac{ 1 }{ \pr[\big]{ b_2 - b_1 } }
  \SmallSum_{ l = 0 }^k
  \alpha_l
  \bigl(
    \left[ b_2 \right]^{ k + 1 - l }
    -
    \left[ b_1 \right]^{ k + 1 - l }
  \bigr)
  .
\end{equation}
Note that \cref{eq:def_yk}
and the fact that $x_0=\alpha_0$
ensure that
\begin{equation}
\label{eq:y0}
  y_0 
=
  \frac{ 1 }{ \left( b_2 - b_1 \right) }
  \SmallSum_{ l = 0 }^0
  \alpha_l
  \bigl(
    \left[ b_2 \right]^{ 1 - l }
    -
    \left[ b_1 \right]^{ 1 - l }
  \bigr)
  =
  \frac{ 1 }{ \left( b_2 - b_1 \right) }
  \cdot
  \alpha_0
  \cdot
  \bigl(
    b_2 
    -
    b_1 
  \bigr) 
  =   
  \alpha_0 
  = 
  x_0 
  .
\end{equation}
This,
\cref{eq:def_yk},
the fact that
for all $k \in \N_0$ it holds that $x_k = \alpha_k + \beta_1 x_{ k - 1 } + \beta_2 x_{ k - 2 }
$,
the fact that
$
  x_{ - 1 } = x_{ - 2 } = 0
$,
and the fact that
$ b_1 + b_2 = \beta_1 $
prove that
\begin{equation}
\label{eq:y1}
\begin{split}
  y_1 
& = 
  \frac{ 1 }{ \left( b_2 - b_1 \right) }
  \SmallSum_{ l = 0 }^1
  \alpha_l
  \bigl(
    \left[ b_2 \right]^{ 2 - l }
    -
    \left[ b_1 \right]^{ 2 - l }
  \bigr) =
  \frac{ 1 }{ \left( b_2 - b_1 \right) } 
  \bigl[
    \alpha_0 
    \left( \left[ b_2 \right]^2 - \left[ b_1 \right]^2 \right) 
    + 
    \alpha_1 
    \left( b_2 - b_1 \right) 
  \bigr]
  \\
&  =
  \alpha_0 \left( b_1 + b_2 \right) 
  +
  \alpha_1 
  =
  \alpha_0 \beta_1 
  +
  \alpha_1 
  = 
  x_0 \beta_1 
  +
  \alpha_1 
  =
  \alpha_1 + \beta_1 x_0 + \beta_2 x_{ - 1 }
  =
  x_1
  .
\end{split}
\end{equation}
Next observe that the quadratic formula implies that
for all $ j \in \{ 1, 2 \} $ it holds
that $(b_j)^2=\beta_1 b_j+\beta_2$.
This and \cref{eq:def_yk} assure that
for all $ k \in \N_0 $ it holds
that
\begin{align}
  y_{ k + 2 }
  & 
  =
  \frac{ 1 }{ \left( b_2 - b_1 \right) }
  \SmallSum_{ l = 0 }^{ k + 2 }
  \alpha_l
  \bigl( 
    \left[ b_2 \right]^{ k + 3 - l }
    -
    \left[ b_1 \right]^{ k + 3 - l }
  \bigr) 
  =
  \tfrac{ 
    \alpha_{ k + 2 }
    \left(
      b_2 - b_1
    \right)
  }{ \left( b_2 - b_1 \right) }
  +
  \SmallSum_{ l = 0 }^{ k + 1 }
  \tfrac{ 
    \alpha_l
  }{ \left( b_2 - b_1 \right) }
  \bigl( 
    \left[ b_2 \right]^{ k + 3 - l }
    -
    \left[ b_1 \right]^{ k + 3 - l }
  \bigr) \nonumber
\\
&
  =
  \alpha_{ k + 2 }
  +
  \SmallSum_{ l = 0 }^{ k + 1 }
  \frac{ \alpha_l }{ \left( b_2 - b_1 \right) }
  \bigl(
    \left[ b_2 \right]^{ k + 1 - l }
    \left[ \beta_1 b_2 + \beta_2 \right]
    -
    \left[ b_1 \right]^{ k + 1 - l }
    \left[ \beta_1 b_1 + \beta_2 \right] 
  \bigr)
\end{align}
and
\begin{align}\label{eq:yk_plus_2}
  y_{ k + 2 }
&
  =
  \beta_1
  \left[
  \SmallSum_{ l = 0 }^{ k + 1 }
  \tfrac{ \alpha_l }{ \left( b_2 - b_1 \right) }
  \Bigl(
    \left[ b_2 \right]^{ k + 2 - l }
    -
    \left[ b_1 \right]^{ k + 2 - l }
  \Bigr)
  \right] 
  +
  \beta_2
  \left[
  \SmallSum_{ l = 0 }^{ k + 1 }
  \tfrac{ \alpha_l }{ \left( b_2 - b_1 \right) }
  \bigl(
    \left[ b_2 \right]^{ k + 1 - l }
    -
    \left[ b_1 \right]^{ k + 1 - l }
  \bigr)
  \right]
  +
  \alpha_{ k + 2 } \nonumber
\\
&
  =
  \beta_1 y_{ k + 1 }
  +
  \beta_2
  \left[
  \SmallSum_{ l = 0 }^k
  \tfrac{ \alpha_l }{ \left( b_2 - b_1 \right) }
  \bigl(
    \left[ b_2 \right]^{ k + 1 - l }
    -
    \left[ b_1 \right]^{ k + 1 - l }
  \bigr)
  \right]
  +
  \alpha_{ k + 2 }
  =
  \beta_1 y_{ k + 1 }
  +
  \beta_2 y_k
  +
  \alpha_{ k + 2 }
  .
\end{align}
Combining \cref{eq:y0}, \cref{eq:y1},
and \cref{eq:yk_plus_2} hence
ensures that for all 
$ k \in \N_0 $
it holds that
$
  y_k = x_k
$.
The proof of \cref{lem:three_term_rec_inhom}
is thus complete.
\end{proof}

\subsection{Elementary estimates for full-history recursions}
\label{sec:full_history}

\begin{lemma}
\label{lem:full_history}
Let $ \gamma \in \{ 0, 1 \} $,
$ \beta \in (0,\infty) $,
let $\alpha_k \in \C$, $k\in\N_0$,
and $x_k\in\C$, $k\in\N_0$,
satisfy 
for all $ k \in \N_0 $
that
\begin{equation}
\label{eq:assumption_x_dynamic}
\begin{split}
&
  x_k
  =
  \alpha_k
  +
  \sum_{ l = 0 }^{ k - 1 }
  \left( k - l \right)^{ \gamma }
  \beta^{ ( k - l ) }
  \left[
    x_l
    +
    \1_{ \N }( l )
    \,
    x_{ \max\{ l - 1 , 0 \} }
  \right]
  .
\end{split}
\end{equation}
Then it holds for all 
$ k \in \N_0 $ that
\begin{align}\label{4_8}
x_k & =
  \sum\limits_{ l = 0 }^k
  \tfrac{ 
  \br{
    \alpha_l 
    - 
    \1_{
      \N
    }( l ) \,
    2^\gamma \beta 
    \alpha_{ \max\{ l - 1 , 0 \} }
    +
    \1_{
      [2,\infty)
    }( l ) \, \gamma
    \beta^2 \alpha_{ \max\{ l - 2 , 0 \} }
  }
  }{ 
    2^{1 + \gamma ( k - l ) }
    \sqrt{
      5^\gamma \beta^2 + 4^\gamma \beta 
    }
 \pr[\big]{
    \br[\big]{
        3^\gamma \beta
        +
        \sqrt{
          5^\gamma \beta^2
          +
          4^\gamma \beta
        }
    }^{ k + 1 - l }
    -
    \br[\big]{ 
        3^\gamma \beta
        -
        \sqrt{
          5^\gamma \beta^2
          +
          4^\gamma \beta
        }
    }^{ k + 1 - l }
  }^{-1} } 
\\ 
& =
  \begin{cases}
  \sum\limits_{ l = 0 }^k
  \frac{ 
  \big[
    \alpha_l 
    - 
    \1_{ \N }( l ) \,
    \beta \alpha_{ \max\{ l - 1 , 0 \} }
  \big]
  \big(
    \big[ 
        \beta
        +
        \sqrt{
          \beta^2
          +
          \beta
        }
    \big]^{ k + 1 - l }
    -
    \big[ 
        \beta
        -
        \sqrt{
          \beta^2
          +
          \beta
        }
    \big]^{ k + 1 - l }
  \big)
  }{ 
    2
    \sqrt{
      \beta^2 + \beta
    }
  }
  &
    \colon \gamma = 0 
  \\ 
  \sum\limits_{ l = 0 }^k
  \frac{ 
  \big[
    \alpha_l 
    - 
    \1_{
      \N
    }( l ) \,
    2 \beta 
    \alpha_{ \max\{ l - 1 , 0 \} }
    +
    \1_{
      [2,\infty)
    }( l ) \,
    \beta^2 \alpha_{ \max\{ l - 2 , 0 \} }
  \big]
  }{ 
    2^{ ( k + 1 - l ) }
    \sqrt{
      5 \beta^2 + 4 \beta 
    }
    \big(
    \big[ 
        3 \beta
        +
        \sqrt{
          5 \beta^2
          +
          4 \beta
        }
    \big]^{ k + 1 - l }
    -
    \big[ 
        3 \beta
        -
        \sqrt{
          5 \beta^2
          +
          4 \beta
        }
    \big]^{ k + 1 - l }
  \big)^{-1}
  }
  &
    \colon \gamma = 1
  \end{cases}
  . \nonumber
\end{align}
\end{lemma}

\begin{proof}[Proof 
of \cref{lem:full_history}]
Throughout this proof let 
$ 
  x_{ - 1 } , x_{ - 2 } , x_{ - 3 } ,
  \alpha_{ - 1 }, \alpha_{ - 2 } , \alpha_{ - 3 }
  \in \C 
$
satisfy that
$
  x_{ - 1 } \allowbreak = x_{ - 2 } = x_{ - 3 } 
  = \alpha_{ - 1 } = \alpha_{ - 2 } = \alpha_{ - 3 }
  = 0
$.
Note that \cref{eq:assumption_x_dynamic} 
ensures that for all $ k \in \N_0 $
it holds that
\begin{equation}
  x_k
  =
  \alpha_k
  +
  \SmallSum_{ l = 0 }^{ k - 1 }
  \left( k - l \right)^{ \gamma }
  \beta^{ ( k - l ) }
  \left[
    x_l
    +
    x_{ l - 1 }
  \right]
  .
\end{equation}
Hence, we obtain that
for all $ k \in \N_0 $ it holds
that
\begin{align}
&  x_{ k + 1 }
  -
  \beta
  x_k
 =
  \left[ 
    \alpha_{ k + 1 }
    +
    \SmallSum_{ l = 0 }^{ k }
    \left( k + 1 - l \right)^{ \gamma }
    \beta^{ ( k + 1 - l ) }
    \left[
      x_l
      +
      x_{ l - 1 }
    \right]
  \right] 
  -
  \beta
  \left[
    \alpha_k
    +
    \SmallSum_{ l = 0 }^{ k - 1 }
    \left( k - l \right)^{ \gamma }
    \beta^{ ( k - l ) }
    \left[
      x_l
      +
      x_{ l - 1 }
    \right]
  \right] \nonumber
\\ & \quad =
  \alpha_{ k + 1 } - \beta \alpha_k
  +
    \beta
    \left[
      x_k
      +
      x_{ k - 1 }
    \right]
  +
    \SmallSum_{ l = 0 }^{ k - 1 }
    \beta^{ ( k + 1 - l ) }
    \left[
    \left( k + 1 - l \right)^{ \gamma }
    -
    \left( k - l \right)^{ \gamma }
    \right]
    \left[
      x_l
      +
      x_{ l - 1 }
    \right]
  . 
\end{align}
This implies that for all $ k \in \N_0 $ it holds
that
\begin{equation}
\begin{split}
  x_{ k + 1 }
  -
  \beta
  x_k
& =
  \alpha_{ k + 1 } 
  - \beta \alpha_k
  +
    \beta
    \left[
      x_k
      +
      x_{ k - 1 }
    \right]
  +
  \1_{ \{ 1 \} }( \gamma )
    \SmallSum_{ l = 0 }^{ k - 1 }
    \beta^{ ( k + 1 - l ) }
    \left[
      x_l
      +
      x_{ l - 1 }
    \right]
  .
\end{split}
\end{equation}
This proves that
for all $ k \in \N_0 $ it holds that
\begin{equation}
\begin{split}
  x_{ k + 1 }
& =
  \alpha_{ k + 1 } - \beta \alpha_k
  +
    2
    \beta
      x_k
      +
    \beta
    \,
      x_{ k - 1 }
  +
  \1_{ \{ 1 \} }( \gamma )
    \SmallSum_{ l = 0 }^{ k - 1 }
    \beta^{ ( k + 1 - l ) }
    \left[
      x_l
      +
      x_{ l - 1 }
    \right]
  .
\end{split}
\end{equation}
This and the fact that $ x_0 = \alpha_0 $ show that
for all $ k \in \{ - 1, 0, 1, 2, \dots \} $ it holds that
\begin{equation}
\label{eq:dynamic_x_2a}
\begin{split}
  x_k
& =
  \left[
    \alpha_k - \beta \alpha_{ k - 1 }
  \right]
  +
    2
    \beta
      x_{ k - 1 }
      +
    \beta
      x_{ k - 2 }
  +
  \1_{ \{ 1 \} }( \gamma )
    \SmallSum_{ l = 0 }^{ k - 2 }
    \beta^{ ( k - l ) }
    \left[
      x_l
      +
      x_{ l - 1 }
    \right]
  .
\end{split}
\end{equation}
Therefore, we obtain that
for all $ k \in \N_0 $ it holds
that
\begin{align}\label{eq:3_96}
  x_k - \beta x_{ k - 1 }
& =
  \left[
  \left[
    \alpha_k - \beta \alpha_{ k - 1 }
  \right]
  +
    2
    \beta
      x_{ k - 1 }
      +
    \beta
      x_{ k - 2 }
  +
  \1_{ \{ 1 \} }( \gamma )
    \SmallSum_{ l = 0 }^{ k - 2 }
    \beta^{ ( k - l ) }
    \left[
      x_l
      +
      x_{ l - 1 }
    \right]
  \right] \nonumber 
\\ & \quad
  -
  \beta
  \left[ 
  \left[
    \alpha_{ k - 1 } - \beta \alpha_{ k - 2 }
  \right]
  +
    2
    \beta
      x_{ k - 2 }
      +
    \beta
      x_{ k - 3 }
  +
  \1_{ \{ 1 \} }( \gamma )
    \SmallSum_{ l = 0 }^{ k - 3 }
    \beta^{ ( k - 1 - l ) }
    \left[
      x_l
      +
      x_{ l - 1 }
    \right]
  \right] \nonumber 
\\ & 
=
  \left[
    \alpha_k - 2 \beta \alpha_{ k - 1 }
    +
    \beta^2 \alpha_{ k - 2 }
  \right]
  +
    2
    \beta
      x_{ k - 1 }
      +
    \beta
    \left[ 1 - \beta \right]
      x_{ k - 2 }
    -
    \beta^2
      x_{ k - 2 }
    -
    \beta^2 x_{ k - 3 } \nonumber 
\\ & \quad
  +
  \1_{ \{ 1 \} }( \gamma )
  \left[
    \SmallSum_{ l = 0 }^{ k - 2 }
    \beta^{ ( k - l ) }
    \left[
      x_l
      +
      x_{ l - 1 }
    \right]
    -
    \SmallSum_{ l = 0 }^{ k - 3 }
    \beta^{ ( k - l ) }
    \left[
      x_l
      +
      x_{ l - 1 }
    \right]
  \right]
\\ & 
=
  \left[
    \alpha_k - 2 \beta \alpha_{ k - 1 }
    +
    \beta^2 \alpha_{ k - 2 }
  \right]
  +
    2
    \beta
      x_{ k - 1 }
      +
    \beta
    \left[ 1 - \beta \right]
      x_{ k - 2 }
    -
    \beta^2
    \left[
      x_{ k - 2 }
      +
      x_{ k - 3 }
    \right] \nonumber 
\\ & \quad
  +
  \1_{ \{ 1 \} }( \gamma )
  \, 
  \beta^2
  \left[ x_{ k - 2 } + x_{ k - 3 } \right] \nonumber 
\\ & 
=
  \left[
    \alpha_k - 2 \beta \alpha_{ k - 1 }
    +
    \beta^2 \alpha_{ k - 2 }
  \right]
  +
    2
    \beta
      x_{ k - 1 }
      +
    \beta
    \left[ 1 - \beta \right]
      x_{ k - 2 }
  -
  \1_{ \{ 0 \} }( \gamma )
  \, 
  \beta^2
  \left[ x_{ k - 2 } + x_{ k - 3 } \right]
  .\nonumber 
\end{align}
This ensures that
for all 
$ k \in \N_0 $ it holds
that
\begin{equation}
  x_k 
=
  \left[
    \alpha_k - 2 \beta \alpha_{ k - 1 }
    +
    \beta^2 \alpha_{ k - 2 }
  \right]
  +
    3
    \beta
      x_{ k - 1 }
      +
    \beta
    \left[ 1 - \beta \right]
      x_{ k - 2 }
  -
  \1_{ \{ 0 \} }( \gamma )
  \, 
  \beta^2
  \left[ x_{ k - 2 } + x_{ k - 3 } \right]
  .
\end{equation}
Combining this, \cref{eq:dynamic_x_2a}, and
\cref{lem:three_term_rec_inhom}
yields that
for all 
$ k \in \N_0 $ 
it holds that
\begin{equation}
\label{eq:xk_recursion}
x_k =
  \begin{cases}
  \sum\limits_{ l = 0 }^k
  \frac{ 
  \big[
    \alpha_l - \beta \alpha_{ l - 1 }
  \big]
  \big(
    \big[ 
        2 \beta
        +
        \sqrt{
          4 \beta^2
          +
          4 \delta \beta
        }
    \big]^{ k + 1 - l }
    -
    \big[ 
        2 \beta
        -
        \sqrt{
          4 \beta^2
          +
          4 \delta \beta
        }
    \big]^{ k + 1 - l }
  \big)
  }{ 
    2^{ ( k + 1 - l ) }
    \sqrt{
      4 \beta^2 + 4  \beta
    }
  }
  &
    \colon \gamma = 0 
  \\ 
  \sum\limits_{ l = 0 }^k
  \frac{ 
  \big(
    \big[ 
        3 \beta
        +
        \sqrt{
          ( 3 \beta )^2
          +
          4 \beta
          -
          4 \beta^2
        }
    \big]^{ k + 1 - l }
    -
    \big[ 
        3 \beta
        -
        \sqrt{
          ( 3 \beta )^2
          +
          4 \beta
          -
          4 \beta^2
        }
    \big]^{ k + 1 - l }
  \big)
  }{ 
  \big[
    \alpha_l - 2 \beta \alpha_{ l - 1 }
    +
    \beta^2 \alpha_{ l - 2 }
  \big]^{-1}
    2^{ ( k + 1 - l ) }
    \sqrt{
      ( 3 \beta )^2 + 4 \beta - 4 \beta^2
    }
  }
  &
    \colon \gamma = 1
  \end{cases}
  .
\end{equation}
This proves that
for all 
$ k \in \N_0 $ it holds that
\begin{equation}
\begin{split}
& 
  x_k 
  =
  \begin{cases}
  \sum\limits_{ l = 0 }^k
  \frac{ 
  \big[
    \alpha_l - \beta \alpha_{ l - 1 }
  \big]
  \big(
    \big[ 
        \beta
        +
        \sqrt{
          \beta^2
          +
          \beta
        }
    \big]^{ k + 1 - l }
    -
    \big[ 
        \beta
        -
        \sqrt{
          \beta^2
          +
          \beta
        }
    \big]^{ k + 1 - l }
  \big)
  }{ 
    2
    \sqrt{
      \beta^2 + \beta
    }
  }
  &
    \colon \gamma = 0 
  \\ 
  \sum\limits_{ l = 0 }^k
  \frac{ 
  \big[
    \alpha_l - 2 \beta \alpha_{ l - 1 }
    +
    \beta^2 \alpha_{ l - 2 }
  \big]
  \big(
    \big[ 
        3 \beta
        +
        \sqrt{
          5 \beta^2
          +
          4 \beta
        }
    \big]^{ k + 1 - l }
    -
    \big[ 
        3 \beta
        -
        \sqrt{
          5 \beta^2
          +
          4 \beta
        }
    \big]^{ k + 1 - l }
  \big)
  }{ 
    2^{ ( k + 1 - l ) }
    \sqrt{
      5 \beta^2 + 4 \beta 
    }
  }
  &
    \colon \gamma = 1
  \end{cases}
  .
\end{split}
\end{equation}
The proof of \cref{lem:full_history}
is thus complete.
\end{proof}

\begin{corollary}
\label{cor:full_history_gamma_1}
Let
$ \gamma \in \{0,1\}$,
$ \beta \in [1,\infty) $,
$
  \alpha_0, \alpha_1, x_0, x_1, x_2, \ldots \in [0,\infty) 
$
satisfy
for all $ k \in \N_0 $
that
\begin{equation}\label{full_history_eq}
  x_k
\leq
  \1_{ \N }( k )
  (\alpha_0+\alpha_1k)\beta^k
  +
  \sum_{ l = 0 }^{ k - 1 }
  \left(k - l\right)^\gamma 
  \beta^{ ( k - l ) }
  \left[
    x_l
    +
    x_{ \max\{ l - 1 , 0 \} }
  \right]
  .
\end{equation}
Then it holds for all 
$ k \in \N_0 $ that
\begin{equation}\label{a_bd}
\begin{split}
x_k \le 
\frac{(\alpha_0 + \alpha_1) \beta^k \1_\N(k)}{(4 + \gamma)^{\nicefrac{1}{2}} (1 + 2^{\nicefrac{(1+\gamma)}{2}})^{-k}}
= 
\begin{cases}
\1_{ \N }( k )(\alpha_0+\alpha_1) 2^{-1} (1+ 2^{\nicefrac{1}{2}} )^k \beta^ k & \colon \gamma = 0 \\[0.5em]
\1_{ \N }( k ) (\alpha_0+\alpha_1) 5^{-\nicefrac{1}{2}} (3\beta) ^ k & \colon \gamma = 1
\end{cases}.
\end{split}
\end{equation}
\end{corollary}

\begin{proof}[Proof 
of \cref{cor:full_history_gamma_1}]
Throughout this proof let 
$ A \colon \N_0 \to [0,\infty) $
satisfy
for all $ k \in \N_0 $
that
$
  A_k
  =
  \1_{ \N }( k )
  (\alpha_0+\alpha_1 k) \beta^k
$.
Note that \cref{full_history_eq} ensures that $x_0 = 0$. This assures that for all $k\in\N_0$ it holds that
\begin{equation}
\begin{split}
x_k
& \leq
  \1_{ \N }( k )
  (\alpha_0+\alpha_1k)\beta^k
  +
  \SmallSum_{ l = 0 }^{ k - 1 }
  \left(k - l\right)^\gamma 
  \beta^{ ( k - l ) }
  \left[
    x_l
    +
    x_{ \max\{ l - 1 , 0 \} }
  \right] \\
& = \1_{ \N }( k )
  (\alpha_0+\alpha_1k)\beta^k
  +
  \SmallSum_{ l = 0 }^{ k - 1 }
  \left(k - l\right)^\gamma 
  \beta^{ ( k - l ) }
  \left[
    x_l
    +
    \1_{ \N }( l )
    \,
    x_{ \max\{ l - 1 , 0 \} }
  \right]
\end{split}
\end{equation}
Next observe that for all $ l \in \N_0$ it holds that 
\begin{equation}\label{a_diff1}
\begin{split}
A_l-\1_{ \N }( l )\beta A_{\max\{l-1,0\}}& =
 \bigl[\1_{ \N }( l )
  (\alpha_0+\alpha_1 l)  - \1_{\N}(l) \1_{\N}(l-1)(\alpha_0+\alpha_1(l-1)) \bigr] \beta^l \\
& = \1_{ \{1\} }( l ) (\alpha_0 + \alpha_1) \beta
\end{split}     
\end{equation}
and
\begin{equation}\label{a_diff2}
\begin{split}
&A_l-\1_{ \N }( l )2\beta A_{\max\{l-1,0\}}+\1_{[2,\infty)}(l)\beta^2A_{\max\{l-2,0\}}
\\&=
\bigl[\1_{\N}(l)(\alpha_0+\alpha_1 l)-\1_{\N}(l)2\1_{\N}(l-1)(\alpha_0+\alpha_1(l-1))+\1_{\N}(l-2)(\alpha_0+\alpha_1(l-2))\bigr] \beta^l
\\&
=\1_{ \{1\} }( l )(\alpha_0+\alpha_1)\beta-\1_{ \{2\} }( l )\alpha_0\beta^2.
\end{split}     
\end{equation}
\cref{lem:full_history} therefore proves that for all $k\in \N_0$ it holds that 
\begin{equation}\label{a_bd1}
\begin{split}
x_k & \le \1_{ \N }( k )(\alpha_0+\alpha_1)\beta \frac{\bigl[\beta+\sqrt{\beta^2+\beta}\bigr]^k-\bigl[\beta-\sqrt{\beta^2+\beta}\bigr]^k}{2\sqrt{\beta^2+\beta}} \\
& \le \1_{ \N }( k )(\alpha_0+\alpha_1)\beta \frac{\bigl[\beta+\sqrt{\beta^2+\beta}\bigr]^k}{2\sqrt{\beta^2+\beta}} 
\le \1_{ \N }( k )\frac{(\alpha_0+\alpha_1)}{2} (1+\sqrt 2)^k \beta ^ k
\end{split}
\end{equation}
and
\begin{equation}\label{a_bd2}
\begin{split}
x_k&\le \1_{ \N }( k )(\alpha_0+\alpha_1)\beta \frac{\left[3\beta+\sqrt{5\beta^2+4\beta}\right]^k-\left[3\beta-\sqrt{5\beta^2+4\beta}\right]^k}{2^k\sqrt{5\beta^2+4\beta}}\\
&\le \1_{ \N }( k )(\alpha_0+\alpha_1)\beta \frac{\left[3\beta+\sqrt{5\beta^2+4\beta}\right]^k}{2^k\sqrt{5\beta^2+4\beta}} 
= \1_{ \N }( k )(\alpha_0+\alpha_1)\frac{\beta^k \left[3+\sqrt{5+\frac 4\beta}\right]^k}{2^k\sqrt{5+\frac 4\beta}}\\
&\le \1_{ \N }( k )\frac{(\alpha_0+\alpha_1)}{\sqrt 5} (3\beta) ^ k.
\end{split}
\end{equation}
Combining \cref{a_bd1} and \cref{a_bd2} hence establishes \cref{a_bd}.
The proof 
of \cref{cor:full_history_gamma_1} is thus complete.
\end{proof}

\subsection{Complexity analysis in the case of stochastic fixed-point equations}\label{sec:4_3}

\begin{proposition}\label{th:final}
Let
$T, \LipConstF, p, q, \alpha, \power, \fb, \fB \in[0,\infty)$, 
$\littleMM_1, \littleMM_2, \littleMM_3, \ldots \in \N$, $\Theta = \bigcup_{n\in\N}\! \Z^n$, 
let $\smallF_d \in C([0,T]\times \R^d \times \R,\R)$, $d\in\N$, and
let $\funcG_d \in C(\R^d,\R)$, $d\in\N$, assume for all $d\in\N$, $t\in[0,T]$, $x = (x_1,x_2,\dots,x_d) \in\R^d$, $v,w\in\R$ that
\begin{equation}\label{th:final_lip_conds}
\liminf_{j\to\infty} \littleMM_j = \infty , \qquad 
\littleMM_{d+1} \le \fB \littleMM_d , \qquad
\abs{ \smallF_d(t,x,v) - \smallF_d(t,x,w) } \le \LipConstF \abs{ v-w } , 
\end{equation}
and $\max\{\abs{ \smallF_d(t,x,0) }, \abs{ \funcG_d(x) } \} \le \LipConstF d^p (1 + \sum_{k=1}^d \abs{ x_k } )^q$, 
let $(\Omega, \cF, \P)$ be a probability space,  
let $\fu^\theta\colon \Omega \to [0,1]$, $\theta\in\Theta$, be i.i.d.\ random variables, 
assume for all $r\in(0,1)$ it holds that $\P(\fu^0 \le r)=r$,
let $\cU^\theta \colon [0,T] \times \Omega \to [0,T]$, $\theta\in\Theta$, satisfy for all $t\in [0,T]$, $\theta\in\Theta$ that $\cU_t^\theta = t + (T-t)\fu^\theta$, let $W^{d,\theta} \colon [0,T] \times  \Omega \to \R^d$, $d\in\N$, $\theta\in\Theta$, be independent standard Brownian motions,
assume for every $d\in\N$ that $(\cU^\theta)_{\theta\in\Theta}$ and $(W^{d,\theta})_{\theta\in\Theta}$ are independent,
let $\smallU_d \in C([0,T]\times \R^d, \R)$, $d\in\N$, satisfy for all $d\in\N$, $t\in[0,T]$, $x\in\R^d$ that
$\E[\abs{ \funcG_d(x + W^{d,0}_{T-t}) } + \int_t^T \abs{ \smallF_d(s,x+W^{d,0}_{s-t},\smallU_d(s,x+W^{d,0}_{s-t})) } \, ds ] < \infty$
and
\begin{equation}
\smallU_d(t,x) = \E\br[\big]{ \funcG_d(x+W^{d,0}_{T-t}) } + \int_t^T \E\br[\big]{ \smallF_d(s,x+W^{d,0}_{s-t},\smallU_d(s,x+W^{d,0}_{s-t})) } \, ds,
\end{equation}
let $\mlp_{n,j}^{d,\theta} \colon [0,T] \times \R^d \times \Omega \to \R$, $d,j,n \in \Z$, $\theta\in\Theta$, satisfy for all $n\in\N_0$, $d,j\in\N$, $\theta\in\Theta$, $t\in[0,T]$, $x\in\R^d$ that 
\begin{align}\label{th:final_mlp}
\mlp_{n,j}^{d.\theta}(t,x) & = \frac{\1_\N(n)}{(\littleMM_{j})^n} \br[\Bigg]{ \sum_{k=1}^{(\littleMM_{j})^n} \funcG_d\pr[\big]{ x + W_{T-t}^{d,(\theta,0,-k)} } } \nonumber \\
& + \sum_{i=0}^{n-1} \frac{(T-t)}{(\littleMM_{j})^{n-i}} \Biggl[ \sum_{k=1}^{(\littleMM_{j})^{n-i}} \Bigl[ \smallF_d\pr[\big]{ \cU_t^{(\theta,i,k)}, x + W_{\cU_t^{(\theta,i,k)}-t}^{d,(\theta,i,k)},\mlp_{i,j}^{d,(\theta,i,k)} \pr[\big]{ \cU_t^{(\theta,i,k)}, x + W_{\cU_t^{(\theta,i,k)}-t}^{d,(\theta,i,k)} } } \\
& - \1_{\N}(i) \, \smallF_d \pr[\big]{ \cU_t^{(\theta,i,k)}, x + W_{\cU_t^{(\theta,i,k)}-t}^{d,(\theta,i,k)}, \mlp_{i-1,j}^{d,(\theta,-i,k)} \pr[\big]{ \cU_t^{(\theta,i,k)}, x + W_{\cU_t^{(\theta,i,k)}-t}^{d,(\theta,i,k)} } } \Bigr] \Biggr], \nonumber
\end{align}
and let $\cost{n}{j}{d} \in \R$, $d,n,j\in\N_0$, satisfy for all $d,j\in\N$, $n\in\N_0$ that 
\begin{equation}\label{fc_def1}
\cost{n}{j}{d} 
\le \1_\N(n) \, \alpha d^\fb (\littleMM_{j})^n 
+ \sum_{k=0}^{n-1} (\littleMM_{j})^{n-k} \pr[\big]{ \alpha d^\fb + \cost{k}{j}{d} + \cost{\max\{k-1,0\}}{j}{d} }.
\end{equation}
Then there exist $\fR \colon \N \times \R \to \N$
and $c = (c_{\fq,\delta})_{(\fq,\delta)\in\R^2} \colon \R^2 \to \R$ 
such that for all $d\in\N$, $\varepsilon,\delta\in(0,\infty)$, $\fq \in [2,\infty)$ with
$\limsup_{j\to\infty} \br[]{\nicefrac{(\littleMM_j)^{\fq/2}}{j}} < \infty$ it holds that
\begin{equation}\label{proof_thus1}
\begin{split}
\pr[]{ \fR(d,\varepsilon) }^{\power} \cost{\fR(d,\varepsilon)}{\fR(d,\varepsilon)}{d}
& \le 
\alpha d^\fb \pr[]{ \fR(d,\varepsilon) }^{\power} (1 + \sqrt{2})^{\fR(d,\varepsilon)} \pr[\big]{ \littleMM_{\fR(d,\varepsilon)} }^{\fR(d,\varepsilon)} \\
& \le
\alpha c_{\fq,\delta} d^{\fb + (p+q)(2+\delta)}(\min\{1,\varepsilon\})^{-(2+\delta)}
\end{split}
\end{equation}
and
\begin{equation}\label{proof_thus2}
\sup_{t\in[0,T]} \sup_{x\in[-\LipConstF,\LipConstF]^d} \pr[\Big]{ \E\br[\Big]{ \abs[\big]{  \smallU_d(t,x) - \mlp_{\fR(d,\varepsilon),\fR(d,\varepsilon)}^{d,0}(t,x) }^\fq } }^{\!\!\nicefrac{1}{\fq}} 
\le \varepsilon.
\end{equation}
\end{proposition}

\begin{proof}[Proof of \cref{th:final}]
Throughout this proof 
let $\fm_\fq = \secondConstant{\fq}$, $\fq \in [2,\infty)$,
let $\mathbb{F}^{d}_t \subseteq \cF$, $d\in\N$, $t\in[0,T]$, satisfy
for all $d\in\N$, $t\in[0,T]$ that
\begin{equation}\label{filtration}
\mathbb{F}_t^{d} = 
\begin{cases}
\bigcap_{s\in(t,T]} \sigmaAlgebra\pr[\big]{ \sigmaAlgebra\pr{ W^{d,0}_r \colon r \in [0,s] } \cup \{A \in \cF \colon \P(A) = 0\} } & \colon t < T \\
\sigmaAlgebra\pr[\big]{ \sigmaAlgebra\pr{ W^{d,0}_s \colon s \in [0,T] } \cup \{A \in \cF \colon \P(A) = 0\} } & \colon t = T
\end{cases},
\end{equation}
let $\Ffn_d \in C([0,T]\times \R^d, \R^d)$, $d\in\N$, satisfy for all $d\in\N$, $t\in[0,T]$, $x\in\R^d$ that
$\Ffn_d(t,x) = 0$,
let $\Gfn_d \in C([0,T]\times\R^d, \R^{d\times d})$, $d\in\N$, satisfy for all $d\in\N$, $t\in[0,T]$, $x,v\in\R^d$ that
$\Gfn_d(t,x) v = v$,
let $\eta_{d,\fq} \in [0,\infty)$, $d\in\N$, $\fq \in [2,\infty)$, satisfy for all $\fq \in [2,\infty)$, $d\in\N$ that
\begin{equation}\label{c_delta}
\eta_{d,\fq} = 
\fm_\fq \LipConstF 2^{\max\{q,1\}} d^{p+q} \pr[\big]{ ( 1 + \LipConstF^2)^{\nicefrac{q}{2}} + (q\fq+1)^{\nicefrac{1}{\fq}} } \exp\pr[\big]{\tfrac{[q(q\fq+3)+1]T}{2} + (\LipConstF+1) T },
\end{equation}
and let $\fR \colon \N \times \R \to [1,\infty]$ satisfy for all $d\in\N$, $\varepsilon\in(0,\infty)$ that
\begin{align}\label{th_r_def}
& \fR(d,\varepsilon) = \\
& \inf \left( \left\{ n \in \N \colon \left[ \sup\left\{ \eta_{d,\fq} \left[ \tfrac{ (1 + 2 \LipConstF T ) \exp\pr[\big]{ \frac{(\littleMM_n)^{\nicefrac{\fq}{2}}}{n} }}{(\littleMM_{n})^{\nicefrac{1}{2}}} \right]^{\!n} 
\le \varepsilon \colon 
\begin{aligned}
& \fq \in [2,\infty) \text{ with } \\
& \limsup_{j\to\infty} \nicefrac{(\littleMM_j)^{\fq/2}}{j} < \infty
\end{aligned} \right\}
\right] \right\} \cup \{\infty\} \right) \nonumber
\end{align}
(cf.\ \cref{rand_const}). 
Observe that \cref{filtration} guarantees that $\mathbb{F}_t^d \subseteq \cF$, $d\in\N$, $t\in[0,T]$, satisfies that
\begin{enumerate}[label=(\Roman*)]
\item\label{filt1a} it holds for all $d\in\N$ that $\{ A\in\cF \colon \P(A) = 0\} \subseteq \mathbb{F}_0^d$ and
\item\label{filt2a} it holds for all $d\in\N$, $t\in[0,T]$ that $\mathbb{F}_t^d = \cap_{s\in(t,T]} \mathbb{F}_s^d$.
\end{enumerate}
Combining \cref{filt1a}, \cref{filt2a}, \cref{filtration}, and Hutzenthaler et al.\ \cite[Lemma 2.17]{HutzenthalerPricing2019} therefore assures that for all $d\in\N$ it holds that $W^{d,0} \colon [0,T] \times \Omega \to \R^d$ is a standard $(\Omega, \cF, \P, (\mathbb{F}_t^{d})_{t\in[0,T]})$-Brownian motion.
In addition, note that \cref{filtration} ensures that for all $d\in\N$, $x\in\R^d$ it holds that $[0,T]\times \Omega \ni (t,\omega) \mapsto x + W_t^{d,0}(\omega) \in \R^d$ is an $(\mathbb{F}_t^{d})_{t\in[0,T]} / \mathcal{B}(\R^d)$-adapted stochastic process with continuous sample paths. This,
the fact that for all $d\in\N$, $t\in[0,T]$, $x\in\R^d$ it holds that $\Ffn_d(t,x) = 0$, and the fact that for all $d\in\N$, $t\in[0,T]$, $x,v\in\R^d$ it holds that $\Gfn_d(t,x)v = v$ yield that for all $d\in\N$, $x\in\R^d$ it holds that $[0,T]\times \Omega \ni (t,\omega) \mapsto x + W_t^{d,0}(\omega) \in \R^d$ satisfies for all $t\in[0,T]$ it holds $\P$-a.s.\ that
\begin{equation}\label{eq:4_9}
x + W^{d,0}_t 
= x + \int_0^t 0 \, ds + \int_0^t dW^{d,0}_s 
= x + \int_0^t \Ffn_d(s,x + W^{d,0}_s) \, ds + \int_0^t \Gfn_d(s,x + W^{d,0}_s) \, dW^{d,0}_s.
\end{equation}
Combining this and Hutzenthaler et al.\ \cite[Lemma 2.6]{HutzenthalerPricing2019} (applied for every $d\in\N$, $x\in\R^d$ with $d \with d$, $m \with d$, $T \with T$, $C_1 \with d$, $C_2 \with 0$, $\mathbb{F} \with \mathbb{F}^{d}$, $\xi \with x$, $\mu \with \Ffn_d$, $\sigma \with \Gfn_d$, $W \with W^{d,0}$, $X \with ([0,T]\times \Omega \ni (t,\omega) \mapsto x + W_t^{d,0}(\omega) \in \R^d)$ in the notation of \cite[Lemma 2.6]{HutzenthalerPricing2019}) ensures that for all $r\in[0,\infty)$, $d\in\N$, $x\in\R^d$, $t\in[0,T]$ it holds that
\begin{equation}\label{eq:4_10}
\E\br[\big]{ \norm{ x + W^{d,0}_t }^r } 
\le \max\{T,1\} \pr[\big]{ ( 1 + \norm{ x }^2 )^{\nicefrac{r}{2}} + (r+1) d^{\nicefrac{r}{2}} } \exp\pr[\big]{ \tfrac{r(r+3)T}{2} } 
< \infty
\end{equation}
(cf.\ \cref{euclid_norm}).
This, the triangle inequality, 
and the fact that for all $v,w\in[0,\infty)$, $r\in(0,1]$ it holds that $(v+w)^r \le v^r + w^r$ assure that for all
$\fq \in [2,\infty)$, 
$d\in\N$, $x\in\R^d$ it holds that 
\begin{equation}\label{e_bd}
\begin{split}
& \sup_{s\in[0,T]}\pr[\big]{ \E\br[\big]{ \pr[]{ 1 + \norm{ x + W^{d,0}_s }^q }^{\fq} } }^{\!\nicefrac{1}{\fq}} 
\le 1 + \sup_{s\in[0,T]} \pr[\big]{ \E\br[\big]{ \norm{ x + W^{d,0}_s }^{q\fq} } }^{\!\nicefrac{1}{\fq}} \\
& \le 1 + \sup_{s\in[0,T]} \pr[\Big]{ \max\{T,1\}\pr[\big]{ ( 1 + \norm{ x }^2)^{\nicefrac{q\fq}{2}} + (p\fq+1) d^{\nicefrac{q\fq}{2}} } \exp\pr[\big]{\tfrac{q\fq(q\fq+3)T}{2}} }^{\!\!\nicefrac{1}{\fq}} \\
& \le 1 + \max\{T^{\nicefrac{1}{\fq}},1\}\pr[\big]{ ( 1 + \norm{ x }^2)^{\nicefrac{q}{2}} + (q\fq+1)^{\nicefrac{1}{\fq}} d^{\nicefrac{q}{2}} } \exp\pr[\big]{\tfrac{q(q\fq+3)T}{2}} \\
& \le 2 \pr[\big]{ ( 1 + \norm{ x }^2)^{\nicefrac{q}{2}} + (q\fq+1)^{\nicefrac{1}{\fq}} d^{\nicefrac{q}{2}} } \exp\pr[\big]{\tfrac{q(q\fq+3)T}{2} + \tfrac{T}{\fq}} \\
& \le 2 \pr[\big]{ ( 1 + \norm{ x }^2)^{\nicefrac{q}{2}} + (q\fq+1)^{\nicefrac{1}{\fq}} d^{\nicefrac{q}{2}} } \exp\pr[\big]{\tfrac{[q(q\fq+3)+1]T}{2}}
< \infty.
\end{split}
\end{equation}
Combining this, \cref{c_delta}, and the fact that for all $d\in\N$, $x\in[-\LipConstF,\LipConstF]^d$ it holds that $\norm{ x } \le \LipConstF d^{\nicefrac{1}{2}}$ demonstrates that for all 
$\fq \in [2,\infty)$,
$d\in\N$ it holds that 
\begin{align}\label{final_c_bd}
& \fm_\fq \LipConstF 2^{\max\{q-1,0\}} d^{p+\nicefrac{q}{2}} (T +1) \exp(\LipConstF T) \left[\sup_{x\in[-\LipConstF,\LipConstF]^d} \sup_{s\in[0,T]} \pr[\big]{ \E\br[\big]{ \pr[]{ 1 + \norm{ x + W^{d,0}_s }^q }^{\fq} } }^{\!\nicefrac{1}{\fq}} \right] \nonumber \\
& \le \fm_\fq \LipConstF 2^{\max\{q-1,0\}} d^{p+\nicefrac{q}{2}} \exp(\LipConstF T + T) \nonumber \\
& \quad \cdot \left[ \sup_{x\in[-\LipConstF,\LipConstF]^d} \br[\Big]{ 2 \pr[\big]{ ( 1 + \norm{ x }^2)^{\nicefrac{q}{2}} + (q\fq+1)^{\nicefrac{1}{\fq}} d^{\nicefrac{q}{2}} } \exp\pr[\big]{\tfrac{[q(q\fq+3)+1]T}{2}} } \right] \\
& \le \fm_\fq \LipConstF 2^{\max\{q,1\}} d^{p+\nicefrac{q}{2}} \pr[\big]{ ( 1 + L^2 d)^{\nicefrac{q}{2}} + (q\fq+1)^{\nicefrac{1}{\fq}} d^{\nicefrac{q}{2}} } \exp\pr[\big]{\tfrac{[q(q\fq+3)+1]T}{2} + (\LipConstF +1)T } \nonumber \\
& \le \fm_\fq \LipConstF 2^{\max\{q,1\}} d^{p+q} \pr[\big]{ ( 1 + \LipConstF^2)^{\nicefrac{q}{2}} + (q\fq+1)^{\nicefrac{1}{\fq}} } \exp\pr[\big]{\tfrac{[q(q\fq+3)+1]T}{2} + (\LipConstF+1) T } = \eta_{d,\fq}
< \infty. \nonumber 
\end{align} 
This and 
\cref{th:final_lip_conds}
guarantee that for all
$\fq \in [2,\infty)$ which satisfy
$\limsup_{j\to\infty} \br[]{\nicefrac{(\littleMM_j)^{\fq/2}}{j}} < \infty$ it holds that
\begin{equation}\label{lim_sup}
\limsup_{n\to\infty} \eta_{d,\fq} \left[ \frac{ (1 + 2 \LipConstF T ) \exp\pr[\big]{ \nicefrac{(\littleMM_n)^{\nicefrac{\fq}{2}}}{n} } }{(\littleMM_{n})^{\nicefrac{1}{2}}} \right]^{\!n} = 0.
\end{equation}
Combining this and \cref{th_r_def} implies
that for all $d\in\N$, $\varepsilon\in(0,\infty)$ it holds that $\fR(d,\varepsilon) \in \N$.
Next observe that the fact that for all $m\in\N$, $r,v_1,v_2,\dots,v_m \in [0,\infty)$ it holds that $[\sum_{k=1}^m v_k]^{r} \le m^{\max\{r-1,0\}}[\sum_{k=1}^m v_k^r]$ and the hypothesis that for all $d\in\N$, $t\in[0,T]$, $x = (x_1,x_2,\ldots,x_d)\in\R^d$ it holds that $\max\{\abs{ \smallF_d(t,x,0) }, \abs{ \funcG_d(x) } \} \le \LipConstF d^p(1 + \sum_{k=1}^d \abs{ x_k } )^q$ ensure that for all $d\in\N$, $x=(x_1,x_2,\ldots,x_d)\in\R^d$ it holds that
\begin{equation}
\begin{split}
& \max\{\abs{ \smallF_d(t,x,0) },\abs{ \funcG_d(x)} \} 
\le \LipConstF d^p \pr[\big]{ 1 + \smallsum_{k=1}^d \abs{ x_k } }^q 
\le \LipConstF 2^{\max\{q-1,0\}} d^p \br[\big]{ 1 + \pr[\big]{ \smallsum_{k=1}^d \abs{ x_k } }^q } \\
& \le \LipConstF 2^{\max\{q-1,0\}} d^p \br[\big]{ 1 + \pr[\big]{ d^{2-1} \smallsum_{k=1}^d \abs{ x_k }^2 }^{\nicefrac{q}{2}} } 
\le \LipConstF 2^{\max\{q-1,0\}} d^{p+\nicefrac{q}{2}} \pr[\big]{ 1 + \norm{ x }^q }.
\end{split}
\end{equation}
This, 
\cref{th:final_lip_conds},
and \cref{mlp_stab_cor} (applied for every $\fq\in[2,\infty)$, $d,j\in\N$ with  
$\littleMM \with \littleMM_{j}$, $\fm \with \fm_\fq$, $p \with q$, $\smallF \with \smallF_d$, $\funcG \with \funcG_d$, $T \with T$, $\LipConstF \with \LipConstF$, $\boundFG \with \LipConstF 2^{\max\{q-1,0\}}d^{p+\nicefrac{q}{2}}$, $\fu^\theta \with \fu^\theta$, $\cU^{\theta} \with \cU^\theta$, $W^\theta \with W^{d,\theta}$, $\smallU \with \smallU_d$, $\mlp_{n}^0 \with \mlp_{n,j}^{d,0}$ in the notation of \cref{mlp_stab_cor})
assure that for all $\fq \in [2,\infty)$, $n\in\N_0$, $d,j\in\N$, $t\in[0,T]$, $x\in\R^d$ it holds that
\begin{align}\label{big_item3}
& \pr[\big]{ \E\br[\big]{ \abs[]{ \smallU_{d}(t,x) - \mlp_{n,j}^{d,0}(t,x) }^\fq } }^{\!\nicefrac{1}{\fq}} \\
& \le \frac{\fm_\fq \LipConstF 2^{\max\{q-1,0\}}d^{p+\nicefrac{q}{2}} (T +1) \exp(\LipConstF T) \pr[\big]{ 1 + 2\LipConstF T }^{\!n} }{(\littleMM_{j})^{\nicefrac{n}{2}} \exp\pr[\big]{-\nicefrac{(\littleMM_{j})^{\nicefrac{\fq}{2}}}{\fq}} } \left[\sup_{s\in[0,T]} \pr[\big]{ \E\br[\big]{ \pr[]{ 1 + \norm{ x + W^{d,0}_s }^q }^{\fq} } }^{\!\nicefrac{1}{\fq}} \right]. \nonumber
\end{align}
Combining this, \cref{c_delta}, \cref{e_bd}, and \cref{final_c_bd}
demonstrates that for all 
$\fq \in [2,\infty)$,
$n\in\N_0$, $d,j\in\N$, $t\in[0,T]$, $x\in[-\LipConstF,\LipConstF]^d$ it holds that
\begin{equation}
\pr[\big]{ \E\br[\big]{ \abs[]{  \smallU_{d}(t,x) - \mlp_{n,j}^{d,0}(t,x) }^\fq } }^{\!\nicefrac{1}{\fq}} 
\le \eta_{d,\fq} \left[\frac{\pr[\big]{ 1 + 2\LipConstF T }^{\!n} \exp\pr[\big]{\nicefrac{(\littleMM_{j})^{\nicefrac{\fq}{2}}}{\fq}}}{(\littleMM_{j})^{\nicefrac{n}{2}}}\right] .
\end{equation}
This, 
\cref{th_r_def}, 
and the fact that for all 
$d\in\N$, $\varepsilon\in(0,\infty)$ it holds that $\fR(d,\varepsilon) \in \N$
prove that for all 
$\fq \in [2,\infty)$ with
$\limsup_{j\to\infty} \br[]{\nicefrac{(\littleMM_j)^{\fq/2}}{j}} < \infty$,
$d\in\N$, $t\in[0,T]$, $x\in[-\LipConstF,\LipConstF]^d$, $\varepsilon\in(0,\infty)$ it holds that
\begin{equation}\label{final_int1}
\begin{split}
\pr[\big]{ \E\br[\big]{ \abs[]{ \smallU_{d}(t,x) - \mlp_{\fR(d,\varepsilon),\fR(d,\varepsilon)}^{d,0}(t,x) }^\fq } }^{\!\nicefrac{1}{\fq}} 
& \le \eta_{d,\fq} \left[\frac{\pr[\big]{ 1 + 2\LipConstF T }^{\!\fR(d,\varepsilon)} \exp\pr[\big]{ (\littleMM_{\fR(d,\varepsilon)})^{\nicefrac{\fq}{2}} } }{(\littleMM_{\fR(d,\varepsilon)})^{\nicefrac{\fR(d,\varepsilon)}{2}}}\right] \\
& \le \eta_{d,\fq} \left[\frac{\pr[\big]{ 1 + 2\LipConstF T } \exp\pr[]{ \nicefrac{ (\littleMM_{\fR(d,\varepsilon)})^{\nicefrac{\fq}{2}} }{ \fR(d,\varepsilon) } } }{ (\littleMM_{\fR(d,\varepsilon)})^{\nicefrac{1}{2}}}\right]^{\! \fR(d,\varepsilon)} 
\le \varepsilon . 
\end{split}
\end{equation}
This establishes \cref{proof_thus1}.
Next note that \cref{fc_def1} implies that for all $d,j\in\N$, $n\in\N_0$ it holds that
\begin{equation}
\begin{split}
\cost{n}{j}{d} 
& \le \1_\N(n) \,\alpha d^\fb (\littleMM_{j})^n + \sum_{k=0}^{n-1} (\littleMM_{j})^{n-k} \alpha d^\fb + \sum_{k=0}^{n-1} (\littleMM_{j})^{n-k} \pr[\big]{ \cost{k}{j}{d} + \cost{\max\{k-1,0\}}{j}{d} } \\
& \le \1_\N(n) \pr[\big]{ \alpha d^\fb + n \alpha d^\fb } (\littleMM_{j})^n + \sum_{k=0}^{n-1} (\littleMM_{j})^{n-k} \pr[\big]{ \cost{k}{j}{d} + \cost{\max\{k-1,0\}}{j}{d} }.
\end{split}
\end{equation}
Combining this and \cref{cor:full_history_gamma_1} (applied for every $d,j\in\N$ with $\gamma \with 0$, $\beta \with \littleMM_{j}$, $\alpha_0 \with \alpha d^\fb$, $\alpha_1 \with \alpha d^\fb$, $(x_n)_{n\in\N_0} \with (\cost{n}{j}{d})_{n\in\N_0}$ in the notation of \cref{cor:full_history_gamma_1}) guarantees that for all $n\in\N_0$, $d,j\in\N$ it holds that
\begin{equation}\label{final_fc1}
\cost{n}{j}{d} \le \1_\N(n) \, \alpha d^\fb (1 + \sqrt{2})^n (\littleMM_{j})^n.
\end{equation}
Furthermore, observe that \cref{th_r_def}, the fact that for all $j\in\N$ it holds that $\littleMM_j \in \N$, and the fact that for all $\fq \in [2,\infty)$, $d\in\N$ it holds that $\eta_{d,\fq} \in [0,\infty)$ ensure that for all 
$\fq \in [2,\infty)$, $d\in\N$, $\varepsilon\in(0,\infty)$ with
$\limsup_{j\to\infty} \br[]{\nicefrac{(\littleMM_j)^{\fq/2}}{j}} < \infty$ and $\fR(d,\varepsilon) \in \N \cap [2,\infty)$ 
it holds that
\begin{equation}
\eta_{d,\fq} \left[\frac{(1 + 2\LipConstF T ) \exp\pr[\big]{ \nicefrac{(\littleMM_{\fR(d,\varepsilon)-1})^{\nicefrac{\fq}{2}}}{(\fR(d,\varepsilon)-1)} } }{(\littleMM_{\fR(d,\varepsilon)-1})^{\nicefrac{1}{2}}}\right]^{\!(\fR(d,\varepsilon)-1)} 
> \min\{1,\varepsilon\}.
\end{equation}
Combining this and \cref{final_fc1}
demonstrates that for all 
$\fq \in [2,\infty)$, $d\in\N$, $\varepsilon,\delta\in(0,\infty)$  with
$\limsup_{j\to\infty} \br[]{\nicefrac{(\littleMM_j)^{\fq/2}}{j}} < \infty$ and $\fR(d,\varepsilon) \in \N \cap [2,\infty)$ 
it holds that
\begin{align}
& 
\pr[]{ \fR(d,\varepsilon) }^{\power} \cost{\fR(d,\varepsilon)}{\fR(d,\varepsilon)}{d}
\le \1_\N(\fR(d,\varepsilon)) \, \alpha d^\fb (1 + \sqrt{2})^{\fR(d,\varepsilon)} \pr[]{ \fR(d,\varepsilon) }^{\power} \pr[\big]{ \littleMM_{\fR(d,\varepsilon)} }^{\! \fR(d,\varepsilon)} \nonumber \\
& \le \frac{ \alpha d^\fb \pr[\big]{ (1 + \sqrt{2})\littleMM_{\fR(d,\varepsilon)} }^{\!\fR(d,\varepsilon)} }{ \pr[]{ \fR(d,\varepsilon) }^{-\power} }\left[\frac{\eta_{d,\fq}}{\min\{1,\varepsilon\}} \left[\frac{(1 + 2\LipConstF T ) \exp\pr[\big]{ \nicefrac{(\littleMM_{\fR(d,\varepsilon)-1})^{\nicefrac{\fq}{2}}}{(\fR(d,\varepsilon)-1)} } }{(\littleMM_{\fR(d,\varepsilon)-1})^{\nicefrac{1}{2}}}\right]^{\!(\fR(d,\varepsilon)-1)} \right]^{\!(2 + \delta)} \nonumber \\
& = \frac{\alpha d^\fb (\eta_{d,\fq})^{(2+\delta)}}{(\min\{1,\varepsilon\})^{(2+\delta)}} \left[ \frac{ \pr[]{ \fR(d,\varepsilon) }^{\power} \pr[\big]{ (1 + 2\LipConstF T ) \exp\pr[\big]{ \nicefrac{(\littleMM_{\fR(d,\varepsilon)-1})^{\nicefrac{\fq}{2}}}{(\fR(d,\varepsilon)-1)} } }^{\!(\fR(d,\varepsilon)-1)(2+\delta)} }{ \pr[\big]{ (1+\sqrt{2})\littleMM_{\fR(d,\varepsilon)} }^{\!-\fR(d,\varepsilon)} (\littleMM_{\fR(d,\varepsilon)-1})^{(\fR(d,\varepsilon)-1)(1+\nicefrac{\delta}{2}))}} \right] \\
& \le \frac{\alpha d^\fb (\eta_{d,\fq})^{(2+\delta)}}{(\min\{1,\varepsilon\})^{(2+\delta)}}\left[\sup_{n\in\N}\left( \frac{ (n+1)^{\power} \pr[\big]{ (1+\sqrt{2}) \littleMM_{n+1} }^{\!(n+1)} \pr[\big]{ (1 + 2\LipConstF T ) \exp\pr[\big]{ \nicefrac{(\littleMM_n)^{\nicefrac{\fq}{2}}}{n} } }^{\!n(2+\delta)} }{(\littleMM_{n})^{n(1+\nicefrac{\delta}{2})}} \right)\right] \nonumber \\
& \le \frac{\alpha d^\fb (1+\sqrt{2}) (\eta_{d,\fq})^{(2+\delta)}}{(\min\{1,\varepsilon\})^{(2+\delta)}} \left[\sup_{n\in\N}\left( \frac{ (n+1)^{\power} \littleMM_{n+1} (\littleMM_{n+1})^{n} \pr[\big]{ (1+\sqrt{2}) (1 + 2\LipConstF T ) \exp\pr[\big]{ \nicefrac{(\littleMM_n)^{\nicefrac{\fq}{2}}}{n} } }  }{(\littleMM_{n})^{n} (\littleMM_{n})^{n(\nicefrac{\delta}{2})}} \right)\right]. \nonumber
\end{align}
This
and 
\cref{th:final_lip_conds}
ensure that for all 
$\fq \in [2,\infty)$, $d\in\N$, $\varepsilon,\delta\in(0,\infty)$  with
$\limsup_{j\to\infty} \br[]{\nicefrac{(\littleMM_j)^{\fq/2}}{j}} < \infty$ and $\fR(d,\varepsilon) \in \N \cap [2,\infty)$ 
it holds that
\begin{align}\label{ccn}
& 
\pr[]{ \fR(d,\varepsilon) }^{\power} \cost{\fR(d,\varepsilon)}{\fR(d,\varepsilon)}{d} \\
& \le \frac{\alpha d^\fb (1+\sqrt{2}) (\eta_{d,\fq})^{(2+\delta)}}{(\min\{1,\varepsilon\})^{(2+\delta)}}\left[\sup_{n\in\N}\left( \frac{ (n+1)^{\power} \fB^{2n+1} \littleMM_1 \pr[\big]{ (1+\sqrt{2}) (1 + 2\LipConstF T ) \exp\pr[\big]{ \nicefrac{(\littleMM_n)^{\nicefrac{\fq}{2}}}{n} }  }^{\!n(2+\delta)} }{(\littleMM_{n})^{n(\nicefrac{\delta}{2})}} \right)\right] \nonumber \\
& \le \frac{\alpha d^\fb \littleMM_1 (1+\sqrt{2}) \fB (\eta_{d,\fq})^{(2+\delta)}}{(\min\{1,\varepsilon\})^{(2+\delta)}} \left[\sup_{n\in\N} \left( \frac{ (n+1)^{\nicefrac{\power}{n}} \pr[\big]{ (1+\sqrt{2})\fB (1 + 2\LipConstF T ) \exp\pr[\big]{ \nicefrac{(\littleMM_n)^{\nicefrac{\fq}{2}}}{n} }  }^{\!(2+\delta)} }{(\littleMM_{n})^{\nicefrac{\delta}{2}}} \right)^{\!\!n}\right] . \nonumber
\end{align}
Moreover, note that \cref{th:final_lip_conds}, \cref{fc_def1}, 
and the fact that for all $\fq\in[2,\infty)$, $d\in\N$ it holds that $\eta_{d,\fq} \in [0,\infty)$ 
ensure that for all 
$\fq \in [2,\infty)$, $d\in\N$, $\varepsilon,\delta\in(0,\infty)$  
with
$\limsup_{j\to\infty} \br[]{\nicefrac{(\littleMM_j)^{\fq/2}}{j}} < \infty$ 
it holds that
\begin{align}\label{cc1}
\cost{1}{1}{d} 
& \le \alpha d^\fb \littleMM_1 + \littleMM_1(\alpha d^\fb + \cost{0}{0}{d} + \cost{0}{0}{d}) 
\le 2\alpha d^\fb \littleMM_1 \le \alpha d^\fb (1 + \sqrt{2}) \fB \littleMM_1 \nonumber \\
& \le \frac{\alpha d^\fb \littleMM_1 (1+\sqrt{2}) \fB (\max\{1,\eta_{d,\fq}\})^{(2+\delta)} }{(\min\{1,\varepsilon\})^{(2+\delta)}  } \\
& \quad \cdot \left[ \sup_{n\in\N} \left( \frac{ (n+1)^{\nicefrac{\power}{n}} \pr[\big]{ (1+\sqrt{2}) \fB (1 + 2\LipConstF T ) \exp\pr[\big]{ \nicefrac{(\littleMM_n)^{\nicefrac{\fq}{2}}}{n} }  }^{\!(2+\delta)} }{(\littleMM_{n})^{\nicefrac{\delta}{2}}} \right)^{\!\!n} \right]. \nonumber 
\end{align}
Combining this and \cref{ccn} demonstrates that for all 
$\fq \in [2,\infty)$, $d\in\N$, $\varepsilon,\delta\in(0,\infty)$  with
$\limsup_{j\to\infty} \br[]{\nicefrac{(\littleMM_j)^{\fq/2}}{j}} < \infty$ 
it holds that
\begin{align}\label{cost_bd_inter}
& 
(\fR(d,\varepsilon))^{\power} \cost{\fR(d,\varepsilon)}{\fR(d,\varepsilon)}{d} \\
& \le \frac{\alpha d^\fb \littleMM_1 (1+\sqrt{2}) \fB (\max\{1,\eta_{d,\fq}\})^{(2+\delta)}}{(\min\{1,\varepsilon\})^{(2+\delta)}}\left[ \sup_{n\in\N} \left( \frac{(n+1)^{\nicefrac{\power}{n}}\pr[\big]{ \fB (1 + 2\LipConstF T ) \exp\pr[\big]{ \nicefrac{(\littleMM_n)^{\nicefrac{\fq}{2}}}{n} }  }^{\!(2+\delta)} }{(1+\sqrt{2})^{-1} (\littleMM_{n})^{\nicefrac{\delta}{2}}} \right)^{\!\!n}\right]. \nonumber
\end{align}
This, the fact that $\littleMM_1 \in \N$, \cref{final_c_bd}, \cref{lim_sup}, and \cref{ccn} prove that for all 
$\fq \in [2,\infty)$, $d\in\N$, $\varepsilon,\delta\in(0,\infty)$ with
$\limsup_{j\to\infty} \br[]{\nicefrac{(\littleMM_j)^{\fq/2}}{j}} < \infty$ it holds that
\begin{align}\label{last_this}
& 
(\fR(d,\varepsilon))^{\power} \cost{\fR(d,\varepsilon)}{\fR(d,\varepsilon)}{d} \nonumber \\
& \le \frac{\alpha d^\fb \littleMM_1 (1+\sqrt{2}) \fB }{(\min\{1,\varepsilon\})^{(2+\delta)}} \left[\sup_{n\in\N} \left( \frac{(n+1)^{\nicefrac{\power}{n}}\pr[\big]{ (1+\sqrt{2})\fB (1 + 2\LipConstF T ) \exp\pr[\big]{ \nicefrac{(\littleMM_n)^{\nicefrac{\fq}{2}}}{n} } }^{\!(2+\delta)} }{(\littleMM_{n})^{\nicefrac{\delta}{2}}} \right)^{\!\!n}\right] \\
& \cdot \Bigl( \max\left\{ 1, \fm_\fq \LipConstF 2^{\max\{q,1\}} d^{p+q} \pr[\big]{ ( 1 + \LipConstF^2)^{\nicefrac{q}{2}} + (q\fq+1)^{\nicefrac{1}{\fq}} } \exp\pr[\big]{\tfrac{[q(q\fq+3)+1]T}{2} + (\LipConstF+1) T } \right\} \Bigr)^{\!(2+\delta)} \nonumber \\
& = \frac{\alpha \littleMM_1 \fB d^{\fb + (p+q)(2+\delta)}}{ (1+\sqrt{2})^{-1} (\min\{1,\varepsilon\})^{(2+\delta)}} \left[ \sup_{n\in\N} \left( \frac{(n+1)^{\nicefrac{\power}{n}}\pr[\big]{ (1+\sqrt{2})\fB (1 + 2\LipConstF T ) \exp\pr[\big]{ \nicefrac{(\littleMM_n)^{\nicefrac{\fq}{2}}}{n} } }^{\!(2+\delta)} }{(\littleMM_{n})^{\nicefrac{\delta}{2}}} \right)^{\!\!n}\right] \nonumber \\
& \quad \cdot \Bigl( \max\left\{ 1, \fm_\fq \LipConstF 2^{\max\{q,1\}} \pr[\big]{ ( 1 + \LipConstF^2)^{\nicefrac{q}{2}} + (q\fq+1)^{\nicefrac{1}{\fq}} } \exp\pr[\big]{\tfrac{[q(q\fq+3)+1]T}{2} + (\LipConstF+1) T } \right\} \Bigr)^{\!(2 + \delta)} 
< \infty. \nonumber
\end{align}
Combining \cref{final_fc1} and \cref{last_this} hence establishes \cref{proof_thus2}.
The proof of \cref{th:final} is thus complete.
\end{proof}

\subsection{Complexity analysis in the case of semilinear partial differential equations}\label{sec:4_4}

\begin{lemma}\label{lem:fn_prop}
Let $p \in (0,\infty)$ and let $\funcM \colon \R \to \N$ satisfy for all $x \in [1,\infty)$ that $\funcM(x) = \funcMrep{x}$. Then
\begin{enumerate}[label=(\roman*)]
\item\label{fn_prop_i1} it holds that
$
\limsup_{x \to \infty} \br[\big]{ \tfrac{( \funcM(x) )^p }{x} + \tfrac{1}{\funcM(x)} } = 0
$ and
\item\label{fn_prop_i2} it holds for all $x \in \N$ that $\funcM(x+1) \le 2 \funcM(x)$.
\end{enumerate}
\end{lemma}

\begin{proof}[Proof of \cref{lem:fn_prop}]
Throughout this proof let $\psi \colon \R \to \R$ satisfy for all $x\in[1,\infty)$ that $\psi(x) = \exp(\abs{\ln(x)}^{1/2})$.
Note that the fact that for all $x \in [1,\infty)$ it holds that $\ln(x) \in [0,\infty)$ assures that for all $x\in(1,\infty)$ it holds that
\begin{equation}\label{eq:fn_1_der}
\tfrac{d}{dx}(\psi(x))^p = \frac{p\pr[\big]{\exp(\abs{\ln(x)}^{1/2})}^p}{2x\abs{\ln(x)}^{1/2}}.
\end{equation}
This and the fact that for all $x \in [1,\infty)$ it holds that $\ln(x) \in [0,\infty)$ ensure that for all $x\in(1,\infty)$ it holds that
\begin{equation}\label{eq:fn_2_der}
\tfrac{d^2}{dx^2} (\psi(x))^p = - \frac{ p \pr[\big]{\exp(\abs{\ln(x)}^{1/2})}^p \br[\big]{ 2\ln(x) - p \abs{\ln(x)}^{1/2} + 1 } }{ 4 x^2 \abs{\ln(x)}^{3/2} }.
\end{equation}
Combining this and \cref{eq:fn_1_der} shows that 
$(\exp([p + \sqrt{ \max\{ 0, p^2-8 \} }]/4 ),\infty) \ni x \mapsto \frac{d}{dx} (\psi(x))^p \in \R$ is decreasing.
This, the fact that \cref{eq:fn_1_der} implies that for all $x\in(1,\infty)$ it holds that $\frac{d}{dx}(\psi(x))^p \in [0,\infty)$, and L'H{\^o}pital's rule establish that
\begin{equation}
\limsup_{x\to\infty} \tfrac{( \psi(x) )^p }{x} = \lim_{x\to\infty} \tfrac{( \psi(x) )^p }{x} = \lim_{x\to\infty} \tfrac{d}{dx} (\psi(x))^p = 0.
\end{equation}
Combining this, the fact that for all $x\in[1,\infty)$ it holds that $\funcM(x) \le \psi(x)$, the fact that for all $x\in[1,\infty)$ it holds that $\funcM(x) \in \N$, and the fact that
for all $x\in[1,\infty)$ it holds that $\funcM$ is non-decreasing proves 
that
\begin{equation}
\limsup_{x \to \infty} \br[\big]{ \tfrac{( \funcM(x) )^p }{x} + \tfrac{1}{\funcM(x)} } \le \limsup_{x\to\infty} \tfrac{(\funcM(x))^p}{x} + \limsup_{x\to\infty} \tfrac{1}{\funcM(x)}
\le \limsup_{x\to\infty} \tfrac{(\psi(x))^p}{x} + \limsup_{x\to\infty} \tfrac{1}{\funcM(x)} = 0.
\end{equation}
This establishes
\cref{fn_prop_i1}.
Next note that 
for all $x\in(1,\infty)$ it holds that
\begin{equation}\label{eq:fn_quot}
\frac{\funcM(x+1)}{\funcM(x)} = \frac{\funcMrep{1+x}}{\funcMrep{x}} \le \frac{\exp(\abs{\ln(1+x)}^{1/2})}{\exp(\abs{\ln(x)}^{1/2})-1} .
\end{equation}
In addition, observe that the fact that for all $x\in(1,\infty)$ it holds that $\ln(x) \in (0,\infty)$ demonstrates that for all $x\in(1,\infty)$ it holds that
\begin{equation}
\begin{split}
& \frac{d}{dx} \frac{\exp((\ln(1+x))^{1/2})}{\exp((\ln(x))^{1/2})-1} \\
& = \frac{\exp(\abs{\ln(x+1)}^{1/2})}{2[\exp(\abs{\ln(x)}^{1/2}) - 1]} \left[ \frac{1}{(1+x)\abs{\ln(x+1)}^{1/2}} - \frac{\exp(\abs{\ln(x)}^{1/2})}{x[\exp(\abs{\ln(x)}^{1/2})-1] \abs{\ln(x)}^{1/2}} \right] \le 0.
\end{split}
\end{equation}
This implies that for all $x\in[3,\infty)$ it holds that
\begin{equation}
\frac{\exp(\abs{\ln(1+x)}^{1/2})}{\exp(\abs{\ln(x)}^{1/2})-1} \le \frac{\exp(\abs{\ln(4)}^{1/2})}{\exp(\abs{\ln(3)}^{1/2})-1} \le 2.
\end{equation}
Combining this and the fact that 
\begin{equation}
\frac{\funcM(2)}{\funcM(1)} = \frac{\funcMrep{2}}{\funcMrep{1}} \le \exp(\abs{\ln(2)}^{1/2}) \le \exp(\ln(2)) = 2
\end{equation}
proves that for all $x\in\N$ it holds that $\funcM(x+1) \le 2 \funcM(x)$.
This establishes \cref{fn_prop_i2}. The proof of \cref{lem:fn_prop} is thus complete.
\end{proof}

\begin{theorem}\label{cor:final}
Let $p, q, T, \kappa, \fb \in [0,\infty)$,
$\Theta = \bigcup_{n\in\N}\! \Z^n$, 
$\smallF \in C(\R,\R)$, 
let $\smallU_d \in C^{1,2}([0,T]\times\R^d,\R)$, $d\in\N$,
assume for all $d\in\N$, $t\in[0,T]$, $x=(x_1,x_2,\dots,x_d)\in\R^d$, $v,w\in\R$ that 
$\abs{ \smallF(v) - \smallF(w) } \le \kappa\abs{ v-w }$, 
$\abs{ \smallU_d(t,x) } \le \kappa d^p (1 + \sum_{k=1}^d \abs{ x_k })^q$, and
\begin{equation}\label{pde_form}
(\tfrac{\partial}{\partial t}\smallU_d)(t,x) + (\Delta_x \smallU_d)(t,x) + \smallF(\smallU_d(t,x)) = 0,
\end{equation}
let $(\Omega,\cF,\P)$ be a probability space,  
let $\fu^\theta\colon \Omega \to [0,1]$, $\theta\in\Theta$, be i.i.d.\ random variables, 
assume for all $r \in (0,1)$ that $\P(\fu^0 \le r) = r$,
let $\cU^\theta \colon [0,T]\times \Omega\to [0,T]$, $\theta\in\Theta$, satisfy for all $t\in [0,T]$, $\theta\in\Theta$ that $\cU_t^\theta = t + (T-t)\fu^\theta$, 
let $W^{d,\theta} \colon [0,T] \times \Omega \to  \R^d$, $\theta \in \Theta$, $d\in\N$, be independent standard Brownian motions,
assume for all $d\in\N$ that $(\cU^\theta)_{\theta\in\Theta}$ and $(W^{d,\theta})_{\theta\in\Theta}$ are independent,
let $\funcM \colon \N \to \N$ and $\mlp_{n,\littleMM}^{d,\theta} \colon [0,T] \times \R^d\times \Omega\to \R$, $d,n,\littleMM\in\Z$, $\theta\in\Theta$, satisfy for all $n\in\N_0$, $d,\littleMM\in\N$, $\theta\in\Theta$, $t\in[0,T]$, $x\in\R^d$ that $\funcM(\littleMM) = \funcMrep{\littleMM}$ and
\begin{align}\label{cor:mlp_form}
& \mlp_{n,\littleMM}^{d,\theta}(t,x) 
=  \SmallSum_{i=0}^{n-1} \tfrac{(T-t)}{(\funcM(\littleMM))^{n-i}} \biggl[ \SmallSum_{k=1}^{(\funcM(\littleMM))^{n-i}} \Bigl[ \smallF\pr[\big]{ \mlp_{i,\littleMM}^{d,(\theta,i,k)}\pr[\big]{\cU_t^{(\theta,i,k)}, x + \sqrt{2}\, W_{\cU_t^{(\theta,i,k)}-t}^{d,(\theta,i,k)} } } \\
& - \1_\N(i) \, \smallF\pr[\big]{ \mlp_{i-1,\littleMM}^{d,(\theta,-i,k)} \pr[\big]{ \cU_t^{(\theta,i,k)}, x + \sqrt{2}\, W_{\cU_t^{(\theta,i,k)}-t}^{d,(\theta,i,k)} } } \Bigr] \biggr] + \tfrac{\1_\N(n)}{(\funcM(\littleMM))^n} \br[\bigg]{ \SmallSum_{k=1}^{(\funcM(\littleMM))^n} \smallU_d\pr[\big]{ T,x + \sqrt{2}\, W_{T-t}^{d,(\theta,0,-k)} } }, \nonumber
\end{align}
and let $\cost{n}{\littleMM}{d} \in \R$, $d,n,\littleMM\in\N_0$, satisfy for all $d,\littleMM\in\N$, $n\in\N_0$ that 
\begin{equation}\label{fc_def2}
\cost{n}{\littleMM}{d} 
\le \1_\N(n) \, \kappa d^\fb (\funcM(\littleMM))^n 
+ \sum_{k=0}^{n-1} (\funcM(\littleMM))^{n-k} \pr[\big]{ \kappa d^\fb + \cost{k}{\littleMM}{d} + \cost{\max\{k-1,0\}}{\littleMM}{d} }.
\end{equation}
Then there exist 
$\fR \colon \N \times \R \to \N$
and
$c = (c_{\fq,\delta})_{(\fq,\delta)\in\R^2} \colon \R^2 \to \R$ such that for all $d\in\N$, $\varepsilon,\delta,\fq \in(0,\infty)$
it holds that $\cost{\fR(d,\varepsilon)}{\fR(d,\varepsilon)}{d} \le c_{\fq,\delta} d^{\fb + (p+q)(2+\delta)}(\min\{1,\varepsilon\})^{-(2+\delta)}$ and
\begin{equation}\label{cor_f_eq}
\sup_{t\in[0,T]} \sup_{x\in[-\sqrt{2}\kappa,\sqrt{2}\kappa]^d} \pr[\Big]{ \E\br[\Big]{ \abs[\big]{  \smallU_d(t,x) - \mlp_{\fR(d,\varepsilon),\fR(d,\varepsilon)}^{d,0}(t,x) }^\fq } }^{\!\!\nicefrac{1}{\fq}} 
\le \varepsilon.
\end{equation}
\end{theorem}

\begin{proof}[Proof of \cref{cor:final}]
Throughout this proof 
let $\maxfn \colon (0,\infty) \to [2,\infty)$ satisfy for all $z\in(0,\infty)$ that $\maxfn(z) = \max\{2,z\}$, let $\fD = \max\{2^{\nicefrac{q}{2}} \kappa,\abs{f(0)}\}$, 
let $\mathbb{F}^{d}_t \subseteq \cF$, $d\in\N$, $t\in[0,T]$, satisfy
for all $d\in\N$, $t\in[0,T]$ that
\begin{equation}\label{filtration11}
\mathbb{F}_t^{d} = 
\begin{cases}
\bigcap_{s\in(t,T]} \sigmaAlgebra\pr[\big]{ \sigmaAlgebra\pr{ W^{d,0}_r \colon r \in [0,s] } \cup \{A \in \cF \colon \P(A) = 0\} } & \colon t < T \\
\sigmaAlgebra\pr[\big]{ \sigmaAlgebra\pr{ W^{d,0}_s \colon s \in [0,T] } \cup \{A \in \cF \colon \P(A) = 0\} } & \colon t = T
\end{cases},
\end{equation}
let $\Ffn_d \in C([0,T]\times \R^d, \R^d)$, $d\in\N$, and $\Gfn_d \in C([0,T]\times\R^d, \R^{d\times d})$, $d\in\N$, satisfy for all $d\in\N$, $t\in[0,T]$, $x,v\in\R^d$ that
$\Ffn_d(t,x) = 0$ and $\Gfn_d(t,x) v = \sqrt{2}\, v$,
let $\Msymb_1,\Msymb_2,\Msymb_3,\ldots \in \N$ satisfy for all $j\in\N$ that $\Msymb_j = \funcM(j)$,
let $V_d \colon [0,T] \times \R^d \to \R$, $d\in\N$, satisfy for all $d\in\N$, $t\in[0,T]$, $x\in\R^d$ that $V_d(t,x) = \smallU_d(t,\sqrt{2}\,x)$,
let $F_d \colon [0,T] \times \R^d\times \R \to \R$, $d\in\N$, satisfy for all $d\in\N$, $t\in[0,T]$, $x\in\R^d$, $w \in \R$ that $F_d(t,x,w) = \smallF(w)$,
and let $\fV_{n,j}^{d,\theta} \colon [0,T] \times \R^d \times \Omega \to \R$, $d,n,j \in \Z$, $\theta \in \Theta$, satisfy for all $d,n,j\in\Z$, $\theta \in\Theta$, $t\in[0,T]$, $x\in\R^d$ that $\fV_{n,j}^{d,\theta}(t,x) = \mlp_{n,j}^{d,\theta}(t,\sqrt{2}\,x)$.
Note that the hypothesis that for all $v,w\in\R$ it holds that $\abs{\smallF(v)-\smallF(w)} \le \kappa\abs{v-w}$ assures that for all $d\in\N$, $t\in[0,T]$, $x\in\R^d$, $v,w\in\R$ it holds that
\begin{equation}\label{cor1}
\abs{F_d(t,x,v)-F_d(t,x,w)} = \abs{f(v)-f(w)} \le \kappa\abs{v-w}.
\end{equation}
Next observe that the
hypothesis that for all $d\in\N$, $t\in[0,T]$, $x\in\R^d$ it holds that $\abs{ \smallU_d(t,x) } \le \kappa d^p \allowbreak (1 \allowbreak + \allowbreak \sum_{k=1}^d \abs{ x_k })^p$ ensures that for all $d\in\N$, $t\in[0,T]$, $x = (x_1,x_2,\dots,x_d) \in\R^d$ it holds that
\begin{equation}\label{cor2}
\begin{split}
& \max\{\abs{ F_d(t,x,0) }, \abs{ V_d(t,x) } \} 
= \max\{\abs{ \smallF(0) },\abs{ \smallU_d(t,\sqrt{2}\,x) }\} \\
& \le \max\{ \abs{\smallF(0)},\kappa d^p (1 + \smallsum_{k=1}^d \abs{\sqrt{2}\, x_k } )^q \} 
\le \fD d^p (1 + \smallsum_{k=1}^d \abs{ x_k } )^q.
\end{split}
\end{equation}
In addition, note that \cref{filtration11} guarantees that $\mathbb{F}_t^d \subseteq \cF$, $d\in\N$, $t\in[0,T]$, satisfies that
\begin{enumerate}[label=(\Roman*)]
\item\label{filt11a} it holds for all $d\in\N$ that $\{ A\in\cF \colon \P(A) = 0\} \subseteq \mathbb{F}_0^d$ and
\item\label{filt22a} it holds for all $d\in\N$, $t\in[0,T]$ that $\mathbb{F}_t^d = \cap_{s\in(t,T]} \mathbb{F}_s^d$.
\end{enumerate}
Combining \cref{filt11a,filt22a}, \cref{filtration11}, and Hutzenthaler et al.\ \cite[Lemma 2.17]{HutzenthalerPricing2019} 
(applied with $m \with d$, $T \with T$, $W \with W$, $\mathbb{H}_t \with \mathbb{F}_t^d$, $(\Omega,\cF,\P,(\mathbb{F}_t)_{t\in[0,T]}) \with (\Omega,\cF,\P,(\sigmaAlgebra(\fwprr_s^{d,0} \colon s \in [0,t]) \cup \{ A\in\cF \colon \P(A) = 0\} )_{t\in[0,T]})$ in the notation of \cite[Lemma 2.17]{HutzenthalerPricing2019}) therefore
assures that for all $d\in\N$ it holds that $W^{d,0} \colon [0,T] \times \Omega \to \R^d$ is a standard $(\Omega,\cF,\P,(\mathbb{F}_t^d)_{t\in[0,T]})$-Brownian motion.
This, the hypothesis that for all $t\in[0,T]$, $x\in\R^d$, $v,w\in\R$ it holds that $\abs{ \smallF(v) - \smallF(w) } \le \kappa \abs{ v-w }$, the hypothesis that for all $d\in\N$, $t\in[0,T]$, $x\in\R^d$ it holds that $\abs{ \smallU_d(t,x) } \le \kappa d^p (1 + \sum_{k=1}^d \abs{ x_k })^q$,
and, e.g., Beck et al.\ \cite[Corollary 3.9]{beck2020nonlinear}
(applied for every $d\in\N$ with $d \with d$, $m \with d$, $T \with T$, $\LipConstF \with \max\{\sqrt{2d}, \kappa\}$, $\mathfrak{C} \with 0$, $f \with f$, $g \with (\R^d \ni x \mapsto \smallU_d(T,x) \in \R)$, $\mu \with \Ffn_d$, $\sigma \with \Gfn_d$, $W^{d,0} \with W^{d,0}$, $(\Omega, \cF, \P, (\mathbb{F}_t)_{t\in[0,T]}) \with (\Omega, \cF, \P, (\mathbb{F}_t^d)_{t\in[0,T]})$ in the notation of \cite[Corollary 3.9]{beck2020nonlinear})
ensure that for all $d\in\N$, $t\in[0,T]$, $x\in\R^d$ it holds that
$\E[\abs{\smallU_d(T,x+\sqrt{2}\,W_{T-t}^{d,0}) } + \int_t^T \abs{ \smallF(\smallU_d(s,x+\sqrt{2}\,W_{s-t}^{d,0}))} \, ds ] < \infty$
and
\begin{equation}\label{cor3}
u_d(t,x) = \E\br[\big]{ \smallU_d(T,x+\sqrt{2}\,W^{d,0}_{T-t}) } + \int_t^T \E\br[\big]{ \smallF(\smallU_d(s,x+\sqrt{2}\,W^{d,0}_{s-t})) } \, ds.
\end{equation}
Combining this, the fact that for all $d\in\N$, $t\in[0,T]$, $x\in\R^d$ it holds that $V_d(t,x) = \smallU_d(t,\sqrt{2}\,x)$, and the fact that for all $d\in\N$, $t\in[0,T]$, $x\in\R^d$, $w\in\R$ it holds that $F_d(t,x,w) = \smallF(w)$ demonstrates that for all $d\in\N$, $t\in[0,T]$, $x\in\R^d$ it holds that
\begin{equation}\label{cor4}
\begin{split}
& \E\left[\abs{V_d(T,x+W_{T-t}^{d,0}) } + \int_t^T \abs{ F_d(s,x+W_{s-t}^{d,0},V_d(s,x+W_{s-t}^{d,0}))} \, ds \right] \\
& = \E\left[\abs{V_d(T,x+W_{T-t}^{d,0}) } + \int_t^T \abs{ \smallF(V_d(s,x+W_{s-t}^{d,0}))} \, ds \right] \\
& = \E\left[\abs{\smallU_d(T,\sqrt{2}(x+W_{T-t}^{d,0})) } + \int_t^T \abs{ \smallF(\smallU_d(s,\sqrt{2}(x+W_{s-t}^{d,0})))} \, ds \right] < \infty
\end{split}
\end{equation}
and
\begin{align}\label{cor5}
V_d(t,x) = \smallU_d(t,\sqrt{2}\,x)
& = \E\br[\big]{ \smallU_d(T,\sqrt{2}(x+W^{d,0}_{T-t})) } + \int_t^T \E\br[\big]{ \smallF(\smallU_d(s,\sqrt{2}(x+W^{d,0}_{s-t}))) } \, ds \nonumber \\
& = \E\br[\big]{ V_d(T,x+W^{d,0}_{T-t}) } + \int_t^T \E\br[\big]{ \smallF(V_d(s,x+W^{d,0}_{s-t})) } \, ds \\
& = \E\br[\big]{ V_d(T,x+W^{d,0}_{T-t}) } + \int_t^T \E\br[\big]{ F_d(s,x+W_{s-t}^{d,0},V_d(s,x+W^{d,0}_{s-t})) } \, ds. \nonumber 
\end{align}
Moreover, note that the fact that for all $j\in\N$ it holds that $\Msymb_j = \funcM(j)$
and \cref{cor:mlp_form} show that for all $n\in\N_0$, $d,j\in\N$, $\theta \in\Theta$, $t\in[0,T]$, $x\in\R^d$ it holds that
\begin{align}\label{4_35}
& 
\mlp_{n,j}^{d,\theta}(t,\sqrt{2}\,x) 
= \SmallSum_{i=0}^{n-1} \tfrac{(T-t)}{(\Msymb_j)^{n-i}} \biggl[ \SmallSum_{k=1}^{(\Msymb_j)^{n-i}} \Bigl[ \smallF\pr[\big]{ \mlp_{i,j}^{d,(\theta,i,k)}\pr[\big]{\cU_t^{(\theta,i,k)}, \sqrt{2}(x +  W_{\cU_t^{(\theta,i,k)}-t}^{d,(\theta,i,k)}) } } \\
& - \1_\N(i) \, \smallF\pr[\big]{ \mlp_{i-1,j}^{d,(\theta,-i,k)} \pr[\big]{ \cU_t^{(\theta,i,k)}, \sqrt{2}(x +  W_{\cU_t^{(\theta,i,k)}-t}^{d,(\theta,i,k)}) } } \Bigr] \biggr]
+ \br[\bigg]{ \SmallSum_{k=1}^{(\Msymb_j)^n}  \frac{\1_\N(n) \, \smallU_d\pr[]{ T,\sqrt{2} (x + W_{T-t}^{d,(\theta,0,-k)}) }}{(\Msymb_j)^n} } .\nonumber
\end{align}
This and the fact that for all $d,n,j\in\Z$, $\theta \in\Theta$, $t\in[0,T]$, $x\in\R^d$ it holds that $\fV_{n,j}^{d,\theta}(t,x) = \mlp_{n,j}^{d,\theta}(t,\sqrt{2}\,x)$ imply that for all $n\in\N_0$, $d,j\in\N$, $\theta \in\Theta$, $t\in[0,T]$, $x\in\R^d$ it holds that
\begin{align}
\mlp_{n,j}^{d,\theta}(t,\sqrt{2}\,x) 
& = \fV_{n,j}^{d,\theta}(t,x) \nonumber \\
& = \SmallSum_{i=0}^{n-1} \tfrac{(T-t)}{(\Msymb_j)^{n-i}} \Biggl[ \SmallSum_{k=1}^{(\Msymb_j)^{n-i}} \Bigl[ \smallF\pr[\big]{ \fV_{i,j}^{d,(\theta,i,k)}\pr[\big]{\cU_t^{(\theta,i,k)}, x +  W_{\cU_t^{(\theta,i,k)}-t}^{d,(\theta,i,k)} } } \\
& \quad - \1_\N(i) \, \smallF\pr[\big]{ \fV_{i-1,j}^{d,(\theta,-i,k)} \pr[\big]{ \cU_t^{(\theta,i,k)}, x +  W_{\cU_t^{(\theta,i,k)}-t}^{d,(\theta,i,k)} } } \Bigr] \Biggr] +  \br[\Bigg]{ \SmallSum_{k=1}^{(\Msymb_j)^n} \frac{\1_\N(n) \, V_d\pr[]{ T,x + W_{T-t}^{d,(\theta,0,-k)} }}{(\Msymb_j)^n} }. \nonumber
\end{align}
Combining this and the fact that for all $d\in\N$, $t\in[0,T]$, $x\in\R^d$, $w\in\R$ it holds that $F_d(t,x,w) = \smallF(w)$ yields that for all $n\in\N_0$, $d,j\in\N$, $\theta \in\Theta$, $t\in[0,T]$, $x\in\R^d$ it holds that
\begin{align}\label{cor6}
\fV_{n,j}^{d,\theta}(t,x) 
& = \frac{\1_\N(n)}{(\Msymb_j)^n} \br[\Bigg]{ \sum_{k=1}^{(\Msymb_j)^n} V_d\pr[\big]{ T,x + W_{T-t}^{d,(\theta,0,-k)} } } \\
& \quad + \sum_{i=0}^{n-1} \frac{(T-t)}{(\Msymb_j)^{n-i}} \Biggl[ \sum_{k=1}^{(\Msymb_j)^{n-i}} \Bigl[ F_d\pr[\big]{\cU_t^{(\theta,i,k)}, x +  W_{\cU_t^{(\theta,i,k)}-t}^{d,(\theta,i,k)}, \fV_{i,j}^{d,(\theta,i,k)}\pr[\big]{\cU_t^{(\theta,i,k)}, x +  W_{\cU_t^{(\theta,i,k)}-t}^{d,(\theta,i,k)} } } \nonumber \\
& \quad - \1_\N(i) \, F_d\pr[\big]{\cU_t^{(\theta,i,k)}, x +  W_{\cU_t^{(\theta,i,k)}-t}^{d,(\theta,i,k)}, \fV_{i-1,j}^{d,(\theta,-i,k)} \pr[\big]{ \cU_t^{(\theta,i,k)}, x +  W_{\cU_t^{(\theta,i,k)}-t}^{d,(\theta,i,k)} } } \Bigr] \Biggr]. \nonumber
\end{align}
Furthermore, observe that \cref{lem:fn_prop} and the fact that for all $\littleMM \in \N$ it holds that $\funcM(\littleMM) = \funcMrep{\littleMM}$ 
imply that 
\begin{enumerate}[label=(\Alph *)]
\item\label{alph1}
it holds
for all $\fq \in (0,\infty)$
that $\limsup_{j\to\infty} \br[]{ \nicefrac{(\Msymb_j)^{m(\fq)/2}}{j} + \nicefrac{1}{\Msymb_j} } = \limsup_{j\to\infty} \br[]{ \nicefrac{(\funcM(j))^{m(\fq)/2}}{j} + \nicefrac{1}{\funcM(j)} } = 0$ and
\item\label{alph2}
it holds for all $j\in\N$ that $\Msymb_{j+1} = \funcM(j+1) \le 2 \funcM(j) = 2 \Msymb_j$.
\end{enumerate}
Combining \cref{alph1,alph2}, \cref{cor1}, \cref{cor2}, \cref{cor4}, \cref{cor5}, \cref{cor6},
and \cref{th:final} (applied with $\fa \with \kappa$, $\fb \with \fb$, $\alpha \with \kappa$, $\beta \with 0$, $p \with p$, $q \with q$, $\fB \with 2$, $\LipConstF \with \fD$, $T \with T$, 
$(\littleMM_j)_{j\in\N} \with (\Msymb_j)_{j\in\N}$, $\funcM \with \funcM$, $f_d \with F_d$, $g_d \with (\R^d \ni x \mapsto V_d(T,x) \in \R)$, $u_d \with V_d$, $(\Omega,\cF,\P) \with (\Omega,\cF,\P)$, $\fu^\theta \with \fu^\theta$, $\cU^\theta \with \cU^\theta$, $W^{d,\theta} \with W^{d,\theta}$, $U_{n,j}^{d,\theta} \with \fV_{n,j}^{d,\theta}$ in the notation of \cref{th:final}) hence guarantees that there exists 
$\fR \colon \N \times \R \to \N$
and
$c = (c_{\fq,\delta})_{(\fq,\delta)\in\R^2} \colon \R^2 \to \R$ such that for all $d\in\N$, $\varepsilon,\delta \in(0,\infty)$, $\fq \in [2,\infty)$
it holds that $\cost{\fR(d,\varepsilon)}{\fR(d,\varepsilon)}{d} \le c_{\fq,\delta} d^{\fb + (p+q)(2+\delta)}(\min\{1,\varepsilon\})^{-(2+\delta)}$ and
\begin{equation}\label{5_31}
\sup_{t\in[0,T]}\sup_{x\in[-\kappa,\kappa]^d}\pr[\big]{ \E\br[\big]{ \abs{ V_d(t,x) - \fV_{\fR(d,\varepsilon),\fR(d,\varepsilon)}^{d,0}(t,x) }^{\fq} } }^{\!\nicefrac{1}{\fq}} \le \varepsilon.
\end{equation}
This and H{\"o}lder's inequality 
prove that 
there exist 
$\fR \colon \N \times \R \to \N$
and
$c = (c_{\fq,\delta})_{(\fq,\delta)\in\R^2} \colon \R^2 \to \R$ such that for all $d\in\N$, $\varepsilon,\delta,\fq \in(0,\infty)$
it holds that
\begin{equation}\label{4_40a}
\cost{\fR(d,\varepsilon)}{\fR(d,\varepsilon)}{d} \le c_{\fq,\delta} d^{\fb + (p+q)(2+\delta)}(\min\{1,\varepsilon\})^{-(2+\delta)}
\end{equation}
and
\begin{equation}\label{4_40}
\begin{split}
& \sup_{t\in[0,T]}\sup_{x\in[-\kappa,\kappa]^d}\E\br[\big]{\abs{ V_d(t,x) - \fV_{\fR(d,\varepsilon),\fR(d,\varepsilon)}^{d,0}(t,x) }^{\fq} } \\
& \le \sup_{t\in[0,T]}\sup_{x\in[-\kappa,\kappa]^d} \pr[\big]{ \E\br[\big]{ \abs{ V_d(t,x) - \fV_{\fR(d,\varepsilon),\fR(d,\varepsilon)}^{d,0}(t,x) }^{\maxfn(\fq)} } }^{\!\nicefrac{\fq}{\maxfn(\fq)}} 
\le \varepsilon^\fq.
\end{split}
\end{equation}
Combining \cref{4_40a}, \cref{4_40}, and the fact that for all $n\in\N_0$, $d,j\in\N$, $\theta\in\Theta$, $t\in[0,T]$, $x\in\R^d$ it holds that $\smallU_d(t,x) = V_d(t,\nicefrac{x}{\sqrt{2}})$ and $\mlp_{n,j}^{d,\theta}(t,x) = \fV_{n,j}^{d,\theta}(t,\nicefrac{x}{\sqrt{2}})$ hence establishes \cref{cor_f_eq}.
The proof of \cref{cor:final} is thus complete.
\end{proof}


\section*{Acknowledgments}

The first author acknowledges funding by the Deutsche Forschungsgemeinschaft (DFG, German Research Foundation) through the research grant HU1889\slash 6-2.
The second author acknowledges funding by the Deutsche Forschungsgemeinschaft (DFG, German Research Foundation) under Germany's Excellence Strategy EXC 2044-390685587, Mathematics M{\"u}nster: Dynamics-Geometry-Structure.
The fourth author acknowledges funding by the National Science Foundation (NSF 1903450).

\bibliographystyle{acm}
\bibliography{bibfile}

\end{document}